\documentclass[11pt,leqno]{amsart}
\usepackage[all,cmtip]{xy}
\usepackage{a4,latexsym}
\usepackage[utf8]{inputenc}
\usepackage{amsfonts}
\usepackage[hidelinks]{hyperref}
\usepackage{amsxtra,amsmath,amsfonts,amscd, amssymb, mathrsfs, amsthm}
\usepackage{color, mathtools}
\usepackage{bbm}
\usepackage{enumitem}
\usepackage{xcolor}
\usepackage{graphicx}
\usepackage{tikz-cd}
\usepackage{braket}
\usepackage[all]{xypic}
\usepackage[toc,page,title,titletoc,header]{appendix}
\usepackage[capitalize]{cleveref}

\numberwithin{paragraph}{section}

\setlist[enumerate]{label=\it{(\roman*)},
	ref=\it{(\roman*)}}
\setlist[enumerate]{label=\it{(\roman*)},
	ref=\it{(\roman*)}}

\newlist{enumerate1}{enumerate}{1}
\setlist[enumerate1]{resume,leftmargin=*,label=\it(\arabic*),ref=\it{(\arabic*)}}

\newlist{enumeratea}{enumerate}{1}
\setlist[enumeratea]{resume,leftmargin=*,label=\it(\alph*),ref=\it{(\alph*)}}

\usepackage{dirtytalk}

\newcommand{\Hom}[2]{\text{Hom}(#1,#2)}

\newcommand{\Gal}[2]{\text{Gal}(#1/#2)}

\newcommand{\C}{\mathbb{C}}
\newcommand{\R}{\mathbb{R}}
\newcommand{\N}{\mathbb{N}}
\newcommand{\Z}{\mathbb{Z}}
\newcommand{\Q}{\mathbb{Q}}

\newcommand{\E}{\mathcal{E}}
\DeclareMathOperator{\Div}{Div}

\newcommand{\Frac}{\mathrm{Frac}}

\newcommand{\metr}{{\|\hspace{1ex}\|}}
\newcommand{\val}{{|\hspace{1ex}|}}

\newcommand{\Spec}{\mathrm{Spec}}

\newcommand{\arnef}{\mathrm{ar\text{-}nef}}
\newcommand{\arint}{\mathrm{ar\text{-}int}}
\newcommand{\relint}{\mathrm{rel\text{-}int}}
\newcommand{\relsnef}{\mathrm{rel\text{-}snef}}
\newcommand{\relnef}{\mathrm{rel\text{-}nef}}
\newcommand{\arsnef}{\mathrm{ar\text{-}snef}}

\newcommand{\nef}{\mathrm{nef}}
\newcommand{\snef}{\mathrm{snef}}
\newcommand{\num}{\mathrm{num}}

\newcommand{\tFS}{\mathrm{tFS}}

\newcommand{\vol}{\mathrm{vol}}
\newcommand{\YZ}{\mathrm{YZ}}
\newcommand{\CM}{\mathrm{CM}}

\newcommand{\gm}{\mathrm{gm}}

\newcommand{\mo}{\mathrm{mo}}
\newcommand{\cpt}{\mathrm{cpt}}
\newcommand{\pur}{\mathrm{pur}}

\newcommand{\OO}{\mathcal{O}}

\renewcommand{\dim}{\operatorname{dim}\nolimits}

\newcommand{\rk}{\mathrm{rk}}

\newcommand{\an}{\mathrm{an}}

\makeatletter
\newcommand\tint{\mathop{\mathpalette\tb@int{t}}\!\int}
\newcommand\bint{\mathop{\mathpalette\tb@int{b}}\!\int}
\newcommand\tb@int[2]{%
	\sbox\z@{$\m@th#1\int$}%
	\if#2t%
	\rlap{\hbox to\wd\z@{%
			\hfil
			\vrule width .45em height \dimexpr\ht\z@+1.4pt\relax depth -\dimexpr\ht\z@+1pt\relax
			\kern.05em 
	}}
	\else
	\rlap{\hbox to\wd\z@{%
			\vrule width .45em height -\dimexpr\dp\z@+1pt\relax depth \dimexpr\dp\z@+1.4pt\relax
			\hfil
	}}
	\fi
}
\makeatother

\newtheorem{theorem}{Theorem}[subsection]

\newtheorem{corollary}[theorem]{Corollary}

\newtheorem{lemma}[theorem]{Lemma}
\newtheorem{lemma*}{Lemma}
\newtheorem{proposition}[theorem]{Proposition}
\newtheorem{prop}[theorem]{Proposition}

\theoremstyle{definition}
\newtheorem{definition}[theorem]{Definition}
\newtheorem{proposition&definition}[theorem]{Proposition\&Definition}
\newtheorem{lemma&definition}[theorem]{Lemma\&Definition}
\newtheorem{theorem&definition}[theorem]{Theorem\&Definition}
\newtheorem{example}[theorem]{Example}

\newtheorem{example*}{Example}
\newtheorem{remark}[theorem]{Remark}

\newtheorem{question*}{Question}

\newtheorem{art}[theorem]{}

\newtheorem{mainthm}{Theorem}

\begin{document}

\title[Concave transforms of compactified $S$-metrized divisors]{Concave transforms of compactified $S$-metrized divisors}

\author[D.~Biswas]{Debam Biswas}
\address{D. Biswas, Mathematik, Universit{\"a}t 
	Regensburg, 93040 Regensburg, Germany}
\email{debambiswas@gmail.com}

\author[Y.~Cai]{Yulin Cai}
\address{Y. Cai, Hangzhou International Innovation Institute, Beihang University, Hangzhou 311115, China}
\email{ylcai5388339@gmail.com}

\begin{abstract}
We associate a concave transform to any compactified $S$-metrized divisor on a quasi-projective variety over an adelic curve. 
Then we show a Hilbert-Samuel type formula for relatively nef compactified $S$-metrized $\YZ$-divisors. 
\end{abstract}


\maketitle
\setcounter{tocdepth}{1}
\tableofcontents


\section{Introduction}
\subsection{Okounkov bodies in algebraic geometry}
The notion of Okounkov bodies is introduced by Okounkov \cite{okounkov1996brunn} to study the volume of any ample line bundle on a projective varietiy which is generalized by Lazarsfeld-Musta\c{t}\u{a} \cite{lazarsfeld2009convex} to arbitrary  \emph{big} line bundles. Let $X$ be a $d$-dimension integral projective variety over a field $K$, $L$ a big line bundle on $X$ and $V_\bullet=\{V_m\}_{m\in\N}$ a \emph{graded linear series} of $L$ i.e. a graded sub-algebra of $\{H^0(X,mL)\}_{m\in\N_{\geq 1}}$. If $V_\bullet$ \emph{contains an ample series} (see \cref{def:ampleser}), we can associate a concave set $\Delta(V_\bullet)\subset\R^d$ to $V_\bullet$. The crucial property is that
\begin{equation} \label{eq:introduction, geometric okounkov}
\lim\limits_{m\to\infty}\frac{\dim_K(V_m)}{m^d/d!}={d!}\cdot{\vol_{\mathbb{R}^d}(\Delta(V_\bullet)).}
\end{equation}
where $\vol_{\R^d}(\Delta)$ denotes the Lebesgue measure of any measurable subset $\Delta\subset \R^d$.
The left-hand side of \eqref{eq:introduction, geometric okounkov} is denoted by $\vol(V_\bullet)$. If $V_m=H^0(X,L^{\otimes m})$, we set $\vol(L)\coloneq\vol(\{H^0(X,mL)\}_{m\in\N})$, called the \emph{volume} of $L$. We have the classical formula
\begin{equation}\label{eq:classical Hilbert-samuel}
\vol(L)= L^d
\end{equation}
which relates volume with the auto-intersection number of $L$.
\subsection{Concave transforms in local case for projective varieties}

Let $K$ be a valued field complete with respect to a norm $\val_v$, $X$ a projective variety over $K$ and $L$ a big line bundle on $X$. We denote by $X^\an$ the Berkovich analytification of $X$, see \cite[\S 3.4, \S 3.5]{berkovich1990spectral}. Let $\metr_\varphi$ be a continuous metric on $L$. Then for any $m\in\N$, $\metr_{m\varphi}\coloneq \metr_\varphi^{\otimes m}$ is a metric on $L^{\otimes m}$ and we have the \emph{sup-norm} $\metr_{m\varphi,\sup}$ on $H^0(X,L^{\otimes m})$:
\begin{equation} \label{introduction: sup-norm}
\metr_{m\varphi,\sup}\coloneq\sup\limits_{x\in X^\an}\{\|s\|_{m\varphi}(x)\}.
\end{equation}
When $K=\C$,  for any $m\in\N$, we set \[\mathcal{B}(m\varphi)\coloneq\{s\in H^0(X,L^{\otimes m})\mid \|s\|_{m\varphi,\sup}\leq 1\} \ \text{ and } \ d_m\coloneq\dim_\C(H^0(X,L^{\otimes m})).\] Nystr\"om \cite{nystrom2014transforming} associated a convex function $c[\varphi]\colon \Delta(L)^\circ\to \R$ to $\varphi$. Given another metric $\psi$ of $L$, he was able to show that
\begin{equation}\label{eq:introduction archimedean okounkov}
\lim\limits_{m\to\infty}\frac{1}{m^{d+1}/(d+1)!}\log\left(\frac{\vol_{\R^{d_m}}(\mathcal{B}^\infty(m\psi))}{\vol_{\R^{d_m}}(\mathcal{B}^\infty(m\varphi))}\right) = 2(d+1)!\int_{\Delta(L)^\circ}(c[\varphi](\lambda)-c[\psi](\lambda))\, d\lambda
\end{equation}

For any arbitrary complete valued field $K$, using the method developed in \cite{boucksom2011okounkov}, by considering a superaddtive function associated to $(L,\varphi)$ in \cite[\S 4.2]{chen2015distribution}, Chen-Maclean constructed a concave transform $G_{(L,\varphi)}\colon\Delta(L)^\circ\to \R$ in \cite[Remark~4.4]{chen2015distribution}. Given another metric $\psi$ of $L$, by \cite[Theorem~4.3]{chen2015distribution}, we have that
\begin{align}\label{eq:introduction nonarchimedean okounkov}
\lim\limits_{m\to\infty}\frac{1}{m^{d+1}/(d+1)!}\log\left({\metr_{m\psi,\det}}/{\metr_{m\varphi,\det}}\right) = (d+1)!\int_{\Delta(L)^\circ}(G_{(L,\varphi)}(\lambda)-G_{(L,\psi)}(\lambda))\, d\lambda,
\end{align}
see also \cite[\S 2.4]{sedillot2023differentiability}. We denote by $\vol(L,\varphi,\psi)$ the value of the left-hand side of \eqref{eq:introduction nonarchimedean okounkov}. If the metrics $\varphi, \psi$ are \emph{semi-positive} in the sense of \cite{zhang1995small} when $v$ is non-trivial, and are uniform limits of \emph{Fubini-Study metrics} when $v$ is trivial, then from the Bedford–Taylor theory \cite{bedford1976dirchlet} in the archimedean case, and from \cite[\S 5.6]{chambert2012formes}  (or \cite[Proposition~2.7]{chambert2006mesures} when $v$ is non-trivial and $K$ has a countable dense subfield) in the non-Archimedean case, for any $0\leq j\leq d$ we can associate a mixed Monge-Amp\'ere measure
\[c_1(L,\varphi)^{j}\wedge c_1(L,\psi)^{d-j}\]
on $X^\an$. The \emph{relative Monge-Ampère energy} is defined as
\[E(\varphi,\psi)\coloneq \frac{1}{d+1}\sum\limits_{j=0}^d\int_{X^\an}\log\left({\metr_\psi/\metr_\varphi}\right)c_1(L,\varphi)^{j}\wedge c_1(L,\psi)^{d-j}.\] 
Then we have the equality
\begin{equation}\label{eq: Hilbert-samuel local}
\vol(L,\varphi,\psi) = (d+1)\cdot E(\varphi,\psi)
\end{equation}
which was proved in \cite[Theorem~A]{berman2010growth} for the Archimedean case (notice that if $X$ is not smooth, we can pass to its resolution of singularity) and in \cite[Theorem~1.1]{boucksom2021non} for the non-Archimedean case (or \cite[Theorem~9.15]{boucksom2021spaces} if $L$ is semi-ample). Notice that the above results in \cite{boucksom2021non} can be applied in the trivially valued case.  The formulas \eqref{eq:introduction archimedean okounkov}, \eqref{eq:introduction nonarchimedean okounkov} are local analogues of \eqref{eq:introduction, geometric okounkov},  and \eqref{eq: Hilbert-samuel local} is a local analogue of \eqref{eq:classical Hilbert-samuel}.

\subsection{Concave transforms in Arakelov theory for projective varieties}
The framework of the construction of concave transforms was given by Boucksom-Chen \cite{boucksom2011okounkov} which is also used in the local case above. Let $X$ be a $d$-dimensional projective variety over a field $K$ and $L$ be a big line bundle on $X$. Let $V_\bullet=\{V_m\}_{m\in\N}$ be a graded sub-algebra of $\{H^0(X,mL)\}_{m\in\N}$ \emph{containing an ample series} (see \cite[Definition~1.1]{boucksom2011okounkov}) and $\mathcal{F}$ be a \emph{multiplicative} filtration (see \cite[Definition~1.3]{boucksom2011okounkov}) of $V_\bullet$. Boucksom-Chen \cite{boucksom2011okounkov} associated a concave transform $G_{\mathcal{F}}\colon \Delta(V_\bullet)^\circ\to\R$ satisfying certain measure convergence property if $\mathcal{F}$ is furthermore \emph{pointwise bounded below} and \emph{linearly bounded above} in the sense of \cite[Definition~1.3]{boucksom2011okounkov}. When $K$ is a number field with the ring of ad\'eles $\mathbb{A}_K$ and $V_\bullet= \{V_m\}_{m\in\N}$ is equipped with a multiplicative norm $\metr_m$ on each $V_m$ such that $\overline{V_m}\coloneq(V_m,\metr_m)$ is an \emph{adelically normed space} in the sense of \cite[2.1]{boucksom2011okounkov}. Set 
\begin{align*} 
	\widehat{V}_m\coloneq \{s\in V_m\mid \|s\|_{m,\omega}\leq 1 \text{ for any $\omega\in\Omega$}\},
\end{align*}
\begin{align}\label{eq:introduction small sections}
	\widehat{\dim}_K(\widehat{V}_m)\coloneq \frac{1}{[K\colon \Q]}\log\#\widehat{V}_m,
\end{align}
\begin{align*}
	\chi(\overline{V_m})\coloneq \frac{1}{[K:\Q]}\log\vol(\widehat{V}_m),
\end{align*}
where $\vol(\cdot)$ is the Haar measure of $V_{m,\mathbb{A}_K}\coloneq V_m\otimes_K\mathbb{A}_K$ normalized by $\vol(V_{m,\mathbb{A}_K}/ V_m)=1$ and  $\widehat{V}_m$ is viewed as a subset of $V_{m,\mathbb{A}_K}$.
Notice that the norms $\metr_m$ give a multiplicative filtration of $V_\bullet$, namely the \emph{filtration by minima} $\mathcal{F}_{\min}$ (see \cite[2.5]{boucksom2011okounkov}). Hence we can associate a concave transform  $G_{\mathcal{F}_{\min}}\colon \Delta(V_\bullet)^\circ\to \R$ to $\overline{V_\bullet}=\{(V_m,\metr_m)\}_{m\in\N}$. If $\mathcal{F}_{\min}$ is linearly bounded above, then the measure convergence property implies that
\begin{equation}\label{eq:introduction volume okounkov}
\lim\limits_{m\to\infty}\frac{\widehat{\dim}_K(\widehat{V}_m)}{m^{d+1}/(d+1)!}=(d+1)!\int_{\Delta(V_\bullet)^\circ}\max\{0,G_{\mathcal{F}_{\min}}(\lambda)\} \, d\lambda.
\end{equation}
If $\mathcal{F}_{\min}$ is furthermore \emph{linearly bounded below} (in particular, $\inf\limits_{\lambda\in \Delta(V_\bullet)^\circ}G_{\mathcal{F}_{\min}}(\lambda)>-\infty$), then 
\begin{equation}\label{eq:introduction chi volume okounkov}
\lim\limits_{m\to\infty}\frac{\chi(\widehat{V}_m)}{m^{d+1}/(d+1)!}=(d+1)!\int_{\Delta(V_\bullet)^\circ}G_{\mathcal{F}_{\min}}(\lambda) \, d\lambda.
\end{equation}
These two formulas \eqref{eq:introduction volume okounkov}, \eqref{eq:introduction chi volume okounkov} are the arithmetic version of \eqref{eq:introduction, geometric okounkov},\eqref{eq:introduction archimedean okounkov}, respectively. In particular, if $L$ is endowed with a \emph{model metric} $\metr_\varphi=(\metr_{\varphi,\omega})_{\omega\in M_K}$ ($M_K$ is the set of places of $K$) i.e. $\metr_\varphi$ induced by a \emph{metrized} \emph{line bundle} on a projective $\OO_K$-model of $X$ (see \cite{zhang1995small}) 
and $V_m=H^0(X,L^{\otimes m})$ with $\metr_m$ given by the sup-norms defined in \eqref{introduction: sup-norm}, then the left-hand sides of \eqref{eq:introduction volume okounkov}, \eqref{eq:introduction chi volume okounkov}, denoted by $\widehat{\vol}(\overline{L})$ and $\widehat{\vol}_\chi(\overline{L})$, are called the \emph{arithmetic volume} and the \emph{arithmetic $\chi$-volume} of $\overline{L}\coloneq (L,\metr_\varphi)$, respectively. Moreover, if $\metr_{\varphi,\omega}$ is induced by a \emph{semi-positive} metrized line bundle on some projective $\OO_K$-model of $X$ with underlying line bundle ample (in this case, the corresponding filtration $\mathcal{F}_{\min}$ satisfies all properties above), we have the arithmetic Hilbert-Samuel formula (see \cite[Th\'eor\`em Principal]{Abbes1995arithmetic} or \cite[Theorem~5.36]{moriwaki2014arakelov}):
\begin{equation}\label{eq:introduction hilber arithmetic}
\widehat{\vol}_\chi(\overline{L})=\overline{L}^{d+1},
\end{equation}
where $\overline{L}^{d+1}$ is the arithmetic intersection number given by \cite{gillet1990arithmetic}.

In \cite{chen2020arakelov}, Chen-Moriwaki proposed a new framework of Arakelov theory. They consider projective varieties over an \emph{adelic curve} which is much more general than the classical cases of number fields or function fields. An \emph{adelic curve} consisted of a field $F$ and absolute values on $F$ which are parameterized by a measure space $(\Omega,\mathcal{A},\nu)$ similar to the notion of $M$-fields considered by Gubler in \cite{MfieldGubler}. Throughout this paper, the adelic curve is assumed to be \emph{proper} which means
that the following product formula holds: for any $a\in F^\times$,
\begin{equation*} 
	\int_\Omega \log|a|_\omega\, \nu(d\omega)=0.
\end{equation*}
More examples of adelic curves are given in \cite[\S 3.2.4-6]{chen2020arakelov}. If $K$ is endowed with a structure of adelic curve $S=(K,\Omega,\mathcal{A},\nu)$ with $\Omega$ is discrete or $K$ countable, let $\metr_\varphi=(\metr_{\varphi,\omega})_{\omega\in\Omega}$ be a family of metrics on $L$ which is \emph{dominated} and \emph{measurable} in the sense of \cite{chen2020arakelov}, then for any $m\in\N$, $H^0(X,L^{\otimes m})$ together with the sup-norms $\metr_{m\varphi,\sup}$ is an \emph{adelic vector bundle} over $S$, so the \emph{Arakelov degree} $\widehat{\deg}(H^0(X,L^{\otimes m}),\metr_{m\varphi,\sup})$ and the \emph{positive Arakelov degree} $\widehat{\deg}_+(H^0(X,L^{\otimes m}),\metr_{m\varphi,\sup})$ are well-defined \cite[Definition~4.3.7,  \S 4.3.4]{chen2020arakelov}. Write $\overline{L}\coloneq (L,\metr_\varphi)$. After generalizing \cite[Theorem~A]{boucksom2011okounkov} to the superadditive filtrations (see \cite[Theorem~6.3.20]{chen2020arakelov}), Chen-Moriwaki \cite[Theorem~6.4.6]{chen2020arakelov} associated a concave transform $G_{\overline{L}}\colon\Delta(L)^\circ\to\R$ satisfying
\begin{equation}\label{eq:introduction volume adelic}
	\lim\limits_{m\to\infty}\frac{\widehat{\deg}_+(H^0(X,L^{\otimes m}), \metr_{m\varphi,\sup})}{m^{d+1}/(d+1)!}=(d+1)!\int_{\Delta(L)^\circ}\max\{0,G_{\overline{L}}(\lambda)\} \, d\lambda.
\end{equation}
If $\inf\limits_{\lambda\in \Delta*(L)}G_{(L,\varphi)}(\lambda)<\infty$, then 
\begin{equation}\label{eq:introduction chi volume adelic}
	\lim\limits_{m\to\infty}\frac{\widehat{\deg}(H^0(X,L^{\otimes m}), \metr_{m\varphi,\sup})}{m^{d+1}/(d+1)!}=(d+1)!\int_{\Delta(L)^\circ}G_{\overline{L}}(\lambda) \, d\lambda.
\end{equation}
As the classical case, the left-hand sides of \eqref{eq:introduction volume adelic}, \eqref{eq:introduction chi volume adelic} are denoted by $\widehat{\vol}(\overline{L})$, $\widehat{\vol}_\chi(\overline{L})$, called the \emph{arithmetic volume}, the \emph{arithmetic $\chi$-volume} of $\overline{L}$, respectively. In their subsequent works \cite{chen2021arithmetic} \cite{chen2022hilbert}, if $\overline{L}$ is \emph{relatively ample} in the sense of \cite[Definition~6.3.2]{chen2022hilbert}, they defined the arithmetic intersection number $(\overline{L}^{d+1}\mid X)_S$ as integrals of local intersection numbers \cite[Definition~4.4.3]{chen2021arithmetic} and prove a Hilbert-Samuel type formula \cite[Theorem~B]{chen2022hilbert}:
\begin{equation}\label{eq:introduction hilber adelic}
	\widehat{\vol}_\chi(\overline{L})=(\overline{L}^{d+1}\mid X)_S
\end{equation}
generalizing \eqref{eq:introduction hilber arithmetic} to the adelic curve case.

\subsection{Arakelov theory for line bundles with singular metrics}
The formulas we presented in the local and global case are for continuous metrics of line bundles over projective varieties. In arithmetic geometry, there are singular metrics appearing naturally, e.g. the Petersson metric of the Hodge bundle on a compactification of the moduli space of principally polarized abelian varieties of dimension $g\geq 1$. In \cite{burgos2007cohomological}, Burgos-Kramer-K\"uhn developed a systematic approach to arithmetic intersection theory for singular metrics with log-log singularities along the boundary. 

Another approach was given recently by Yuan-Zhang \cite{yuan2021adelic}. Let $K$ be a number field or the function field of a curve over a field, $U$ a quasi-projective variety. Let $X$ be a projective compactification of $U$ such that $B=X \setminus U$ is a Cartier divisor. We take a positive \emph{arithmetic model divisor} $\overline{\mathcal{B}}$ for $B$, then $\overline{\mathcal{B}}$ induces a \emph{model metrized divisor} $\overline{B}\coloneq(B, g_B)$, called a \emph{boundary divisor} of $U$. A \emph{compactified metrized line bundle} $\overline{L}$ is given by a line bundle $L_U$ on $U$ and a sequence of line bundles $L_j$ on projective compactifications $X_j$ equipped with model metrics $\metr_j$ satisfying the following conditions: 
\begin{itemize}
	\item $L_j|_U=L_U$ for any $j\in\N_{\geq 1}$;
	\item let $s$ be a rational section of $L_U$, and $s_j$  the extension of $s$ to $L_j$, then for any $\varepsilon\in \Q_{>0}$, there is $j_0\in \N_{\geq 1}$ such that for all $i,j \geq j_0$,
	$$-\varepsilon B \leq {\rm div}(s_i) - {\rm div}(s_j) \leq \varepsilon B \quad \text{and} \quad -\varepsilon g_B \leq \log\|s_j\|_j - \log\|s_i\|_i \leq \varepsilon g_B,$$
	where the inequalities take place after pulling back to a joint projective compactification of $U$.
\end{itemize}
Yuan-Zhang then defined the arithmetic intersection number for \emph{arithmetically nef} (called \emph{nef} in \cite{yuan2021adelic}) compactified line bundles. Using the energy approach, Burgos-Kramer \cite{burgos2024on} could extend the arithmetic intersection number to \emph{relatively nef} compactified line bundles which allowed them to define the arithmetic intersection number of the line bundle of Siegel–Jacobi forms with canonical metrics, and their extension covers the arithmetic intersection number defined in \cite{burgos2007cohomological}.  
In \cite{cai2024abstract}, Gubler and the second author generalized the arithmetic intersection numbers of compactified metrized line bundles in case where $K$ is any field endowed with a structure of adelic curve. They also used the energy approach of Burgos and Kramer to extend the arithmetic intersection number to more singular metrics.

Yuan-Zhang worked on the geometric case in \cite{yuan2021adelic} at the same time, and defined similarly the \emph{compactified line bundles} and corresponding intersection theory on a quasi-projective variety over an arbitrary field. Let $L$ be a compactified line bundle on $U$. We can define the space of global sections $H^0(U,L)$ and the volume $\vol(L)$ of $L$:
\begin{align*}
	\vol(L)\coloneq \limsup\limits_{m\to\infty}\frac{\dim_K(H^0(U,L^{\otimes m}))}{m^d/d!}
\end{align*} 
which is in fact a limit, see \cite[Theorem~5.2.1~(1)]{yuan2021adelic}. Using the continuity of $\vol(\cdot)$ \cite[Theorem~5.2.1~(2)]{yuan2021adelic}, if $L$ is \emph{nef} in the sense of \cite{yuan2021adelic}, then we have a Hilbert-Samuel type formula \cite[Theorem~5.2.2~(1)]{yuan2021adelic}:
\[\vol(L)= L^d,\]
where $L^d$ is the auto-intersection number of $L$ given by \cite[Proposition~4.1.1]{yuan2021adelic}. 

In \cite{nijeribanra}, the first author followed the method in \cite{lazarsfeld2009convex} and defined the Okounkov body $\Delta(L)$ associated to $L$ with the property
\[\vol(L)=d!\cdot \vol_{\R^d}(\Delta(L)).\]

For the arithmetic case, let $\overline{L}$ be a compactified metrized line bundle on $U$ with underlying compactified line bundle $L$. As the projective case, using \eqref{eq:introduction small sections}, we can define the \emph{arithmetic volume} $\widehat{\vol}(\overline{L})$ of $\overline{L}$, see \cite[Definition~5.1.3]{yuan2021adelic}. Using the continuity of $\widehat{\vol}(\cdot)$ in \cite[Theorem~5.2.1~(2)]{yuan2021adelic}, if $L$ is arithmetically nef, then we have the following 
arithmetic Hilbert-Samuel formula \cite[Theorem~5.2.2~(1)]{yuan2021adelic}:
\[\widehat{\vol}(\overline{L})=\overline{L}^{d+1},\]
where $\overline{L}^{d+1}$ is the arithmetic intersection number given by \cite[Proposition~4.1.1]{yuan2021adelic}.

In fact, before Yuan-Zhang's theory, Berman-Freixas \cite[Theorem~1.1]{bermansingular} established an arithmetic Hilbert Samuel formula for semi-positive metrics with log-log singularity. Instead of considering sup-norms, they consider the $L^2$-norms at archimedean places.  
\subsection{Main results}
The goal of this paper is to use the theory of concave transforms along the lines of \cite{boucksom2011okounkov} and \cite[Chapter 6]{chen2020arakelov} to study the arithmetic volumes and arithmetic $\chi$-volumes of compactified metrized divisors on quasi-projective varieties. 

Let $S=(K,\Omega,\mathcal{A},\nu)$ be a proper adelic curve such that $\mathcal{A}$ is discrete or that $K$ is countable, and let $U$ be a $d$-dimensional normal quasi-projective variety over $K$. {We further assume that $S$ satisfies the tensorial minimal slope property for some $C\in\R_{\geq 0}$ in \cref{def:tensorialproperty}. }
Let $\overline{D}=(D,g)$ be a \emph{compactified $S$-metrized divisor} of $U$ in the sense of \cref{def:CM divisors on non-proper}. In \cref{section:arithmetic volume and concave transforms}, we associate a concave transform $G_{\overline{D}}\colon \Delta(D)^\circ\to \R$ to $\overline{D}$ if $D$ is big, where $\Delta(D)$ be the Okounkov body of $D$. The first main result is given as follows:
\begin{mainthm}[\cref{theorem:concavemain}]
	Let $\overline{V_\bullet}=\{\overline{V_m}\}_{m\in\N}$ be the graded $K$-algebra of adelic vector bundles associated to $\overline{D}$ defined in \cref{def:volume}. If $D$ is big, then 
	\begin{align} \label{eq:introduction main 1}
\lim_{m\to\infty}\frac{\widehat{\deg}_+(\overline{V_m})}{m^{d+1}/(d+1)!}={(d+1)!} \int_{\Delta(D)^{\circ}} \max\{G_{\overline{D}}(\lambda),0\} \, d\lambda,
	\end{align}
and
	\begin{align} \label{eq:introduction main 2}
		\limsup\limits_{m\to\infty}\frac{\widehat{\deg}(\overline{V_m})}{m^{d+1}/(d+1)!}\leq(d+1)!\int_{\Delta(D)^{\circ}} G_{\overline{D}}(\lambda) \, d\lambda
	\end{align}
with equality if $\widehat{\mu}_{\min}^{\mathrm{asy}}(\overline{D})>-\infty$, where $d\lambda$ is the standard Lebesgue measure on $\Delta(D)\subseteq \R^d$. 
\end{mainthm}
The left-hand sides of \eqref{eq:introduction main 1}, \eqref{eq:introduction main 2} are denoted by  $\widehat{\vol}(\overline{D})$, $\widehat{\vol}_\chi(\overline{D})$, called the \emph{arithmetic volume}, the \emph{arithmetic $\chi$-volume} of $\overline{D}$, respectively (we will show in \cref{theorem:consistentvolumedef} that the notion arithmetic volume coincide with the one defined in \cite[Definition~5.1.3]{yuan2021adelic} if $S$ is given by a number field). We denote $\widehat{\vol}_\chi^\num(\overline{D})\coloneq(d+1)!\int_{\Delta(D)^{\circ}} G_{\overline{D}}(\lambda) \, d\lambda$ which plays an important role in this paper. We also give applications of the concave transform $G_{\overline{D}}$, for example, we show following fundamental inequality in \cref{thm:essential minimum}:
\begin{align*}
\sup\limits_{\text{$V\subset U$ open}}\inf\limits_{x\in V(\overline{K})}h_{\overline{D}}(x) \geq \frac{\widehat{\vol}_\chi^{\mathrm{num}}(\overline{D})}{(d+1)\cdot\vol(D)},
\end{align*}
and generalize the height inequality \cite[Theorem~5.3.5~(3)]{yuan2021adelic} to the adelic curve case in \cref{theorem:heightinequality}.

In \cref{section: arithmetic hilbert} and \cref{section: equidistribtion}, we assume that $S$ is given by a number field or the function field of a curve. We consider the set $\widehat{\Div}_{S,\Q}(U)_{\relnef}^{\arnef,\YZ}$ of \emph{relatively nef compactified $\YZ$-divisors} $\overline{E}$ (see \cref{definition of compactified YZ-divisors}) such that there is \emph{arithmetically nef compactified $\YZ$-divisor} (see \cref{definition of compactified YZ-divisors}) whose underlying compactified divisor is the same as the one of $\overline{E}$. The arithmetic intersection number $(\overline{E}^{d+1}\mid U)_S$ of an element $\overline{E}$ in $\widehat{\Div}_{S,\Q}(U)_{\relnef}^{\arnef,\YZ}$ is well-defined, see \cite[Theorem~11.3]{cai2024abstract}. We then obtain the following Hilbert-Samuel type formula. 

\begin{mainthm}[\cref{thm:Hilbert-Samuel formula}]
	\label{thm:introduction hilbert}
	If $\overline{D}\in \widehat{\Div}_{S,\Q}({U})_{\relnef}^{\arnef,\YZ}$ and $D$ is big. Then
	\[\widehat{\vol}^{\num}_\chi(\overline{D})=(\overline{D}^{d+1}\mid U)_S.\]
	If moreover, 
	$\widehat{\mu}_{\min}^{\mathrm{asy}}(\overline{D})>-\infty$, then 
	\[\widehat{\vol}_\chi(\overline{D})=(\overline{D}^{d+1}\mid U)_S.\]
\end{mainthm}
An obvious consequence is that $(\overline{D}^{d+1}\mid U)_S>0$ implies that $\overline{D}$ is \emph{big} in the sense of \cref{def:volume}, see \cref{cor:big if intersection number>0}.
Another application of concave transforms is the following equidistribution theorem generalizing \cite[Theorem~5.4.3]{yuan2021adelic} and \cite[Theroem~(C)]{nijerdiff}. 

\begin{mainthm}[\cref{theorem:equidsitributionforfinitenergy}]
	Let $\overline{D}=(D,g)\in \widehat{\Div}_{S,\Q}(U)^{\YZ}_{\cpt}$ with $D$ big. Let $(x_m)_{m\in I}\subset U(\overline{K})$ be a generic net of points which is small with respect to $\overline{D}$. Then for any $c\in\mathscr{L}^1(\Omega,\mathcal{A},\nu)$ such that $\overline{D}(c)$ is big, for any $v\in\Omega$ and for any $f_v\in C_c(U_v^\an)$, we have that
	\begin{align}
		\lim_{m\in I} \frac{1}{\#O(x_m)}\sum_{\sigma\in\Gal{\overline{K}}{K}}f_v(\sigma(x_m))=\frac{\langle \overline{D}(c)^d\rangle\cdot (0,f_v)}{\vol(D)}.
	\end{align}
(here $O(x_m)$ is the Galois orbits of $x_m$ and $\langle\overline{D}(c)^d\rangle\cdot (0,f_v)$ is the positive intersection number defined in \cite{nijerdiff}, see also \cref{positive intersection number}.)
\end{mainthm}
 Notice that we don't require $\overline{D}$ to be arithmetically integrable which is a rather strong condition, see \cite[Example~7.13]{cai2024abstract}. 

\subsection*{Notation}

For any scheme $X$, we denote by $\Div(X)$ the group of Cartier divisors on $X$. An algebraic variety $X$ over a field $k$ is defined as a geometrically integral separated scheme of finite type over $k$, and we denote by $k(X)$ its function field.

Let $S=(K,\Omega,\mathcal{A},\nu)$ be an adelic curve, see \cref{def:adelic curve}. We denote by $\mathscr{L}^1(\Omega,\mathcal{A},\nu)$ the space of integrable functions $\phi:\Omega\to\R$. For any $\omega\in K$, we denote $\val_\omega$ the absolute value corresponding to $\omega$ and $K_\omega$ the completion of $K$ with respect to $\val_\omega$.

For any vector space $V$ over $K$,  we denote $V_\omega\coloneq V\otimes_KK_\omega$. Similarly, for any algebraic variety $X$ over $K$ and any ($\Q$-)Cartier divisor $D$ on $X$, we denote by $X_\omega, D_\omega$ the base change of $X, D$ respectively. 
For a measurable subset $\Delta$ of $\R^d$,  $\Delta^\circ$ denotes the topological interior of $\Delta$ in $\R^d$, and $\vol_{\R^d}(\Delta)$ denotes the Lebesgue measure of $\Delta$.

\subsection*{Acknowledgements}

\section{Adelic curves and adelic vector bundles}
\subsection{Adelic Curves}
We begin by introducing the notion of \emph{adelic curves} in this subsection. 
Our main reference is \cite{chen2020arakelov}.
\begin{definition}[\cite{chen2020arakelov}~\S~3.1] 	\label{def:adelic curve}
	Let $K$ be a field and $M_K$ the set of all absolute values on $K$. An \emph{adelic structure on $K$} is a measure space $(\Omega,\mathcal{A},\nu)$ equipped with a map $\phi: \Omega\rightarrow M_K, \ \ \omega\mapsto |\cdot|_\omega$ such that for any $a\in K^\times$, the function $\log|a|_{\omega}: \Omega\rightarrow \R$ is measurable. The data $(K, (\Omega, \mathcal{A},\nu),\phi)$ (or simply $(K,\Omega, \mathcal{A}, \nu)$) is called an \emph{adelic curve}.
	Furthermore, an adelic curve $(K, (\Omega, \mathcal{A},\nu),\phi)$ is said to be \emph{proper} if the \emph{product formula} holds: for any $a\in K^\times$,
	\begin{equation*} 		\label{product formula}
		\int_\Omega \log|a|_\omega\, \nu(d\omega)=0.
	\end{equation*}
\end{definition}
\begin{remark}\label{notation for adelic curves}
	Given an adelic curve $(K, (\Omega, \mathcal{A},\nu),\phi)$, we fix the following notation throughout the paper:	\[\Omega_\infty:=\{\omega\in \Omega\mid \text{$\phi(\omega)$  archimedean}\},\]
		\[\Omega_\mathrm{fin}:=\{\omega\in \Omega\mid \text{$\phi(\omega)$ non-archimedean}\},\]
		\[\Omega_0:=\{\omega\in \Omega\mid \text{$\phi(\omega)$ trivial}\}\subset \Omega_\mathrm{fin}.\]
\end{remark}
\subsection{Adelic vector bundles}
Throughout the rest of this section, we fix a proper adelic curve $S=(K,\Omega,\mathcal{A},\nu)$.
Recall that given a function $\psi\colon \Omega\to \R$, we define its \emph{upper integral} as
\[\tint_{\Omega} \psi(\omega)\,\nu(d\omega)\coloneq \inf\left\{\int_{\Omega} \phi(\omega)\,\nu(d\omega)\mid \phi\in {\mathscr{L}^1(\Omega,\mathcal{A},\nu)},\ \phi\ge \psi \text{ almost everywhere on } \Omega\right\}\]
and its \emph{lower integral} as
\[\bint_{\Omega}\psi(\omega)\,\nu(d\omega)\coloneq \sup\left\{\int_{\Omega} \phi(\omega)\,\nu(d\omega)\mid \phi\in {\mathscr{L}^1(\Omega,\mathcal{A},\nu)},\ \phi\le \psi \text{ almost everywhere on } \Omega\right\}.\]
We say such a function $\psi$ is \emph{upper dominated} (resp. \emph{lower dominated}) if $\psi$ has a finite upper integral (resp. lower integral).

\begin{art}     \label{def:adelicvectorbundle}
	By an \emph{$S$-normed vector space} over $K$, we mean a finite dimensional vector space $V$ over $K$ equipped with a family of norms $\metr=(\metr_\omega)_{\omega\in\Omega}$ denoted by $\overline{V}\coloneq (V,\metr)=(V,(\metr_\omega)_{\omega\in\Omega})$ where $\metr_\omega$ is a norm on $V_\omega\coloneq V\otimes_K K_\omega$. Recall that given an $S$-normed vector space $\overline{V}$ over $K$, we consider the dual vector space $V^{\vee}$ along with dual norms $\metr_\omega^{\vee}$ of $\metr_\omega$ at each $\omega\in\Omega$, see \cite[\S 1.1.5]{chen2020arakelov}. This gives us an $S$-normed vector space which we call the \emph{dual $S$-normed vector space} of $\overline{V}$ over $K$ and denote it by $\overline{V}^{\vee}\coloneq(V^\vee,\metr^\vee)$.
	
	Let $\overline{V}$ be an $S$-normed vector space over $K$. We say that $\overline{V}$ is an \emph{adelic vector bundle} on $S$ if the following conditions hold:
	\begin{enumerate}
		\item it is \emph{measurable}: for any $s\in V$, the function
		\[\omega{(\in \Omega)}\mapsto \|s\|_\omega\]
		is measurable.
		\item it is \emph{dominated}: for any $s\in V$ and $t\in V^{\vee}$ both the functions 
		\[\omega{(\in \Omega)}\mapsto \log\|s\|_\omega \ \text{ and } \ \omega{(\in \Omega)}\mapsto \log \|t\|_\omega^{\vee}\]
		are upper dominated.
	\end{enumerate}
\end{art}

\begin{art} \label{distance of norms}
	Let $V$ be a vector space over $K$ and $\metr=(\metr_\omega)_{\omega\in\Omega}$, $\metr'=(\metr_\omega')_{\omega\in\Omega}$ be families of norms on $V$. For any $\omega\in\Omega$, set 
	\[d_\omega(\metr,\metr')\coloneq\sup\limits_{s\in V_\omega\setminus\{0\}}\left|\log\|s\|_\omega-\log\|s\|_{\mathbf{e},\omega}\right|\]
	called the \emph{local distance} of $\metr$ and $\metr'$ at $\omega$. The \emph{global distance} of $\metr, \metr'$ is defined as 
	\[d(\metr, \metr')\coloneq \tint_{\Omega} d_\omega(\metr,\metr')\,\nu(d\omega).\]
\end{art}
\begin{art}
	\label{basic invariant of adelic vector bundle}
	Let $\overline{V}=(V,\metr)$ be an adelic vector bundle on $S$. We can construct the following spaces.
	\begin{itemize}
		\item The \emph{determinant} $\det(\overline{V})$ of $\overline{V}$ is given by fixing the determinant $\det(V)$ of $V$ to be the underlying vector space and the norm to be the determinant norm at each $\omega\in\Omega$ (see \cite[Definition~1.1.65]{chen2020arakelov}). It is proven in \cite[Proposition~4.1.32]{chen2020arakelov} that $\det(\overline{V})$ is an adelic vector bundle of dimension $1$ over $K$.
		\item For a vector subspace $W$ of $V$, we can restrict the norms of $\overline{V}$ at each $\omega\in\Omega$ to obtain a normed family with underlying vector space $W$ which is called the \emph{restricted norm}. The subspace $W$ with the restricted norm is an adelic vector bundle by \cite[Proposition~4.1.32~(2)]{chen2020arakelov}. 
		\item For a quotient vector space of $G$ of $V$, we can consider $G$ along with the quotient norms at every $\omega\in\Omega$, see \cite[\S 1.1.3]{chen2020arakelov}, to construct a family of norms, called the \emph{quotient norm}, whose underlying space is $G$. The quotient space with the quotient norm is an adelic vector bundle by \cite[Proposition~4.1.32~(2)]{chen2020arakelov}.
	\end{itemize} 
	Using the notation above, we are able to  define the following invariants of $\overline{V}$.
	\begin{enumerate1}
		\item The \emph{Arakelov degree} of $\overline{V}$ is defined as \[\widehat{\deg}(\overline{V})=\widehat{\deg}(\det(\overline{V}))\coloneq \int_{\Omega}-\log \|s\|_{\det,\omega}\, \nu(d\omega),\]
		where $s$ is any non-zero section of $\det(V)$ and it is well-defined since $S$ is proper and $\dim_K\det(V)=1$.
		\item The \emph{positive Arakelov degree} of $\overline{V}$ is defined as
		\[\widehat{\text{deg}}_+(\overline{V})\coloneq \sup_{\overline{W}}\{\widehat{\text{deg}}(\overline{W})\},\]
		where $\overline{W}$ varies over all subspaces of $\overline{V}$ equipped with restriction norms.
		\item The \emph{slope} of $\overline{V}$ is defined as
		\[\widehat{\mu}(\overline{V})\coloneq \frac{\widehat{\deg}(\overline{V})}{\dim_K(V)}.\]
		\item \label{max and minimal slope} The \emph{maximal slope} and \emph{minimal slope} of $\overline{V}$ are defined as 
		\[\widehat{\mu}_{\max}(\overline{V})\coloneq \sup_{\overline{W}}\{\widehat{\mu}(\overline{W})\}\ \text{and}\ \widehat{\mu}_{\min}(\overline{V})\coloneq \inf_{\overline{G}}\{\widehat{\mu}(\overline{G})\},\]
		respectively,
		where $\overline{W}$ varies over all subspaces of $\overline{V}$ with the restricted norms and $\overline{G}$ varies over all quotient spaces of $\overline{V}$ with quotient norms.
	\end{enumerate1}
	Let $\overline{M}=(M,\metr_M)$, $\overline{N}=(N,\metr_N)$ be two adelic vector bundles, and $f\colon M\to N$ a $K$-linear map. For any $\omega\in\Omega$, there is an induced map of normed vector spaces \[f_\omega\colon (M_\omega,\metr_{{M},\omega})\to (N_\omega,\metr_{{N},\omega})\] and we can consider its \emph{operator norm} $\|f\|_\omega\coloneq \|f_\omega\|$. Then we define the \emph{height} of $f$ as
	\[h(f)\coloneq \bint_{\Omega} \log \|f\|_\omega\,\nu(d\omega).\]
\end{art}
\subsection{Purification}
For our later application in the case of the classical global fields, we consider pure norms in order to relate our arithmetic volumes with the classical ones.
\begin{art} \label{purification} 
	Let $\overline{V}=(V,\metr)$ be an $S$-normed vector space. For any $\omega\in\Omega\setminus\Omega_0$, we denote 
	\[\widehat{V}_\omega\coloneq \{s\in V_\omega\mid \|s\|_\omega\leq 1\}.\]
	The \emph{purification} of $\overline{V}$, denoted by $\overline{V}_\pur$, is the same underlying vector space $V$ equipped with a family of norms $\metr_{\mathrm{pur}}\coloneq \{\metr_{\pur,\omega}\}_{\omega\in\Omega}$ defined as follows: 
	\begin{itemize}
		\item for any $\omega\in\Omega\setminus(\Omega_0\cup \Omega_{\infty})$ and $0
\neq s\in V_\omega$, 
		\[\|s\|_{\pur,\omega}\coloneq \inf\{|\alpha|\mid \alpha\in K_\omega^{\times},\ s\in \alpha\widehat{V}_\omega\};\]
            \item for any $\omega\in\Omega_{\infty}$ and $0\neq s\in V_{\omega}$, we set 
            \[\|s\|_{\pur,\omega}=\|s\|_{\omega};\]
		\item for any $\omega\in\Omega_0$ and $0\not=s\in V_\omega$, $\|s\|_\omega=1$.
	\end{itemize}
	Notice that $\metr_\omega\leq\metr_{\pur,\omega}$ for any $\omega\in\Omega\setminus\Omega_{\infty}$. Moreover it is not always true that $\metr_{\omega}=\metr_{\pur,\omega}$ in the case when $\omega\in\Omega_{\mathrm{fin}}$ which leads to the following definition of purity. 
	\begin{definition}
             \label{def:purity}
             For $\omega\in\Omega$, we say that $\overline{V}$ is \emph{pure at $\omega$} if $\metr_\omega=\metr_{\pur,\omega}$. We say that $\overline{V}$ is \emph{pure} if it is pure at each $\omega\in\Omega$.
	\end{definition}  
\end{art}
\begin{remark} \label{rmk: pure at non-archimedean}
	Let $V=(V,\metr)$ be an $S$-normed vector space such that $\metr_\omega$ is ultrametric for any $\omega\in\Omega_{\mathrm{fin}}$. By \cite[Proposition~1.1.30~(3), Proposition~1.1.32~(2)]{chen2020arakelov}, $\overline{V}$ is pure at $\omega\in\Omega$ if and only if $(V_\omega,\metr_\omega)$ is pure in the sense of \cite[Definition~1.1.29]{chen2020arakelov}. In particular, if $\omega\in\Omega_{\mathrm{fin}}$ gives a discrete absolute value $\val_\omega$ on $K$, then the following statements are equivalent: 
	\begin{itemize}
		\item $\overline{V}$ is pure at $\omega$;
		\item $\|V\|_\omega=|K|_\omega$, see also \cite[Proposition~2.9]{burgos2016arithmetic};
		\item $\metr_\omega$ is induced by some \emph{lattice} $\mathcal{V}_\omega$ of $V_\omega$, see \cite[Definition~1.1.29, Proposition~1.1.30~(2)]{chen2020arakelov}.
	\end{itemize}
	Here, by a lattice of $V_\omega$, we mean a sub-$K_\omega^\circ$-module $\mathcal{V}_\omega$ of $V_\omega$ generating $V_\omega$ as a vector space over $K_\omega$ with $\mathcal{V}_\omega$ bounded in $V_\omega$ for some norm of $V_\omega$, see \cite[Definition~1.1.23]{chen2020arakelov}.
\end{remark}
For further application, we give the following lemma. 

\begin{lemma}
	\label{lemma:puredomination}
Let $\overline{V}$ be an $S$-normed vector bundle. If $\overline{V}$ is dominated, then for any $s'\in V^{\vee}$, the function $\omega (\in \Omega)\mapsto \log\|s'\|^\vee_{\pur,\omega}$ is upper dominated where $(V^\vee, \metr_\pur^\vee)$ is the dual of $\overline{V}_\pur$.
\end{lemma}
\begin{proof}
	Since $\metr_\omega\leq \metr_{\pur,\omega}$, the identity on $V$ induces a morphism of normed vector spaces 
	\[f_\omega\colon(V_\omega,\metr_{\pur,\omega})\rightarrow (V_\omega,\metr_\omega)\]
	for any $\omega\in\Omega$ and clearly the operator norm $\|f\|_\omega\le 1$. Then passing to the dual and using \cite[Proposition~1.1.22]{chen2020arakelov},  for any $s'\in V^{\vee}$ and $\omega\in\Omega$
	\[\|s'\|^{\vee}_{\pur,\omega}\le \|s'\|^{\vee}_\omega\]
	Since $\overline{V}$ is dominated, we have that $\log\|s'\|_\omega^\vee$ is upper dominated, then so is $\log\|s'\|^{\vee}_{\pur,\omega}$ by the above inequality using \cite[Proposition~A.4.2~(2)]{chen2020arakelov}.
\end{proof}
\subsection{Harder-Narasimhan Filtration}

\begin{art}\label{def:harder filtration}
	Let $\overline{V}$ be an adelic vector bundle on $S$. The \emph{Harder-Narasimhan filtration} of $V$ is an $\R$-indexed non-increasing filtration $\mathcal{H}$ on $V$ defined as 
	\[\mathcal{H}^t(V)\coloneq \mathrm{Vect}_K(\{W\subseteq V\mid \widehat{\mu}_{\min}(\overline{W})\ge t\}) \]
	for $t\in\R$, where $\mathrm{Vect}_K(\cdot)$ denotes the vector space generated over $K$ by the subset, see also \cite[Proposition~4.3.46]{chen2020arakelov}.
	The \emph{jumping numbers} of the Harder-Narasimhan filtration is defined as usual by
	\[\widehat{\mu}_i(\overline{V})\coloneq \sup\{t\in \R\mid \text{dim}_K(\mathcal{H}^t(V))\ge i\}\]
	for every integer $i\in \N$. By \cite[Proposition~4.3.46]{chen2020arakelov}, the last slope of $\overline{V}$ is exactly the minimal slope $\widehat{\mu}_{\min}(\overline{V})$ defined in \cref{basic invariant of adelic vector bundle}~\ref{max and minimal slope}. See \cite[Proposition~4.3.50, Proposition~4.3.51, Corollary~4.3.52]{chen2020arakelov} for the relation between the jumping numbers and the Arakelov degree of $\overline{V}$.
\end{art}
\subsection{Graded algebra of adelic vector bundles}
In this subsection, we introduce graded algebras where each of the graded pieces are adelic vector bundles. Later on we will introduce arithmetic volumes as the rate of asymptotic growth of positive degrees of such algebras. In order to have a satisfying theory of filtration, we need to introduce a technical condition on our base adelic curve. Note that given two adelic vector bundles $\overline{E}=(E,\metr_E)$ and $\overline{F}=(F,\metr_F)$ over an adelic curve $S$, we can define their \emph{$(\varepsilon,\pi)$-tensor product} by setting the norm at every $\omega\in\Omega$ to be the $(\varepsilon,\pi)$-tensor product $\metr_{{E},\omega}\otimes_{\varepsilon,\pi} \metr_{{F},\omega}$ (see \cite[Chapter 4]{chen2020arakelov} for more details). We denote this adelic vector bundle by $\overline{E}\otimes_{\varepsilon,\pi} \overline{F}$. 
\begin{art}[\cite{chen2020arakelov}~Definition~4.3.73, Definition~6.3.23]
	\label{def:tensorialproperty}
	Let $C\in \R_{\geq 0}$. We say that $S$ has 
	\begin{itemize}
		\item the \emph{tensorial minimal slope property of level $\ge C_0$} if for any two adelic vector bundles $\overline{E}$ and $\overline{F}$ on $S$, we have
		\[\widehat{\mu}_{\min}(\overline{E}\otimes_{\varepsilon,\pi} \overline{F})\ge \widehat{\mu}_{\min}(\overline{E})+\widehat{\mu}_{\min}(\overline{F})-C_0\log(\dim_K(E)\dim_K(F));\]
		\item the \emph{strong tensorial minimal slope property of level $\ge C_0$} if for any $n\in\N_{\geq 2}$ and any family $\{\overline{E_i}\}_{1\leq i\leq n}$ of adelic vector bundles on $S$, we have
		\[\widehat{\mu}_{\min}(\overline{E_1}\otimes_{\varepsilon,\pi}\cdots\otimes_{\varepsilon,\pi} \overline{E_n})\geq\sum\limits_{i=1}^n\widehat{\mu}_{\min}(\overline{E_i})-C_0\log(\dim_K(E_i));\]
		\item the \emph{Minkowski property of level $\ge C_0$} if for any adelic vector bundles $\overline{E}$ on $S$, we have
		\[{\nu}_{1}(\overline{E})\geq\widehat{\mu}_{\max}(\overline{E})-C_0\log(\dim_K(E)),\]
		where \[\nu_1(\overline{E})\coloneq \sup\left\{t\in\R\mid \exists 0\not=s\in E \text{ such that } \int_{\Omega}-\log\|s\|_\omega\, \nu(d\omega) \geq t\right\}\]
		is the \emph{first} (\emph{logarithmic}) \emph{minimal} of $\overline{E}$,
		see \cite[Definition~4.3.69]{chen2020arakelov} for the more general definition.
	\end{itemize}
		Note that Chen-Moriwaki showed in \cite[Corollary~5.6.2]{chen2020arakelov} that $S$ has the tensorial minimal slope property of level $\ge \frac{3}{2}\nu(\Omega_{\infty})$ when the underlying field $K$ is perfect. In particular number fields and function fields of characteristic $0$ are included in this formalism. When $S$ is given by a regular projective curve over some field $k$, then $S$ has the Minkowski property of level $\geq C_0$ for some $C_0\in\R_{\geq 0}$, see \cite[Remark~4.3.74]{chen2020arakelov}, which implies that $S$ has the tensorial minimal slope property of level $\geq C_0$ by \cite[Corollary~4.3.76]{chen2020arakelov}.
\end{art}

\begin{definition}[\cite{chen2020arakelov}~Definition~6.3.24]
	\label{def:gradedalgerba}
	Let $R\coloneq 
	\Frac(\overline{K}[[z_1,\ldots z_d]])$ be the field of formal power series. We call a \emph{graded $K$-algebra of $S$-normed vector spaces} (\emph{with respect to $R$}) to be any family $\overline{E_{\bullet}}\coloneq \{\overline{E_m}\}_{m\in\N}=\{(E_m,\metr_m)\}_{m\in\N}$ of $S$-normed vector spaces satisfying
	\begin{enumerate}
		\item \label{gradedalgebra i} 
		$\bigoplus\limits_{m=0}^\infty E_{\overline{K},m}T^m$ is contained in a $\overline{K}$-subalgebra of finite type inside $R[T]$, where $E_{\overline{K},m}\coloneq E_m\otimes_K\overline{K}$;
		\item \label{gradedalgebra ii} for any $m\in \N$, $\metr_{m,\omega}$ is ultrametric for any $\omega\in \Omega_{\mathrm{fin}}$;
		\item \label{gradedalgebra iii} for any $m_1,m_2\in \N_{\geq 1}$, $s_1\in E_m$, $s_2\in E_n$ and $\omega\in \Omega$, we have
		\[\|s_1\cdot s_2\|_{m_1+m_2,\omega}\le \|s_1\|_{m_1,\omega}\cdot \|s_2\|_{m_2,\omega}\]
	\end{enumerate}
	A \emph{graded $K$-algebra of adelic vector bundles} is a graded $K$-algebra of $S$-normed vector spaces $\overline{E_\bullet}$ such that each $\overline{E_m}$ is an adelic bundle bundle.
\end{definition}
\begin{remark}
Our definition here is slightly more general than \cite[Definition~6.3.24]{chen2020arakelov}, with this definition, we are able to remove the assumption that $X$ has a rational regular point $p\in X(K)$ in \cite[Theorem~6.4.6]{chen2020arakelov} as pointed out in \cite[Remark~6.3.29]{chen2020arakelov}. In this paper, we will prove a similar result for quasi-projective varieties, but without this assumption, see \cref{theorem:concavemain}.
\end{remark}
\begin{remark}\label{remark:kodaira dimension}
	Recall from \cite[Proposition~6.3.18]{chen2020arakelov} that for any graded $K$-algebra $E_{\bullet}=\{E_m\}_{m\in\N}$ such that $\bigoplus\limits_{m=0}^\infty E_{\overline{K},m}T^m$ is  $\Frac(\overline{K}[[z_1,\ldots z_d]])$ where $E_{\overline{K},m}\coloneq E_m\otimes_K\overline{K}$, there is a minimal non-negative integer $0\leq \kappa\leq d$ such that $\dim_K(E_m)=\dim_K(E_{m,\overline{K}})=O(m^\kappa)$,
	i.e. there is $C\in\R_{>0}$ such that $|\dim_K(E_m)/m^\kappa|<C$ for any $m\in\N$,  see \cite[Definition~6.3.9, Proposition~6.3.18]{chen2020arakelov}, 
	and we call this integer the \emph{Kodaira dimension} of $E_{\bullet}$. 
\end{remark}



We next introduce the \emph{arithmetic volume} of a graded $K$-algebra of adelic vector bundles. 
\begin{definition}
	\label{def:arithvol}
	Let $\overline{E_{\bullet}}$ be a graded $K$-algebra of adelic vector bundles with respect to $R\coloneq\Frac(\overline{K}[[z_1,\ldots z_d]])$. We define the \emph{arithmetic volume} and \emph{arithmetic $\chi$-volume} of $\overline{E_\bullet}$ as
	\[\widehat{\vol}(\overline{E_{\bullet}})\coloneq\limsup_{m\to\infty} \frac{\widehat{\deg}_+(\overline{E_m})}{m^{d+1}/(d+1)!}\]
	and
	\[\widehat{\vol}_\chi(\overline{E_{\bullet}})\coloneq\limsup_{m\to\infty} \frac{\widehat{\deg}(\overline{E_m})}{m^{d+1}/(d+1)!},\]
	respectively. 
\end{definition}

\begin{definition}
	\label{def:asymptoticslopes}
	Let $\overline{E_{\bullet}}$ be a graded $K$-algebra of adelic vector bundles. 
	We define the \emph{asymptotic maximal slope} and \emph{asymptotic minimal slope} of $\overline{E_{\bullet}}$ as
	\[\widehat{\mu}_{\max}^{\mathrm{asy}}(\overline{E_{\bullet}})\coloneq \limsup_{m\to\infty}\frac{\widehat{\mu}_{\max}(\overline{E_m})}{m}\]
	and
	\[\widehat{\mu}_{\min}^{\mathrm{asy}}(\overline{E_{\bullet}})\coloneq {\liminf\limits_{m\to\infty}}\frac{\widehat{\mu}_{\min}(\overline{E_m})}{m},\]
	respectively.
\end{definition}
\begin{remark}
	Note that these quantities are not necessarily bounded. However we will show that for compactified divisors the asymptotic maximal slope is always finite and it turns out that under suitable positivity assumptions on the geometric divisor, the asymptotic minimal slope is also bounded.
\end{remark}

\subsection{Over global fields}
\label{over global field}
We close this section by considering the adelic vector bundles over a global field.
\begin{art} \label{adelic curve defined by number field and projective curves}
	Let $C$ be the spectrum of the ring of integers of a number field $K$ or a smooth projective curve over some field $k$. Let $K$ be the function field of $C$. We can define an adelic curve, denoted by $S$, as follows. Set $\Omega_{\mathrm{fin}}$ the set of closed points on $C$ and
	\begin{align*}
		\Omega_\infty\coloneq \begin{cases}
			\Hom{K}{\C}, & \text{ if $K$ is a number field;}\\
			\emptyset& \text{ if $K$ is the function field of a smooth curve over $k$.}
		\end{cases}
	\end{align*}  We equip $\Omega\coloneq \Omega_\infty\coprod \Omega_{\mathrm{fin}}$ with the discrete $\sigma$-algebra $\mathcal{A}$. Then every $\omega\in\Omega$ determines an absolute value on $K$ in an obvious way. Let $\nu$ be the measure on $(\Omega,\mathcal{A})$ such that for any $\omega\in\Omega$,
	\[\nu(\omega)=\nu(\{\omega\}) \coloneq\begin{cases}
		\frac{[K_\omega\colon\Q_\omega]}{[K\colon\Q]}& \text{ if $K$ is a number field;}\\
		[k(\omega)\colon k]& \text{ if $K$ is the function field of a smooth curve over $k$.}\end{cases}\] 
	Then $S$ is proper. 
\end{art}

In the following, we fix $C$, the spectrum of the ring of integers of a number field $K$ or a smooth projective curve over some field $k$ with function field $K$, and let $S=(K,\Omega,\mathcal{A},\nu)$ be the adelic curve associated to $C$ given in \cref{adelic curve defined by number field and projective curves}.

Classically, the \emph{coherent} $S$-normed vector spaces over $K$ as defined in \cref{def:adelicvectorspaces} are frequently studied. For a graded $K$-algebra of coherent adelic vector bundles $\overline{E_\bullet}$, the arithmetic volume $\widehat{\vol}(\overline{E_\bullet})$ can be defined in two ways, either using positive degree as we did in \cref{def:arithvol} or using small sections in \cref{def:classical arithvol}. We will show these two definitions of arithmetic volume coincide if $\overline{E_\bullet}$ is \emph{asymptotically pure}. 

Recall from \cite[\S 1.2.3]{chen2020arakelov} that for an $S$-normed space $\overline{V}=(V,\metr)$, a basis $\mathbf{e}=\{e_1,\cdots, e_r\}$ of $V$ is a \emph{orthonormal basis} of $V_\omega$ if
\begin{itemize}
	\item $\|e_i\|_\omega=1$ for any $i=1,\cdots, r$;
	\item $\|\lambda_1e_1+\cdots+\lambda_re_r\|_\omega\geq \max\limits_{1\leq i\leq r}\{|\lambda_i|_\omega\}$ for any $(\lambda_1,\cdots, \lambda_r)\in K^r$.
\end{itemize}
In particular, if $\omega\in\Omega_{\mathrm{fin}}$ and $\metr_\omega$ is ultrametric, 
then the second condition is equivalent to that $\|\lambda_1e_1+\cdots+\lambda_re_r\|_\omega= \max\limits_{1\leq i\leq r}\{|\lambda_i|_\omega\}$ for any $(\lambda_1,\cdots, \lambda_r)\in K^r$.
\begin{example}[\cite{chen2020arakelov}~Example~4.1.5]
	\label{example:basischoice}
	Let $V$ be a vector space over $K$ and $\mathbf{e}=\{e_1, \cdots, e_r\}$ a basis of $V$. We can associate $\mathbf{e}$ an $S$-norm $\metr_{\mathbf{e}}=\{\metr_{\mathbf{e},\omega}\}_{\omega\in\Omega}$ as follows: for any $s=\sum\limits_{i=1}^r \lambda_ie_i\in V_{\omega}=V\otimes_{K}K_\omega$, we set
	\begin{align*}
		\|s\|_{\mathbf{e},\omega}\coloneq\begin{cases}
			\sum\limits_{i=1}^r|\lambda_i|_\omega& \text{ if $\omega\in\Omega_\infty$;}\\
			\max\limits_{1\leq i\leq r}\{|\lambda_i|_\omega\}& \text{ if $\omega\in\Omega_{\mathrm{fin}}$.}
		\end{cases}
	\end{align*}
	Then by \cite[Example~4.1.5]{chen2020arakelov}, the basis $\mathbf{e}$ as above gives a structure of adelic vector bundle which we denote by $\overline{V}_{\mathbf{e}}$. Moreover, $\overline{V}_{\mathbf{e}}$ is pure by \cite[Proposition~2.9]{burgos2016arithmetic}.
\end{example}

\begin{definition}
	\label{def:adelicvectorspaces}
	Let $\overline{V}=(V,\metr)$ be an $S$-normed vector space such that $\metr_\omega$ is ultrametric for any $\omega\in\Omega_{\mathrm{fin}}$. We say that $\overline{V}$
	\begin{enumerate1}
		\item (\cite[Definition~2.10]{burgos2016arithmetic}) is \emph{generically trivial} if there is a basis $\mathbf{e}\coloneq \{e_1,\cdots, e_r\}$ of $V$ such that $\mathbf{e}$ forms an orthonormal basis of $V_\omega$ over $K_\omega$ for all {but finitely many} $\omega\in\Omega$;
		\item (\cite[Definition~4.4.4]{burgos2016arithmetic}) is \emph{coherent} if for any $s\in V$, we have that $\|s\|_\omega\leq 1$ for all but finitely many $\omega\in\Omega$.
	\end{enumerate1}
\end{definition}
The generically trivial $S$-normed vector spaces and adelic vector bundles are closely related.

\begin{lemma}
	\label{lemma:adelicvectorspaceisvectorbundle}
	Let $\overline{V}=(V,\metr)$ be an $S$-normed vector space such that $\metr_\omega$ is ultrametric for any $\omega\in\Omega_{\mathrm{fin}}$.
	\begin{enumerate1}
		\item \label{generically trivial is coherent adelic}We have that $\overline{V}$ is generically trivial if and only if there is a basis $\mathbf{e}$ of $V$ such that $\metr_\omega=\metr_{\mathbf{e},\omega}$ for all but finitely many $\omega\in\Omega$. In particular, if $\overline{V}$ is a generically trivial, then it is a coherent adelic vector bundle.
		\item \label{equivalence of adelic and generically trivial} If $\overline{V}$ is pure at all but finitely many $\omega\in\Omega$, then $\overline{V}$ is an adelic vector bundle if and only if $\overline{V}$ is generically trivial. In particular, in this case, $\overline{V}$ is coherent by \ref{generically trivial is coherent adelic}.
		\item \label{purification is adelic} If $\overline{V}$ is a coherent adelic vector bundle, then the purification $\overline{V}_{\pur}$ is a pure generically trivial adelic vector bundle.
	\end{enumerate1}
\end{lemma}
\begin{proof}
	\ref{generically trivial is coherent adelic} The first statement is from the definition. If $\overline{V}$ is a generically trivial, by \cite[Corollary~4.1.10]{chen2020arakelov}, then $\overline{V}$ is an adelic vector bundle. On the other hand, notice that $\overline{V}_{\mathbf{e}}$ is coherent for any basis $\mathbf{e}$, so $\overline{V}$ is coherent by the first statement.
	
	\ref{equivalence of adelic and generically trivial} If $\overline{V}$ is generically trivial, then it is an adelic vector bundle by \ref{generically trivial is coherent adelic}. For the converse statement, assume that $\overline{V}$ is an adelic vector bundle. Let $\mathbf{e}$ be a basis of $V$, and $\overline{V}_{\mathbf{e}}=(V,\metr_{\mathbf{e}})$ the associated adelic vector bundle. Since $\overline{V}$ and $\overline{V}_{\mathbf{e}}$ are adelic vector bundles, by \cite[Corollary~4.1.8]{chen2020arakelov} we have that
	\[d(\metr,\metr_{\mathbf{e}})= \sum\limits_{\omega\in\Omega}\sup\limits_{s\in V_\omega\setminus\{0\}}\left|\log\|s\|_\omega-\log\|s\|_{\mathbf{e},\omega}\right|<\infty\] 
	Now observing that all pure norms have discrete value groups in the non-archimedean places by \cite[Proposition 2.9]{burgos2016arithmetic}, we deduce that $\metr_\omega$ and $\metr_{\mathbf{e},\omega}$ must agree at all but finitely many $\omega\in\Omega$. This implies that $\overline{V}$ is generically trivial by \ref{generically trivial is coherent adelic}.
	
	\ref{purification is adelic} Let $s\in V$. By the definition of the purification, for $\omega\in\Omega$, if $\|s\|_\omega\le 1$, then $\|s\|_{\pur,\omega}\le 1$. Since $\overline{V}$ is coherent, we have that $\|s\|_\omega\le 1$ for almost all $\omega\in\Omega$. So $\|s\|_{\pur,\omega}\le 1$ for almost all $\omega\in\Omega$. Hence
	\[\tint_\Omega \log \|s\|_{\pur,\omega}\, \nu(d\omega)<\infty.\] Now noting that the $\sigma$-algebra $\mathcal{A}$ is discrete, by \cref{lemma:puredomination} we can easily conclude that $\overline{V}_{\pur}$ is an adelic vector bundle. Moreover, $\overline{V}_{\pur}$ is generically trivial by \ref{equivalence of adelic and generically trivial}.
\end{proof}
\begin{definition}
	\label{def:inpurity}
	Let $\overline{V}=(V,\metr)$ be an $S$-norm vector space such that $\metr_\omega$ is ultrametric for any $\omega\in\Omega_{\mathrm{fin}}$. For any $\omega\in\Omega$, notice that $\metr_\omega\leq\metr_{\pur,\omega}$. We set
	\[\sigma_\omega(\overline{V})\coloneq d_\omega(\metr,\metr_\pur)=\sup\limits_{s\in V_\omega\setminus\{0\}}(\log\|s\|_\omega-\log\|s\|_{\pur,\omega})\] 
	The \emph{impurity} of $\overline{V}$ is defined as 
	\[{\sigma}(\overline{V})\coloneq d(\metr,\metr_\pur)=\sum\limits_{\omega\in\Omega}\sigma_\omega(\overline{V})\in [0,+\infty].\]
	From the definition, $\overline{V}$ is pure if and only $\sigma(\overline{V})=0$.
	By \cref{lemma:adelicvectorspaceisvectorbundle}~\ref{purification is adelic} and \cite[Corollary~4.1.8]{chen2020arakelov}, if $\overline{V}$ is a coherent adelic vector bundle, then ${\sigma}(\overline{V})<+\infty$.
\end{definition}
\begin{remark}
	Keep the notation in \cref{def:inpurity}. If $K$ is a number field and $\overline{V}$ is coherent, then the impurity $\sigma(\overline{V})$ is exactly the one defined in \cite[\S 4.4.3]{chen2020arakelov}. We explain the reason as follows. Set $$\mathcal{V}\coloneq\{s\in V\mid \|s\|_\omega\leq 1 \text{ for all $\omega\in\Omega_{\mathrm{fin}}$}\}.$$ Since $\overline{V}$ is coherent, by \cite[Proposition~4.4.2~(2), Remark~4.4.1~(2)]{chen2020arakelov}, for any $\omega\in\Omega_{\mathrm{fin}}$, we have that $$\mathcal{V}\otimes_{\OO_K}K_\omega^\circ=\{x\in V_\omega\mid \|s\|_\omega\leq 1\}.$$
	By definition of purification, the norm $\metr_{\pur,\omega}$ is exactly the one arising from $\mathcal{V}\otimes_{\OO_K}K_\omega^\circ$ in \cite[Definition~1.1.23]{chen2020arakelov}. This implies our claim by definitions of impurity.
\end{remark}

As \cite[Definition~2.14, Definition~2.18]{burgos2016arithmetic}, we have the following definition.

\begin{definition}
	Let $\overline{V}=(V,\metr)$ be a coherent $S$-normed vector space. We define the space of its \emph{small sections} as 
	\[\widehat{V}\coloneq \{s\in V\mid \|s\|_\omega\le 1\ \text{for all}\ \omega\in\Omega\}\]
	and define 
	\[\widehat{h}^0(\overline{V})\coloneq \begin{cases}
		\log (\#\widehat{V})& \text{if $K$ is a number field}\\
		\dim_k(\widehat{V})& \text{if $K$ is the function field of a curve of a field $k$}
	\end{cases} \]
	
	We further assume that $\overline{V}$ is generically trivial. 
	\begin{itemize}
		\item Assume that $K$ is a function field of curve over a field $k$. Let $\mathbf{b}$ be a basis of $V$ over $K$. For any $\omega\in\Omega$, we choose a basis $\mathbf{b}_\omega$ of $\widehat{V}_\omega\coloneq \{s\in V_\omega\mid \|s\|_\omega\leq 1\}$ over $K_\omega^\circ$. The \emph{Euler characteristic} of $\overline{V}$ is defined as
		\begin{align} \label{chi for function field}
			\chi(\overline{V})\coloneq \sum\limits_{\omega\in\Omega}\nu(\omega)\log|\det(\mathbf{b}_\omega/\mathbf{b})|_\omega,
		\end{align}
		where $\mathbf{b}_\omega/\mathbf{b}$ denotes the matrix of $\mathbf{b}_{\omega}$ with respect to the basis $\mathbf{b}$. This quantity does not depend on the choice of bases.
		\item Assume that $K$ is a number field. Let $\mathbb{A}_K$ be its ring of adeles and $V_{\mathbb{A}_K}\coloneq V\otimes_K \mathbb{A}_K$. The \emph{Euler characteristic} of $\overline{V}$ is defined as 
		\begin{align} \label{chi for number field}
			\chi(\overline{V})\coloneq \frac{1}{[K:\Q]}\log \vol(\widehat{V}),
		\end{align}
		where $\widehat{V}\subseteq V_{\mathbb{A}_K}$ is viewed as the adelic unit ball with respect to the norms $\metr_\omega$ and $\vol(\cdot)$ denotes the Haar measure on $V_{\mathbb{A}_K}$ normalized by 
		$\vol(V_{\mathbb{A}_K}/V)=1$.
	\end{itemize}
	
\end{definition}
\begin{remark} \label{rmk:relation of deg and chi}
	Let $\overline{V}$ be a generically trivial adelic vector bundle. If $K$ is a function field, by \cite[Proposition~4.3.18]{chen2020arakelov}, we have that $\widehat{\deg}(\overline{V})=\chi(\overline{V})$.
	If $K$ is a number field, the expression \eqref{chi for number field} for $\chi(\overline{V})$ can have a similar formula as \eqref{chi for function field}, see \cite[Remark~2.20]{burgos2016arithmetic}. Unlike the function field case, $\widehat{\deg}(\overline{V})\not=\chi(\overline{V})$ in general. However, we have that \[\widehat{\deg}(\overline{V}) = \chi(\overline{V})+O(\dim_K(V)\log(\dim_K(V)))\] see \cite[(3.5)]{boucksom2011okounkov}.
\end{remark}

\begin{remark}\label{rmk: small section of purification}
	Let $\overline{V}=(V,\metr)$ be a coherent adelic vector bundle. By the definition of purification, for any $s\in V$ and $\omega\in\Omega$, we have that $\|s\|_\omega\leq 1$ if and only if $\|s\|_{\pur,\omega}\leq 1$. Hence 
	\[\widehat{h}^0(\overline{V}) = \widehat{h}^0(\overline{V}_\pur).\]
	If furthermore, $\overline{V}$ is generically trivial, by \cite[Proposition~2.23]{burgos2016arithmetic},  then
	\[\chi(\overline{V})=\chi(\overline{V}_\pur).\]
\end{remark}

Next, we consider the classical arithmetic volume of a graded $K$-algebra of coherent $S$-normed vector spaces. 

\begin{definition} \label{def:classical arithvol}
	Let $\overline{E_\bullet}$ be a graded $K$-algebra of coherent $S$-normed vector spaces (i.e. $\overline{E_m}$ is coherent for each $m\in\N$).
	We define the \emph{classical arithmetic volume}
	\[\widehat{\vol}^\YZ(\overline{E_\bullet})\coloneq \limsup_{m\to\infty} \frac{\widehat{h}^0(\overline{E_m})}{m^{d+1}/(d+1)!},\]
	where $d$ is the Kodaira dimension of the graded $K$-algebra $E_{\bullet}$. 
	
	We further assume that $\overline{E_m}$ is generically trivial for any $m\in\N$. We define the \emph{classical arithmetic $\chi$-volume}
	\[\widehat{\vol}_\chi^\YZ(\overline{E_\bullet})\coloneq \limsup_{m\to\infty} \frac{\chi(\overline{E_m})}{m^{d+1}/(d+1)!}.\]
\end{definition}

\begin{definition}[\cite{chen2020arakelov}~Definition~7.5.1]
	Let $\overline{E_\bullet}$ be a graded $K$-algebra of coherent $S$-normed vector spaces. We say that $\overline{E_\bullet}$ is \emph{asymptotically pure} if
	\[\limsup\limits_{m\to\infty}\frac{\sigma(\overline{E_m})}{m}=0.\]
\end{definition}

\begin{prop} \label{arith volume and classical arith volume}
	Let $\overline{E_\bullet}$ be a graded $K$-algebra of coherent (resp. generically trivial) adelic vector bundles. One has the following inequality:
	\[\widehat{\vol}^\YZ(\overline{E_\bullet}) \leq \widehat{\vol}(\overline{E_\bullet}) \text{ (resp. $\widehat{\vol}_\chi^\YZ(\overline{E_\bullet}) \leq \widehat{\vol}_\chi(\overline{E_\bullet})$)},\]
	the equality holds if $\overline{E_\bullet}$ is asymptotically pure.
\end{prop}
\begin{proof}
	We first consider the arithmetic volume. 	
	Let $\overline{E_{\pur, m}}=(E_m, \metr_{\pur,m})$ be the purification of $\overline{E_m}=(E_m, \metr_m)$. Since each $\overline{E_m}$ is a coherent adelic vector bundle, by \cref{lemma:adelicvectorspaceisvectorbundle}~\ref{purification is adelic}, $\overline{E_{\pur,\bullet}}\coloneq\bigoplus\limits_{m=0}^\infty\overline{E_{\pur, m}}$ is a graded $K$-algebra of pure generically trivial adelic vector spaces. By \cref{rmk:relation of deg and chi}, $\widehat{h}^0(\overline{E_{\pur,m}}) = \widehat{h}^0(\overline{E_m})$ and hence $\widehat{\vol}^\YZ(\overline{E_{\pur,\bullet}})=\widehat{\vol}^\YZ(\overline{E_\bullet})$. On the other hand, by \cite[Proposition~2.13]{burgos2016arithmetic}, $\overline{E_{\pur,m}}$ is the adelic vector bundle associated to some Hermitian vector bundle $\overline{\E_m}$ on $C$ if $K$ is a number field or to some vector bundle $\E_m$ on $C$ if $K$ is the function field of a curve. By \cite[Proposition~4.3.23, Proposition~4.3.24]{chen2020arakelov}, we have that
	\begin{align}\label{comparision of positive degree and small section}
		|\widehat{h}^0(\overline{E_{\pur,m}})-\widehat{\deg}_+(\overline{E_{\pur,m}})|\le C_1\cdot \dim_K(E_m)\log(\dim_K(E_m))
	\end{align}
	for some constant $C_1>0$ dependent only on $K$. Combine \eqref{comparision of positive degree and small section} and the fact that $\frac{\dim_K(E_m)}{m^d}$ is bounded when $m\to\infty$, we have that $\widehat{\vol}(\overline{E_{\pur,\bullet}}) = \widehat{\vol}^\YZ(\overline{E_{\pur,\bullet}})$. It remains to compare  $\widehat{\vol}(\overline{E_{\pur,\bullet}})$ and $\widehat{\vol}(\overline{E_{\bullet}})$. Notice that \begin{equation}\label{relation between purification and self}
		e^{-\sigma_\omega(\overline{E_m})}\metr_{\pur,m,\omega}\leq \metr_{m,\omega}\leq \metr_{\pur,m,\omega}
	\end{equation} for any $\omega\in\Omega$. By \cite[Proposition~4.3.21]{chen2020arakelov}, we have that
	\begin{align*}
		\widehat{\deg}_+(\overline{E_{\pur,m}})\leq \widehat{\deg}_+(\overline{E_m})&\leq \widehat{\deg}_+(E_m,(e^{-\sigma_\omega(\overline{E_m})}\metr_{\pur,m,\omega})_{\omega\in\Omega})\\
		&\leq \widehat{\deg}_+(\overline{E_{\pur,m}})+\dim_K(E_m)\sigma(\overline{E_m}).
	\end{align*}
	Hence $\widehat{\vol}(\overline{E_{\pur,\bullet}})\leq \widehat{\vol}(\overline{E_{\bullet}})$ and 
	\begin{align*}
		\widehat{\vol}(\overline{E_{\bullet}})&\leq \limsup_{m\to\infty} \frac{\widehat{\deg}_+(\overline{E_{\pur,m}})+\dim_K(E_m)\sigma(\overline{E_m})}{m^{d+1}/(d+1)!}\\
		&\leq \widehat{\vol}(\overline{E_{\pur,\bullet}})+ (d+1)!\limsup_{m\to\infty} \frac{\dim_K(E_m)\sigma(\overline{E_m})}{m^{d+1}}.
	\end{align*}
	Again, notice that $\frac{\dim_K(E_m)}{m^d}$ is bounded when $m\to\infty$, if $\overline{E_\bullet}$ is asymptotically pure, i.e. $\limsup\limits_{m\to\infty}\frac{\sigma(\overline{E_m})}{m}=0$, then $\widehat{\vol}(\overline{E_{\bullet}})\leq \widehat{\vol}(\overline{E_{\pur,\bullet}})$. This completes the proof for arithmetic volume.
	
	For the arithmetic $\chi$-volume, the proof is similar.  Assume that each $\overline{E_m}$ is generically trivial, we keep the notation as before. By \cite[Proposition~2.23]{burgos2016arithmetic}, we have that $\chi(\overline{E_{\pur,m}}) = \chi(\overline{E_m})$ and hence $\widehat{\vol}_\chi^\YZ(\overline{E_{\pur,\bullet}})=\widehat{\vol}_\chi^\YZ(\overline{E_\bullet})$. By \cref{rmk:relation of deg and chi}, we have that
	\begin{align}\label{comparision of degree and chi}
		|\chi(\overline{E_{\pur,m}})-\widehat{\deg}(\overline{E_{\pur,m}})|\leq C_2\cdot\dim_K(E_m)\log(\dim_K(E_m))
	\end{align}
	for some constant $C_2>0$ dependent only on $K$. This implies that $\vol_\chi(E_{\pur,\bullet})=\vol_\chi^\YZ(E_{\pur,\bullet})$. It remains to compare  $\widehat{\vol}_\chi(\overline{E_{\pur,\bullet}})$ and $\widehat{\vol}_\chi(\overline{E_{\bullet}})$. By \eqref{relation between purification and self}, we have
	\[\widehat{\deg}(\overline{E_{\pur,m}})\leq \widehat{\deg}(\overline{E_m})\leq \widehat{\deg}(\overline{E_{\pur,m}})+\dim_K(E_m)\sigma(\overline{E_m}).\]
	As above, this implies that $\widehat{\vol}_\chi(\overline{E_{\pur,\bullet}}) \leq \widehat{\vol}_\chi(\overline{E_\bullet})$ and the equality holds if $\overline{E_\bullet}$ is asymptotically pure, which completes the proof.
\end{proof}

\section{Review of compactified $S$-metrized divisors}

\label{review of compactified divisors}

In this section, we recall the notion of compactified $S$-metrized divisors on quasi-projective varieties over adelic curves defined in \cite{cai2024abstract}. Recall that given a $\Q$-vector space $M$, and a subset $N$, the the closure $\overline{N}$ of $N$ with respect to the \emph{finite subspace topology} is given by $\overline{N}=\varinjlim\limits_{E}\overline{(N\cap E)}$, where $E$ runs through all finite dimensional subspaces of $M$ and $\overline{(N\cap E)}$ is the closure of $N\cap E$ in $E$ with respect to the canonical euclidean topology.

\subsection{Compactified (geometric) divisors}

\begin{art} \label{geometric intersection number}
	Let $X$ be a projective variety over a field $K$. We set $N_{\gm}(X)$ the set of nef divisors in $\Div(X)$, and $N_{\gm,\Q}(X)$ the cone in $\Div_\Q(X)\coloneq \Div(X)\otimes_\Z\Q$ generated by $N_{\gm}(X)$. For a quasi-projective variety $U$ over $K$, a (\emph{geometric}) \emph{boundary divisor} of $U$ is a pair $({X}_0,B)$ consisting of a projective $K$-model $U\hookrightarrow X_0$ and an effective divisor $B\in {\Div}_\Q({X}_0)$ such that $|B|={X}_0\setminus U$. We set
	\begin{equation}\label{eq:geometric model divisors1}
		\Div_{\Q}(U)_{\mo}\coloneq\varinjlim_X \Div_\Q(X),  \ \ N_{\gm,\Q}(U)\coloneq\varinjlim_XN_{\gm,\Q}(X),
	\end{equation}
	where $X$ ranges over all projective $K$-models of $U$. Let $(X_0,B)$ be a boundary divisor of $U$. The (\emph{$B$-})\emph{boundary topology} on  $\Div_{\Q}(U)_{\mo}$ is defined such that a basis of neighborhoods of a divisor $D$ is given by
	\begin{align*}
		B(r,D)\coloneq\{E\in {\Div}_\Q(U)\mid -rB\leq {E}-{D}\leq rB\},
		\ \ r\in \Q_{>0}.
	\end{align*}
	Notice that the boundary topology is independent of the choice of $(X_0,B)$. We write $\widetilde{\Div}_\Q(U)_\cpt$ (resp. $\widetilde{\Div}_\Q(U)_{\snef}$) the completion of  $\Div_{\Q}(U)_{\mo}$ (resp. ${N}_{\gm,\Q}(U)$) with respect to the boundary topology and $\widetilde{\Div}_\Q(U)_{\mathrm{int}}\coloneq \widetilde{\Div}_\Q(U)_{\snef}-\widetilde{\Div}_\Q(U)_{\snef}$. Moreover, write $\widetilde{\Div}_\Q(U)_{\nef}$ for the closure of $\widetilde{\Div}_\Q(U)_{\snef}$ in $\widetilde{\Div}_\Q(U)_{\mathrm{int}}$ with respect to the finite subspace topology. An element in  $\widetilde{\Div}_\Q(U)_\cpt$ is called a \emph{compactified} (\emph{geometric}) \emph{divisor}, we have a natural map $$\widetilde{\Div}_\Q(U)_\cpt\to \Div_\Q(U), \ D\mapsto D|_U.$$ For $D\in \widetilde{\Div}_\Q(U)_\cpt$, we say $D$ is \emph{strongly nef} (resp. \emph{nef}, \emph{integrable}) if $D\in \widetilde{\Div}_\Q(U)_{\snef}$ (resp. $\widetilde{\Div}_\Q(U)_{\nef}$, $\widetilde{\Div}_\Q(U)_{\mathrm{int}}$).
	We have a symmetric multilinear map 
\[\underbrace{\widetilde{\Div}_{\Q}({U})_{\mathrm{int}}\times\cdots\times\widetilde{\Div}_{\Q}({U})_{\mathrm{int}}}_{d\text{-times}}\to \R, \ \ ({D_1},\cdots, {D_d}) \mapsto {D_1}\cdots {D_d},\]
see \cite[Proposition~4.1.1]{yuan2021adelic}. For simplicity, for $D\in \widetilde{\Div}_{\Q}({U})_{\mathrm{int}}$  we denote $\deg_D(U)\coloneq D^d$.
\end{art}
\begin{remark}
	In \cite[Definition~3.4]{cai2024abstract}, the space $\Div_{\Q}(U)_{\mo}$ defined to be $\varinjlim_X \Div_\Q(X)$, where $X$ ranges over all proper $K$-models of $U$ (although $U$ is quasi-projective), so the completion $\widetilde{\Div}_{\Q}(U)_\cpt$ in \cite{cai2024abstract} may be larger than the one defined above. 
	In the local and global cases below, the same situation happens. The reason why we only consider projective $K$-models is that the ampleness condition in \cref{def:ampleser} is crucial in our paper and that the results in \cite{yuan2021adelic} are induced from the ones in projective case. 
\end{remark}

\begin{art} \label{geometric volume}
	Let $U$ be a $d$-dimensional normal quasi-projective variety over a field $K$ and ${D}\in\widetilde{\Div}_{\Q}(U)_\cpt$. Notice that an element $f$ in the function field $K(U)$ of $U$ can determine a compactified divisor denoted by $\mathrm{div}(f)$. We can introduce the space of \emph{global sections} of $\overline{D}$ as     
	\[H^0(U,{D})\coloneq \{f\in K(U)^{\times}\mid \mathrm{div}(f)+{D}\ge 0\}\cup\{0\}\]
	which is a $K$-vector space since $U$ is normal by \cite[Proposition~2.4.13~(6)]{chen2020arakelov} (on projective $K$-models of $U$). The set $H^0(U,{D})$ is not necessarily stable under addition if $U$ is not normal, see \cite[Example~2.4.15]{chen2020arakelov}.
	We have a natural injective map $H^0(U,D)\hookrightarrow H^0(U,D|_U)$. If $D$ is induced by a divisor on some projective model $X$ of $U$ which still denoted by $D$, then $H^0(X,D)\subset H^0(U,D)$, the equality holds if $X$ is integrally closed in $U$ by \cite[Lemma~2.3.7]{yuan2021adelic}. We have the following properties: 
	\begin{enumerate}
		\item (\cite[Lemma~5.1.6, Theorem~5.2.1~(1)]{yuan2021adelic}) \label{definition of volume of divisors} The $K$-vector space $H^0(U,{D})$ has finite dimension and the limit \[\vol(D)\coloneq \lim\limits_{m\to\infty}\frac{\dim_K(H^0(U,m{D}))}{m^d/d!}\] exists;
		\item (\cite[Theorem~5.2.1~(2), Theorem~5.2.9]{yuan2021adelic}) \label{continuity of geometric volume} If $D_n$ converges to $D$ in $\widetilde{\Div}_\Q(U)_\cpt$ with respect to the boundary topology, then $\lim\limits_{n\to\infty}\vol(D_n)=\vol(D)$.
	\end{enumerate}
	We call $\vol(D)$ the \emph{volume} of $D$. We say that ${D}\in\widetilde{\Div}_{\Q}(U)_\cpt$ is \emph{big} if ${\vol}(D)>0$. We refer to \cite[\S 5.1--5.2]{yuan2021adelic} for details and various properties of volumes. 
\end{art}

\subsection{Local theory: mixed Monge-Amp\`ere measures}
In this subsection, we fix a field $K$ with an absolute value $\val_v$ and a $d$-dimensional quasi-projective variety $U$ over $K$. We denote by $U^\an=U_v^\an$ the Berkovich analytification of $U$ with respect to the place $v$, see \cite[\S 3.4--3.5]{berkovich1990spectral}. 
\begin{art} \label{Green functions} 
	For a $\Q$-Cartier divisor $D\in \Div_\Q(U)\coloneq \Div(U)\otimes_\Z\Q$,  $D$ can be written by local equations $f \in \OO_U(V)^\times \otimes_\Z \Q$.  Then we get a well-defined continuous real function $\log|f|$ on $V^\an$ for such open subsets $V$.  A \emph{Green function} for $D$ is a continuous real function $g_D$ such that on $(U \setminus |D|)^\an$ such that for any local equation $f$ of $D$ on an open subset $V$ of $U$, we have that  $g_D+\log|f|$ is a continuous function on $V^\an$.
	
	We denote by $\widehat{\Div}(U)$ (resp. $\widehat{\Div}_\Q(U)$) the group of pairs $(D,g)$ with $D\in \Div(U)$ (resp. $D\in\Div_\Q(U)$) and $g$ a Green function for $D$. Obviously, we have that $\widehat{\Div}_\Q(U)\simeq\widehat{\Div}(U)\otimes_\Z\Q$.
\end{art}

\begin{art} \label{local model divisors}
	For a projective variety $X$ over $K$, we denote $\widehat{\Div}_{\Q}(X)_\mo$ the group of $\Q$-Cartier divisors with \emph{model} Green functions, $N_{\mo,\Q}(X)$ (resp. $N_{\tFS,\Q}(X)$) the cone in $\widehat{\Div}_{\Q}(X)$ generated by the Cartier divisors with \emph{semi-positive} model Green functions (resp. \emph{$\tFS$}-Green functions) and $\widehat{\Div}_{\Q}(X)_\tFS \coloneq N_{\tFS,\Q}(X)-N_{\tFS,\Q}(X)$.
\end{art}
\begin{art} \label{local nef divisor}
	Let $U$ be a quasi-projective variety  over $K$.  Assume that $v$ is non-trivial. A \emph{boundary divisor} of $U$ is a pair $({X}_0,\overline{B})$ consisting of a projective $K$-model $U\hookrightarrow X_0$ and an effective divisor $\overline{B}=(B,g_B)\in \widehat{\Div}_\Q({X}_0)_\mo$ such that $|B|={X}_0\setminus U$. We set
	\begin{equation}\label{eq:model divisors}
		\widehat{\Div}_\Q(U)_\mo\coloneq\varinjlim_X \widehat{\Div}_{\Q}(X)_\mo,\ \  N_{\mo,\Q}(U)\coloneq\varinjlim_XN_{\mo,\Q}(X)
	\end{equation}
	where $X$ ranges over all projective $K$-models of $U$. Then a boundary divisor $(X_0,B)$ gives a \emph{boundary topology} on $\widehat{\Div}_\Q(U)_\mo$ as \cref{geometric intersection number} which is independent of the choice of $(X_0,B)$. We denote by $\widehat{\Div}_\Q(U)_\cpt$ (resp.  $\widehat{\Div}_{\Q}(U)_\snef$) the completion of $\widehat{\Div}_\Q(U)_\mo$ (resp. $N_{\mo,\Q}(U)$) with respect to the boundary topology and $\widehat{\Div}_{\Q}(U)_{\mathrm{int}}\coloneq \widehat{\Div}_{\Q}(U)_\snef-\widehat{\Div}_{\Q}(U)_\snef$. Moreover, we denote by $\widehat{\Div}_\Q(U)_{\nef}$ the closure of $\widehat{\Div}_{\Q}(U)_\snef$ in $\widehat{\Div}_{\Q}(U)_{\mathrm{int}}$ with respect to the finite subspace topology. 
    When $v$ is trivial, we can do a similar procedure on $(\widehat{\Div}_{\tFS,\Q},N_{\tFS})$, the corresponding spaces are also denoted similarly. An element $\overline{D}$ in $\widehat{\Div}_{\Q}(U)$ is called a \emph{compactified metrized divisor} on $U$, it is called \emph{strongly nef} (resp. \emph{nef}) if $\overline{D}\in \widehat{\Div}_\Q(U)_\snef$ (resp. $\widehat{\Div}_\Q(U)_\nef$). We have a homomorphism
	 \[\widehat{\Div}_\Q(U)_{\cpt}\rightarrow \widetilde{\Div}_\Q(U)_\cpt\] 
	and an injective homomorphism
	\[\widehat{\Div}_\Q(U)_{\cpt}\hookrightarrow\widehat{\Div}_\Q(U),\]
	see \cite[4.20]{cai2024abstract}. So an element in $\widehat{\Div}_{\Q}(U)_\cpt$ can be uniquely written as $(D,g)$ with $D\in\widetilde{\Div}_\Q(U)_\cpt$ and $g$ a Green function for $D|_U$.
	
	There is a unique way to associate to any $d$-dimension quasi-projective $U$ over $K$ and to a family $\overline{D_1},\dots, \overline{D_d}\in\widehat{\Div}_\Q(U)_\nef$ a positive Radon measure $c_1(\overline{D_1})\wedge\cdots\wedge c_1(\overline{D_d})$ on $U^\an$ satisfying multilinearity, symmetry, continuity (respect to boundary topology) and other properties, see \cite[Proposition~4.45]{cai2024abstract}.
\end{art}

\begin{remark} \label{nef become strongly nef after shrinking U}
	If we have finitely many nef compactified metrized divisors $\overline{D_j}=(D_j,g_j)$ for $j=0,\dots, k$, then there is an open subset $U'\subset U$ such that the restriction of $\overline{D_j}$ on $U'$ is strongly nef, see \cite[Remark~4.34]{cai2024abstract}.
\end{remark}
\subsection{Global theory: Compactified $S$-metrized divisors}

\label{subsection:Global theory: Compactified (model) divisors}

In the rest of this section, we fix an adelic curve $S=(K,\Omega,\mathcal{A},\nu)$ satisfying the following conditions: 
\begin{enumeratea}
	\item \label{finite measure for archimedean} $\nu(\Omega_\infty) < \infty$;
	\item \label{exists finite measure for subset} the set $\nu(\mathcal{A}) \not\subset \{0,+\infty\}$;
	\item \label{K countable or A discrete} either the $\sigma$-algebra $\mathcal{A}$ is discrete, or that $K$ is countable.
\end{enumeratea}
We also fix a $d$-dimensional quasi-projective variety $U$ over $K$ and set $U_{\overline{K}}\coloneq U\times_{\Spec(K)}\Spec(\overline{K})$.
\begin{art} \label{adelic divisors}
An \emph{$S$-Green function} for a divisor $D\in \Div_\Q(U)$ is a family of Green functions $g_{D,\omega}$ for the base change $D_{\omega}$ of $D$ to $U_\omega$ with $\omega$ running over $\Omega$. We denote $\widehat{\Div}_{S,\Q}(U)$ the group of pairs $(D,g_D)$ with $D$ a $\Q$-Cartier divisor and $g_D$ an \emph{$S$-measurable}, \emph{locally $S$-bounded} $S$-Green function for $D$, see \cite[Definition~6.3, Definition~6.9]{cai2024abstract}.
Given $\overline{D}\coloneq (D,g_D)\in\widehat{\Div}_{S,\Q}(U)$ and $x\in U(\overline{K})$, we denote its \emph{height function} as 
\[h_{\overline{D}}(x)\coloneq \int_{\Omega_{\overline{K}}} g_{\omega}(x)\ \nu_{\overline{K}}(d\omega),\]
where $S_{\overline{K}}=(\overline{K}, \Omega_{\overline{K}},\mathcal{A}_{\overline{K}}, \nu_{\overline{K}})$ denotes the canonical extension of the adelic curve structure on $K$ to the algebraic closure $\overline{K}$ as explained in \cite[\S 3.4.2]{chen2020arakelov}.
	
	If $U=X$ is projective, a pair $(D,g_D)\in \widehat{\Div}_{S,\Q}(X)$ if and only if $(D,g_D)$ is an \emph{adelic $\Q$-divisor} in the sense of Chen and Moriwaki, \cite[\S 6.2.3]{chen2020arakelov}, see \cite[Remark~6.21]{cai2024abstract}. Note that in this case, it also ocincides with the definition of height function given in \cite[Definition 6.2.1]{chen2020arakelov}.
\end{art}

\begin{art} \label{global adelic Green functions for proper varieties}
	Let $X$ be a $d$-dimensional projective variety over $K$.
	We consider the submonoid
	$N_{S,\Q}(X)$  of $\widehat{\Div}_{S,\Q}(X)$ consisting of $(D,g_D)\in \widehat{\Div}_{S,\Q}(X)$ with $g_{D,\omega}$ a uniform limit of semipositive model (resp.~tFS-)Green functions for $D_\omega$ on $X_\omega^\an$ if $\omega \in \Omega \setminus \Omega_0$ (resp.~$\omega \in \Omega_0$). Write $M_{S,\Q}(X)\coloneq N_{S,\Q}(X)-N_{S,\Q}(X)$ which is an ordered $\Q$-vector space: $(D,g_D)\geq 0$ if and only if $D\geq 0$ and $g_{D,\omega}\geq 0$ for any $\omega\in\Omega$. We have a symmetric multilinear map
	\begin{align}\label{symmetric multilinear map}
		M_{S,\Q}(X)^{d+1}\to \R, \ \ (D_0,\cdots, D_d)\mapsto (D_0\cdots D_d\mid X)_S.
	\end{align}
	For any $\overline{D}\in M_{S,\Q}(X)$ and integral closed subscheme $Z$ of $X$, set $h_{\overline{D}}(Z)\coloneq (\overline{D}^{\dim(Z)+1}\mid Z)_S$ (the arithmetic intersection number is well-defined without assuming that $X$ is geometrically integral, see \cite{chen2021arithmetic}). Notice that the restriction of $\overline{D}$ on $Z$ may not be defined, but the restriction of the corresponding \emph{$S$-metrized} (\emph{$\Q$-})\emph{line bundle} is defined (see \cite[\S 7.2]{cai2024abstract}), so $h_{\overline{D}}(Z)$ is well-defined.
	
	We say that $\overline{D}\in N_{S,\Q}(X)$ is \emph{$S$-ample} if $D$ is ample and if
	there is $\varepsilon > 0$ such that for every integral closed subscheme $Z$ of $X$, we have
	\[h_{\overline{D}}(Z)\geq \varepsilon\deg_D(Z)(\dim(Z)+1).\]
	The $S$-ample divisors form a cone $N^+_{S,\Q}(X)$ of $M_{S,\Q}(X)$. We say that $\overline{D}\in N_{S,\Q}(X)$ is \emph{$S$-nef} if $\overline{D}$ is in the closure of $N^+_{S,\Q}(X)$ in $M_{S,\Q}(X)$ with respect to the finite subspace topology. The $S$-nef divisors form a subcone $N'_{S,\Q}(X)$ of $N_{S,\Q}(X)$ which is preserved under pull-backs and tensor products. 
\end{art}

\begin{art} \label{def:CM divisors on non-proper}
	Set \[\widehat{\Div}_{S,\Q}(U)_\CM\coloneq\varinjlim_{X}\widehat{\Div}_{S,\Q}(X), \ \ N_{S,\Q}(U)\coloneq \varinjlim_X N_{S,\Q}(X),  \ \ N_{S,\Q}'(U)\coloneq \varinjlim_X N_{S,\Q}'(X),\]
 where $X$ runs through all projective $K$-model of $U$. 
	
	A {\emph{weak boundary divisor}} (resp. a \emph{boundary divisor}) of $U$ is a pair $({X}_0,\overline{B})$ consisting of a projective $K$-model $U\hookrightarrow {X}_0$ over $K$ and  
	$\overline{B}\in M_{S,\Q}(U)_{\geq 0}$  such that $|B|\subset{X}_0\setminus U$ (resp. $|B|={X}_0\setminus U$). 
	The weak boundary divisors form a directed subset $T$ of $M_{S,\Q}(U)_{\geq 0}$. As in \cref{geometric intersection number}, a weak boundary divisor $(X_0,\overline{B})$ gives a topology, called \emph{$\overline{B}$-boundary topology}, on $\widehat{\Div}_{S,\Q}(U)_\CM$.
	We set $\widehat{\Div}_{S,\Q}(U)_\CM^{d_{\overline{B}}}$ (resp. $N_{S,\Q}(U)^{d_{\overline{B}}}$, resp. $N_{S,\Q}'(U)^{d_{\overline{B}}}$) as the completion of $\widehat{\Div}_{S,\Q}(U)_\CM$ (resp. $N_{S,\Q}(U)$, resp. $N_{S,\Q}'(U)$) with respect to the $\overline{B}$-boundary topology. The space
	\[\widehat{\Div}_{S,\Q}(U)_{\cpt}\coloneq\varinjlim_{\overline{B}\in T}\widehat{\Div}_{S,\Q}(U)_\CM^{d_{\overline{B}}}\]
	is called the space of \emph{compactified $S$-metrized} (\emph{$\Q$-})\emph{divisors} of $U$. We also set
	\[\widehat{\Div}_{S,\Q}(U)_{\relsnef}\coloneq\varinjlim_{\overline{B}\in T}N_{S,\Q}(U)^{d_{\overline{B}}}, \ \ \widehat{\Div}_{S,\Q}(U)_{\arsnef}\coloneq\varinjlim_{\overline{B}\in T}N_{S,\Q}'(U)^{d_{\overline{B}}},\]
	\[\widehat{\Div}_{S,\Q}(U)_\relint\coloneq \widehat{\Div}_{S,\Q}(U)_\relsnef-\widehat{\Div}_{S,\Q}(U)_\relsnef,\] 
	\[\widehat{\Div}_{S,\Q}(U)_\arint\coloneq \widehat{\Div}_{S,\Q}(U)_\arsnef-\widehat{\Div}_{S,\Q}(U)_\arsnef,\]
	and denote by $\widehat{\Div}_{S,\Q}(U)_\relnef$
	(resp. $\widehat{\Div}_{S,\Q}(U)_\arnef$) the closure of $\widehat{\Div}_{S,\Q}(U)_\relsnef$ (resp. $\widehat{\Div}_{S,\Q}(U)_\arsnef$) in $\widehat{\Div}_{S,\Q}(U)_\relint$ (resp. $\widehat{\Div}_{S,\Q}(U)_\arint$) with respect to the finite subspace topology. A compactified $S$-metrized divisors $\overline{D}\in \widehat{\Div}_{S,\Q}(U)_\cpt$ is called \emph{relatively integrable} (resp. \emph{strongly relatively nef}, resp. \emph{relatively nef}, resp. \emph{arithmetically integrable}, resp. \emph{strongly arithmetically nef}, resp. \emph{arithmetically nef}) if $\overline{D}\in \widehat{\Div}_{S,\Q}(U)_\relint$ (resp. $\widehat{\Div}_{S,\Q}(U)_\relsnef$, resp. $\widehat{\Div}_{S,\Q}(U)_\relnef$, resp. $\widehat{\Div}_{S,\Q}(U)_\arint$, resp. $\widehat{\Div}_{S,\Q}(U)_\arnef$, resp. $\widehat{\Div}_{S,\Q}(U)_\arnef$).
	
	We have injective homomorphisms
	\begin{align} \label{eq:injectivity of compactified divisors}
		\widehat{\Div}_{S,\Q}(U)_{\arint}\hookrightarrow\widehat{\Div}_{S,\Q}(U)_{\relint}\hookrightarrow\widehat{\Div}_{S,\Q}(U)_\cpt\hookrightarrow \widehat{\Div}_{S,\Q}(U)
	\end{align}
	(see \cite[Proposition~7.4]{cai2024abstract} for the injectivity) and a forgetting homomorphism
	\[\widehat{\Div}_{S,\Q}(U)_\cpt\to \widetilde{\Div}_{\Q}(U)_\cpt.\]
	By the injectivity of \eqref{eq:injectivity of compactified divisors}, throughout this paper, we write an element $\overline{D}\in\widehat{\Div}_{S,\Q}(U)_\cpt^\YZ$ as $(D,g)$ with $D\in\widetilde{\Div}_{\Q}(U)_\cpt$ and $g$ an $S$-measurable, locally $S$-bounded $S$-Green function for $D|_U$. For a compactified divisor $D\in\widetilde{\Div}_{\Q}(U)_\cpt$, an \emph{$S$-Green function for $D$} is an $S$-Green function $g$ for $D|_U$ such that $(D,g)\in\widehat{\Div}_{S,\Q}(U)_\cpt$. 
    Let $\overline{D}=(D,g)\in \widehat{\Div}_{S,\Q}(U)_\cpt$ and $x\in U(\overline{K}$), we write $h_{\overline{D}}(x)\coloneq \int_{\Omega}g_\omega(x)\, \nu(d\omega).$
\end{art}
\begin{remark} \label{remark:cpt for projective case}
	If $U=X$ is projective, by \cite[Proposition~7.10]{cai2024abstract}, we have that
	\[\widehat{\Div}_{S,\Q}(X)_\CM=\widehat{\Div}_{S,\Q}(X)_\cpt=\widehat{\Div}_{S,\Q}(X)_\relint = \widehat{\Div}_{S,\Q}(X)_\arint,\]
	\[\text{$\widehat{\Div}_{S,\Q}(X)_{\relnef}=N_{S,\Q}(X)$ \ \ and \ \  $\widehat{\Div}_{S,\Q}(X)_{\arnef}=N_{S,\Q}'(X)$}.\]
\end{remark}
\subsection{Extension of Yuan-Zhang's arithmetic intersection number}

\label{Extension of Yuan-Zhang's arithmetic intersection number}
We have a symmetric bilinear map 
\[\underbrace{\widehat{\Div}_{S,\Q}({U})_\arint\times\cdots\times\widehat{\Div}_{S,\Q}({U})_\arint}_{(d+1)\text{-times}}\to \R, \ \ (\overline{D_0},\dots, \overline{D_d}) \mapsto (\overline{D_0}\cdots \overline{D_d}\mid U)_S,\] see \cite[Theorem~7.22]{cai2024abstract}.
We want to extend this map to relatively nef case, see \cite[Definition~4.6]{burgos2024on} for archimedean case and \cite[Theorem~11.3]{cai2024abstract} for general case.
\begin{art}\label{def:equivalent singularities}
	Let $D\in \widetilde{\Div}_\Q(U)_\cpt$ and $g_1, g_2$ measurable $S$-Green functions for $D$, see \cref{global adelic Green functions for proper varieties}. We say that $g_1$ is \emph{more singular than $g_2$}, 
	denoted by $[g_1]\leq [g_2]$, if there is a $\nu$-integrable function $C\in \mathscr{L}^1(\Omega,\mathcal{A},\nu)$ such that $g_{1,\omega}\leq g_{2,\omega}+C(\omega)$ for any $\omega\in\Omega$. We say that $g_1$, $g_2$ \emph{have equivalent singularities}, 
	denoted by $[g_1]=[g_2]$, if $[g_2]\leq [g_1]$ and $[g_1]\leq[g_2]$.
\end{art}
Denote by $\widehat{\Div}_{S,\Q}({U})_{\relnef}^\arnef$ the monoid of relatively nef compactified $S$-metrized divisors $\overline{D}=(D,g)\in \widehat{\Div}_{S,\Q}({U})_{\relnef}$ satisfying the following property: there is an $S$-Green function $g'$ for ${D}$ such that $(D,g')\in \widehat{\Div}_{S,\Q}({U})_{\arnef}.$

\begin{art} \label{extension of YZ intersection number}
	Let $\overline{D_j}=({D}_j,g_j)\in \widehat{\Div}_{S,\Q}({U})_{\relnef}$ for $j=0,\dots, d$. Assume that there are $\overline{D_j'}=({D}_j,g_j')\in \widehat{\Div}_{S,\Q}({U})_{\arnef}$ with $[g_j]\leq [g_j']$ for $j=0,\dots, d$  (in particular, $\overline{D_j}\in \widehat{\Div}_{S,\Q}({U})_{\relnef}^\arnef$). For any $\omega\in \Omega$, we set 
	\[E(\mathbf{g}_\omega',\mathbf{g}_\omega)\coloneq \sum\limits_{j=0}^d\int_{U_\omega^\an}(g_{j,\omega}-g_{j,\omega}') \, c_1(\overline{D_{0,\omega}})\wedge\cdots\wedge c_1(\overline{D_{j-1,\omega}})\wedge c_1(\overline{D_{j+1,\omega}'})\wedge \cdots\wedge c_1(\overline{D_{d,\omega}'}),\]
	\[E(\mathbf{g}',\mathbf{g})\coloneq \int_{\Omega} E(\mathbf{g}_\omega',\mathbf{g}_\omega) \, \nu(d\omega)\]
	and define
	\begin{align*}(\overline{D_0}\cdots\overline{D_d}\mid U)_S\coloneq (\overline{D_0'}\cdots\overline{D_d'}\mid U)_S+E(\mathbf{g}',\mathbf{g}).
	\end{align*}
	By \cite[Lemma~10.7]{cai2024abstract}, the map $\omega\mapsto E(\mathbf{g}_\omega',\mathbf{g}_\omega)$ is $\nu$-integrable. The value $(\overline{D_0}\cdots\overline{D_d}\mid U)_S$ is in $\R\cup\{-\infty\}$ and it is finite if $[g_j]=[g_j']$ for any $j$. It is shown in \cite[Proposition~10.10, Theorem~11.2]{cai2024abstract} that $(\overline{D_0}\cdots\overline{D_d}\mid U)_S$ is independent of the choices of $\overline{D_j'}$. Notice that every $\overline{D}=(D,g)\in \widehat{\Div}_{S,\Q}(U)_{\relnef}^\arnef$ satisfies the assumption based on the following fact: if $\overline{D'}=(D,g')\in \widehat{\Div}_{S,\Q}({U})_\arnef$, then $h\coloneq \max\{g,g'\}$ is an $S$-Green function for ${D}$ with $({D},h)\in \widehat{\Div}_{S,\Q}({U})_\arnef$ and $[g]\leq [h]$. Hence, we have a symmetric mulitlinear map
	\[\underbrace{\widehat{\Div}_{S,\Q}({U})^\arnef_\relnef\times\cdots\times\widehat{\Div}_{S,\Q}({U})^\arnef_\relnef}_{(d+1)\text{-times}}\to \R\cup\{-\infty\}, \ \ (\overline{D_0},\dots, \overline{D_d}) \mapsto (\overline{D_0}\cdots \overline{D_d}\mid U)_S.\]
\end{art}
\subsection{Graded linear series}
\begin{art} \label{flag of a quasi-projective variety}
	An \emph{admissible flag} of $U_{\overline{K}}$ is a flag \[Y_\bullet\colon \{\text{pt}\}=
	Y_d\subset Y_{d-1}\subset\cdots \subset Y_0=U_{\overline{K}}\]
	of irreducible subvarieties of $U_{\overline{K}}$ such that $\mathrm{codim}(Y_i)=i$ in $U_{\overline{K}}$ and that the closed point $Y_d$ is regular in each $Y_i$. Giving an admissible flag $Y_\bullet$ is equivalent to choosing a closed regular point $p\coloneq Y_d$ in $U_{\overline{K}}$ and a system of parameters $\{z_1, \cdots, z_d\}$ such that $\widehat{\OO_{U_{\overline{K}},p}}=\overline{K}[[z_1,\cdots, z_d]]$, where $\kappa(p)$ is the residue field of $p$, and $\widehat{\OO_{U_{\overline{K}},p}}$ is the Krull completion of the local ring $\OO_{U_{\overline{K}},p}$ with respect to its maximal ideal.  Hence we have a valuation $v_Y$ on $\widehat{\OO_{U_{\overline{K}},p}}$ as follows: for any $f= \sum\limits_{\alpha\in \N^d}a_\alpha z^\alpha\in \widehat{\OO_{U_{\overline{K}},p}}=\overline{K}[[z_1,\cdots, z_d]]$,
	\[v_Y(f)\coloneq \min\{\alpha\in \N^d\mid a_\alpha\not=0\},\]
	where the minimum is taken with respect to the lexicographic order on the variables $z_1,\cdots, z_d$. 
\end{art}

In the rest of this section, we assume that $U$ is normal and fix an admissible flag $Y_\bullet$ of $U_{\overline{K}}$. Then we have a valuation function $v_Y$ on $H^0(U, D)\otimes_K\overline{K}$ for any $D\in\widetilde{\Div}_{\Q}(U)_\cpt$.

Next we introduce the notion of a graded linear series containing an ample series. It can be thought of as a generalization of the bigness property of $D$ to arbitrary graded linear sub-series of $D$.

\begin{definition}
	\label{def:ampleser}
	A \emph{graded linear series} of $D\in \widetilde\Div_\Q(U)_\cpt$ is a graded sub-algebra $V_\bullet=\{V_m\}_{m\in\N}$ of $\{H^0(U,mD)\}_{m\in\N}$ with $\dim_K(V_m)<\infty$ for any $m\in\N$. 
	We say a graded linear series $V_\bullet$ of $D$ \emph{contains an ample series} if the following two condition holds:
	\begin{enumerate}[resume,leftmargin=*,label=\it(\alph*),ref=\it{(\alph*)}]
		\item \label{ampleser a} $V_m\neq \{0\}$ for large enough $m$;
		\item there is a positive integer $n$, a projective model $X$ of $U$ and an ample divisor $A$ on it such that $nD\ge A$ and the canonical injection $H^0(U,mA)\xhookrightarrow{\otimes s^{\otimes m}} H^0(U,mnD)$ factors through $V_{mn}$, i.e.
		\[H^0(U,mA)\xhookrightarrow{\otimes s^{\otimes m}}V_{mn}\xhookrightarrow{}{H^0}(U,mnD),\]
		for some non-zero $s\in H^0(U,nD-A)$ for all $m\in\N$.
	\end{enumerate}
\end{definition}

\begin{art} \label{okounkov body of a graded linear series}
	Let $V_\bullet=\{V_m\}_{m\in\N}$ be a graded linear series  of $D$. As \cref{geometric volume}~\ref{definition of volume of divisors}, we can define the \emph{algebraic volume} of $V_\bullet$ as 
	\[\vol(V_\bullet)\coloneq\limsup_{m\to\infty}\frac{\dim_K(V_m)}{m^d/d!}.\]
	Let $D_{\overline{K}}$ be the pull-back of $D$ to $U_{\overline{K}}$. Then the valuation $v_Y$ is defined on $H^0(U_{\overline{K}},D_{\overline{K}})=H^0(U,D)\otimes_{K}\overline{K}$. We denote set the semi-group
	\[\Gamma(V_\bullet)\coloneq \left\{(m,\gamma)\in \N_{\geq 1}\times\Z^{d}\mid \gamma=v_Y(s) \text{ for some $s\in V_m\otimes_K\overline{K}$}\right\}\]
	and $\Gamma(V_\bullet)_m\coloneq \Gamma(V_\bullet)\cap (\{m\}\times\Z^d)$. The associated \emph{Okounkov body of $V_\bullet$} is defined as 
	\[\Delta(V_\bullet)\coloneq \text{ the closure of $\bigcup\limits_{m\in\N_{\geq 1}}\frac{1}{m}\Gamma(V_\bullet)_m$ in $\{1\}\times\R^d$ }.\]
	We set $\Gamma(D)\coloneq \Gamma\left(\{H^0(U,mD)\}_{m\in\N}\right)$ and $\Delta(D)\coloneq \Delta\left(\{H^0(U,mD)\}_{m\in\N}\right)$.
	
	On the other hand, let $\Gamma(V_\bullet)_\Z$ be the subgroup of $\Z^{d+1}$ generated by $\Gamma(V_\bullet)$, and $\Gamma(V_\bullet)_{\Z,0}\coloneq \Gamma(V_\bullet)_\Z\cap (\{0\}\times\Z^d)$. Notice that $V_{\overline{K},\bullet}$ is a $\overline{K}$-subalgebra of $\overline{K}[[z_1,\cdots, z_d]]$, we have the Kodaira dimension $\kappa=\dim(\Delta(V_\bullet))\leq d$ of $V_\bullet$, see \cref{remark:kodaira dimension}. Then $\kappa=\rk_\Z(\Gamma(V_\bullet)_{\Z,0})$ by \cite[Proposition~6.3.6~(1)]{chen2020arakelov}. Moreover, by \cite[Proposition~6.3.18]{chen2020arakelov}, 
	\[\lim\limits_{m\in\N(V_\bullet), m\to\infty}\frac{\dim_K(V_m)}{m^\kappa} = \vol_{\R^\kappa}(\Delta(V_\bullet))>0,\]
	where $\N(V_\bullet)\coloneq \{m\in\N\mid \dim_K(V_m)\not=0\}$.
	In particular, for $D\in \widetilde{\Div}_{\Q}(U)_\cpt$, $D$ is big if and only if the Kodaira dimension of $\{H^0(U,mD)\}_{m\in\N}$ is $d$.
\end{art}

We state our next lemma which will be repeatedly used throughout this paper.

\begin{lemma}
	\label{lemma:refokoun}
	Let $V_\bullet=\{V_m\}_{m\in\N}$ be a graded linear series of $D\in \widetilde\Div_\Q(U)_\cpt$. Suppose $V_\bullet$ contains an ample series. Then $\Gamma(V_\bullet)_\Z=\Z^{d+1}$. 
	In particular, we have
	\begin{align} \label{algebraic volume and geometric volume}
		{\vol(V_\bullet)} =\lim\limits_{m\to\infty}\frac{\dim_K(V_m)}{m^d/d!}={d!}\cdot{\vol_{\mathbb{R}^d}(\Delta(V_\bullet))}={d!}\cdot{\vol_{\mathbb{R}^d}(\Delta(V_\bullet)^{\circ})}
	\end{align}
	where $S^{\circ}$ denotes the topological interior of any subset $S\subseteq \mathbb{R}^d$.
\end{lemma}
\begin{proof}
	
We choose a projective $K$-model $X$ of $U$ and an ample divisor $A$ as in \cref{def:ampleser} which can be furthermore chosen to be very ample. As argued in the proof of \cite[Lemma 2.2]{lazarsfeld2009convex}, there exist sections $s_0, s_1,\cdots, s_d$ in $H^0(X,A)$, hence in $H^0(U,A)$, with $v_Y(s_0)=0$ and $v_Y(s_i)=e_i$ for $1\leq i\leq d$, where $\{e_i\}_{1\leq i\leq d}$ is the standard basis of $\mathbb{R}^d$ and $0$ is the zero vector. Choose $m$  such that there is a sequence of inclusions
	\[H^0(U,A)\xhookrightarrow{\otimes t_0}V_m\xhookrightarrow{}H^0(U,mD)\]
	for a non-zero $t_0\in H^0(U,mD-A)$. Write $f_0\coloneq v_Y(t_0)$. Since $V_m\neq \{0\}$ for large enough $m$, we can choose $m$ above such that there is a non-zero section $t_1\in V_{m+1}$. Write $f_1\coloneq v(t_1)$. Then clearly the vectors $(m,f_0)$, $(m,f_0+e_i)$ and $(m+1,f_1)$ are all in $\Gamma(V_\bullet)$. 
	Hence $\Gamma(V_\bullet)_{\Z}=\Z^{d+1}$, $\Gamma(V_\bullet)_{\Z,0}\simeq\Z^{d}$. This implies that the Kodaira dimension of $V_\bullet$ is $d$, so \[\lim\limits_{m\to\infty}\frac{\dim_K(V_m)}{m^d}=\lim\limits_{m\in\N(V_\bullet), m\to\infty}\frac{\dim_K(V_m)}{m^d}=\vol_{\mathbb{R}^d}(\Delta(V_\bullet)),\]
	the first equality uses \cref{def:ampleser}~\ref{ampleser a}. Hence \eqref{algebraic volume and geometric volume} holds.
\end{proof}

We end this subsection by noting a result on the behavior of the Okounkov bodies under perturbation.

\begin{art}\label{convergence of concave bodies}
A \emph{convex body} in $\R^d$ is a compact concave subset in $\R^d$ with non-empty interior. For example, by \cref{okounkov body of a graded linear series}, for $D\in \widetilde{\Div}_\Q(U)$, $D$ is big if and only if $\Delta(D)$ is a convex body in $\R^d$ (i.e. the corresponding Kodaira dimension is $d$). We denote by $\mathscr{K}^d$ the set of concave bodies in $\R^d$. 

Let $\Delta_1, \Delta_2$ be compact concave subsets in $\R^d$ (not necessarily with interior points). We define the \emph{Hausdorff distance} between $\Delta_1, \Delta_2$ as
\[d_H(\Delta_1, \Delta_2)\coloneq \inf\{\varepsilon\in \R_{\geq0}\mid \Delta_1\subset\Delta_2+\varepsilon\mathbb{B}, \Delta_2\subset\Delta_1+\varepsilon\mathbb{B}\}\]
where $\mathbb{B}$ is the unit ball in $\R^d$. We also set
\[d_S(\Delta_1,\Delta_2)\coloneq \vol_{\R^d}((\Delta_1\cup\Delta_2)\setminus(\Delta_1\cap\Delta_2)).\]
Then $d_H$ and $d_S$ are metrics, called the \emph{Hausdorff metric} and the \emph{symmetric difference metric}, respectively, on $\mathscr{K}^d$. 
 Let  $(\Delta_j)_{j\in\N_{\geq 1}}\subset \mathscr{K}^d$ and $\Delta\in \mathscr{K}^d$. By \cite[Theorem~7]{shephard1965metrics}, $\lim\limits_{j\to\infty}d_H(\Delta_j,\Delta)=0$ if and only if $\lim\limits_{j\to\infty}d_S(\Delta_j,\Delta)=0$.
\end{art}
\begin{lemma}\label{lemma:interior point convergence of concave}
	Let $(\Delta_j)_{j\in \N_{\geq 1}}$ be a sequence of compact concave bodies in $\R^d$ converging to $\Delta$, i.e. $\lim\limits_{j\to\infty}d_H(\Delta_j,\Delta)=0$. Then for any $\lambda\in \Delta^\circ$, we have $\lambda\in \Delta^\circ_j$ when $j$ is large enough. 
\end{lemma}
\begin{proof}
	Let $\lambda\in \Delta^\circ$ and $r\in\R_{>0}$ such that $B\coloneq\{x\in \R^d\mid |x-\lambda|\leq r\}\subset \Delta$. We consider $B_j\coloneq B\cap \Delta_j$. Since $B\setminus B_j \subset (\Delta\cup\Delta_j)\setminus (\Delta\cap \Delta_j)$ and $\lim\limits_{j\to\infty}d_S(\Delta_j,\Delta)=0$ by \cref{convergence of concave bodies}, we have that $\lim\limits_{j\to\infty}d_S(B_j,B)=0$, i.e. $B_j$ converge to $B$. Since $B_j\subset B$, we have that
	\begin{align}\label{limit of convex body}
\lim\limits_{j\to\infty}\vol_{\R^d}(B_j)=\vol_{\R^d}(B)
	\end{align}
	and $B_j$ is a convex body for $j$ large enough. To show $\lambda\in \Delta_j^\circ$ for $j$ large, it suffices to show that $\lambda\in B_j^\circ$ for $j$ large. If $\lambda\not\in B_j^\circ$, since $B_j$ is a convex body,
	\begin{itemize}
		\item there is a hyperplane in $\R^d$ separating $\lambda$ and $B_j$ when $\lambda\not\in B_j$;
		\item there is a supporting hyperplane of $B_j$ containing $\lambda$ when $\lambda$ is in the boundary of $B_j$ by \cite[Theorem~11.6]{rockafellar1970convex}.
	\end{itemize} 
	In either case, $B_j$ is contained in a semi-ball of $B$. By \eqref{limit of convex body}, we know that $\lambda\in B_j^\circ$ for $j$ large enough. This completes the proof of the lemma.
\end{proof}
\begin{lemma}
	\label{lemma:convbodagain}
	Let $(D_j)_{j\in\N_{\geq 1}}$ be a Cauchy sequence in $\widetilde{\Div}_\Q(U)_\cpt$ converging to $D$. Assume that $D$ is big. Then 	
	\[\lim\limits_{j\to \infty}d_H(\Delta(D_j),\Delta(D))=0,\]
	where $d_H(\cdot,\cdot)$ is the usual Hausdorff metric on the space of compact convex bodies.
	In particular, if $(D_j)_{j\in\N_{\geq 1}}$ is decreasing (resp. increasing) defining $D$, then 
	\[\Delta(D)=\bigcap\limits_{j=1}^\infty\Delta(D_j)\]
	\[\text{(resp. $\Delta(D)^\circ=\bigcup_{j=1}^\infty\Delta(D_j)^{\circ}$).}\]
\end{lemma}
\begin{proof}
	The first claim is a consequence of \cite[Corollary 1.19]{nijeribanra}. For the second claim, if $D_j$ is decreasing, then $\Delta(D)\subset \bigcap\limits_{j=1}^\infty\Delta(D_j)$. By \cref{convergence of concave bodies} $\lim\limits_{j\to \infty}d_H(\Delta(D_j),\Delta(D))=0$ implies that \[\lim\limits_{j\to \infty}d_S(\Delta(D_j),\Delta(D)) = \lim\limits_{j\to \infty}\vol_{\R^d}(\Delta(D_j))-\vol_{\R^d}(\Delta(D))=0,\] so $\vol_{\R^d}\left(\bigcap\limits_{j=1}^\infty\Delta(D_j)\right)=\vol_{\R^d}(\Delta(D))$. This implies that  $\Delta(D)=\bigcap\limits_{j=1}^\infty\Delta(D_j)$ since they are both concave bodies. If  $D_j$ is increasing, then $\bigcup\limits_{j=1}^\infty\Delta(D_j)^{\circ}\subset \Delta(D)^\circ$. By \cref{lemma:interior point convergence of concave}, the equality holds. 
\end{proof}

\section{Arithmetic volumes and concave transforms of compactified $S$-metrized divisors}
\label{section:arithmetic volume and concave transforms}
In this section, we will define the arithmetic volume, arithmetic $\chi$-volume and concave transform of a compactified $S$-metrized divisor, and use these notion to study essential minimal and give a height inequality. Throughout this section, we fix an adelic curve $S=(K,\Omega,\mathcal{A},\nu)$ satisfying the condition \ref{finite measure for archimedean} \ref{exists finite measure for subset} \ref{K countable or A discrete} in \S \ref{subsection:Global theory: Compactified (model) divisors} and the tensorial minimal slope property of level $\ge C_0$ for some $C_0\in\R_{\geq0}$ (see \cref{def:tensorialproperty}). We also fix a $d$-dimensional normal quasi-projective variety $U$ over $K$. We 
	fix an admissible flag $Y_\bullet$ of $U_{\overline{K}}$ with $Y_d=p$, see \cref{flag of a quasi-projective variety} for the definition of admissible flag, then we have the algebra $\widehat{\OO_{U_{\overline{K}},p}}=\overline{K}[[z_1,\cdots, z_d]]$. By a graded $K$-algebra of adelic vector bundle, we mean a graded $K$-algebra of adelic vector bundle with respect to $R\coloneq\Frac(\overline{K}[[z_1,\cdots, z_d]])$ in the sense of \cref{def:gradedalgerba}.

\subsection{Auxiliary graded linear series for compactified $S$-metrized divisors}
In this subsection, we associate a graded $K$-algebra of adelic vector bundles to any compactified $S$-metrized divisor. We need to show that this definition actually gives us a graded $K$-algebra of adelic vector bundles. We derive the necessary results from the corresponding facts in the projective case.
\begin{art} \label{sup-norm} 
	For any compactified $S$-metrized divisor $\overline{D}=(D,g)\in \widehat{\Div}_{S,\Q}(U)_\cpt$ introduced in \cref{def:CM divisors on non-proper} and any $\omega\in \Omega$, we can define the \emph{sup-norm} on $H^0(U_\omega,D_\omega)$ as follows: for any $s\in H^0(U_\omega,D_\omega)$, set
	\[\log\|s\|_{\sup,\omega}\coloneq\sup\{(-g_\omega+\log|s|_\omega)(x)\mid x\in U_\omega^\an\}.\]
	Note that by the above definition, $\metr_{\sup,\omega}$ can take $\infty$ as a value and hence it is not an honest norm. 
	Our idea will be to look at those sections in $H^0(U,D)$ which satisfies an appropriate dominancy.
\end{art}

\begin{definition}
	\label{def:auxglobsec}
	Let $\overline{D}\in\widehat{\Div}_{S,\Q}(U)_\cpt$. The space of \emph{auxiliary global sections} as 
	\[{H^0_+(U,\overline{D})}\coloneq\{s\in H^0(U,D)\mid \text{{$\|s\|_{\sup,\omega}<\infty$ for any $\omega\in \Omega$,}} \tint_{\Omega} \log\|s\|_{\sup,\omega} \,\nu(d\omega)<\infty\}.\]
	With the above definition we associate an \emph{auxiliary graded $K$-algebra} to $\overline{D}$ as \break $\{H^0_+(U,m\overline{D})\otimes_K\overline{K}\}_{m\in\N}$ 
	which is a $\overline{K}$-subalgebra of $R[T]$, see \cref{flag of a quasi-projective variety}. 
	Each of the graded pieces $H^0_+(U,m\overline{D})$ equipped with the sup-norm $\metr_{\sup,m,\omega}$ for each $\omega\in\Omega$ becomes an $S$-normed $K$-vector space.
\end{definition}
\begin{remark} \label{remark:auxiliariy for divisors from projective model}
	Let $\overline{D}\in \widehat{\Div}_{S,\Q}(U)_\CM$, and $X$ a projective $K$-model of $U$ such that $\overline{D}\in \widehat{\Div}_{S,\Q}(X)$. If $X$ is integrally closed in $U$, then $H^0_+(U, \overline{D})=H^0(U,D)=H^0(X,D)$. Indeed, the second equality is from \cref{geometric volume}. For the first equality, notice that for any $\omega\in\Omega$, $U^\an_\omega$ is dense in $X^\an_\omega$, then the sup-norm on $H^0(U,D)$ is exactly the one on $H^0(X,D)$ via the identity, so $\|s\|_{\sup,\omega}<\infty$ for any $s\in H^0(U,D)$. Moreover, by \cite[Theorem~6.1.13]{chen2020arakelov}, we have that $\tint_{\Omega} \log\|s\|_{\sup,\omega} \,\nu(d\omega)<\infty$ for any $s\in H^0(X,D)$. This proves our claim.
\end{remark}
We want to show next that $\{(H^0_+(U,m\overline{D}),\metr_{\sup,m})\}_{m\in\N}$ is a graded $K$-algebra of adelic vector bundles, so that we can study their arithmetic volumes. The main task is to prove that $(H^0_+(U,m\overline{D}),\metr_{\sup,m})$ is an adelic vector bundle (see \cref{def:adelicvectorbundle}) for each $m\in\N$ and we will use approximating projective models to deduce that. We need the following lemma. 
\begin{lemma}
	\label{lemma:convergences of sup-norms}
	Let $\overline{D}\in\widehat{\Div}_{S,\Q}(U)_\cpt$, and $(\overline{D_{j}})_{j\in \N_{\geq 1}}\subset\widehat{\Div}_{S,\Q}(U)_\cpt$ a sequence decreasingly converging to $\overline{D}$ with respect to the $\overline{B}$-boundary topology for some weak boundary divisor $\overline{B}$ of $U$. Then for every $s\in H^0_+(U,\overline{D})\subset  H^0_+(U,\overline{D_j})$ and $\omega\in\Omega$, we have that
	\[\lim\limits_{j\to\infty}\log\|s\|_{j,\sup,\omega}=\log\|s\|_{\sup,\omega},\]
	where $\metr_{\sup,\omega}$ (resp. $\metr_{j,\sup,\omega}$) is the sup-norm on $H^0_+(U,\overline{D})$ (resp. $H^0_+(U,\overline{D_j})$) at $\omega$.
\end{lemma}
\begin{proof}
	Write $\overline{D}=(D,g)$ and $\overline{D_j}=(D_j,g_j)$ with $D, D_j\in \widetilde{\Div}_\Q(U)_\cpt$. Let $s\in H^0_+(U,\overline{D})\subseteq H^0_+(U,\overline{D_j})$ and $\omega\in\Omega$. The effectivity relation $\overline{D_j}\ge \overline{D}$ implies that for any $x\in U^\an_\omega$, we have
	\[-g_{j,\omega}(x)+|s(x)|_\omega\leq -g_\omega(x)+|s(x)|_\omega,\]
	which implies $\log\|s\|_{j,\sup,\omega}\leq \log\|s\|_{\sup,\omega}$. Similarly, the monotonicity of the Cauchy sequence clearly implies that $\log\|s\|_{j,\sup,\omega}$ is increasing. Then $\lim\limits_{j\to\infty}\log\|s\|_{j,\sup,\omega}$ exists and \begin{equation}\label{eq:limits of sup-norm}
		\lim\limits_{j\to\infty}\log\|s\|_{j,\sup,\omega}\leq \log\|s\|_{\sup,\omega}.
	\end{equation} On the other hand, let $\varepsilon\in\R_{>0}$. By definition, there is a point $x\in U_\omega^\an$ such that \[-g_\omega(x)+\log|s(x)|_\omega\ge \log\|s\|_{\sup,\omega}-\varepsilon.\] 
	Since $\overline{D_j}$ converges to $\overline{D}$ with respect to the $\overline{B}$-boundary topology for some weak boundary divisor $\overline{B}$ of $U$, we have that \begin{align*}
		\lim\limits_{j\to\infty}\log\|s\|_{j,\sup,\omega}&\ge\lim\limits_{j\to\infty}(-g_{j,\omega}(x)+\log|s(x)|_\omega)\\
		&=-g_\omega(x)+\log|s(x)|_\omega\\
		&\ge \log\|s\|_{\sup,\omega}-\varepsilon.
	\end{align*}
	Since $\varepsilon$ is arbitrary, we have that $\lim\limits_{j\to\infty}\log\|s\|_{j,\sup,\omega} \geq \log\|s\|_{\sup,\omega}$. Combining this with \eqref{eq:limits of sup-norm}, we complete the proof.
\end{proof}
\begin{proposition} \label{H+ is adelic}
	Let $\overline{D}\in \widehat{\Div}_{S,\Q}(U)_\cpt$. Then $(H^0_+(U,\overline{D}),\metr_{\sup})$ is an adelic vector bundle. Moreover, the graded $K$-algebra $\{(H^0_+(U,m\overline{D}),\metr_{\sup,m})\}_{m\in\N}$ equipped with the sup-norms is a graded $K$-algebra of adelic vector bundles in the sense of \cref{def:gradedalgerba}.
\end{proposition}
\begin{proof}
	We prove the proposition in the following steps.
	
	\vspace{2mm} \noindent
	Step 1: {\it The $S$-normed space $(H^0_+(U,\overline{D}),\metr_{\sup})$ is measurable.}
	
	\vspace{2mm} \noindent
	Let $(\overline{D_{j}})_{j\in \N_{\geq 1}}$ be a  sequence in $\widehat{\Div}_{S,\Q}(U)_\CM$ decreasingly converging to $\overline{D}$ with respect to the $\overline{B}$-boundary topology for some weak boundary divisor $\overline{B}$ of $U$. Assume that $\overline{D_j}\in \widehat{\Div}_{S,\Q}(X_j)$ for some projective $K$-model of $U$ which is integrally closed in $U$. Then $H_+^0(U,\overline{D})=H^0(U,\overline{D_j})=H^0(X_j,D_j)$ by Remark~\ref{remark:auxiliariy for divisors from projective model}. Let $\metr_{j,\sup,\omega}$ the sup-norm on $H^0(X,D_j)$ at $\omega\in\Omega$. By \cite[Proposition~6.1.26]{chen2020arakelov}, for any $j\in \N_{\geq 1}$ and $s\in H^0_+(U,\overline{D})\subset H^0(X,{D}_j)$, the function $\log\|s\|_{j,\sup,\omega}$ is $\mathcal{A}$-measurable. By \cref{lemma:convergences of sup-norms}, this implies that $\log|s|_{\sup,\omega}$ is $\mathcal{A}$-measurable. Hence $(H^0_+(U,\overline{D}),\metr_{\sup})$ is measurable.
	
	\vspace{2mm} \noindent
	Step 2: {\it The $S$-normed space $(H^0_+(U,\overline{D}),\metr_{\sup})$ is dominated.}
	
	\vspace{2mm} \noindent
	Let $\overline{V}^\vee=(V^\vee, \metr^\vee)$ be the dual of $(H^0_+(U,\overline{D}),\metr_{\sup})$. By definition of $H^0_+(U,\overline{D})$ in \cref{def:auxglobsec}, we have that 
	\[\tint_{\Omega}\log\|s\|_{\sup,\omega} \,\nu(d\omega)<\infty\]
	for any $s\in H^0_+(U,\overline{D})$. It remains to show that 
	\[\tint_{\Omega}\log\|t\|_{\omega}^\vee \,\nu(d\omega)<\infty\]
	for any $t\in V^\vee$. Let $\overline{D'}\in \widehat{\Div}_{S,\Q}(X)$ for some projective $K$-model $X$ of $U$ such that $X$ is integrally closed in $U$ and $\overline{D'}\ge \overline{D}$ in $\widehat{\Div}_{S,\Q}(U)_\CM$. By Remark~\ref{remark:auxiliariy for divisors from projective model}, we have that $H^0_+(U,\overline{D'})=H^0(X,D')$. We denote 
	${\overline{V'}}^\vee=({V'}^\vee, {\metr'}^\vee)$ the dual of $H^0(X,D')$ equipped with the sup-norm $\metr_{\sup}'$. 
	Since $\overline{D}\le \overline{D'}$, we have a canonical injective map of $K$-vector spaces \[f\colon H^0_+(U,\overline{D})\rightarrow H^0_+(U,\overline{D'})=H^0(X,D')\] 
	which induces a contractive map  
	\[f_{\omega}\colon (H^0_+(U,\overline{D})\otimes_K K_\omega,\metr_{\sup,\omega})\rightarrow (H^0(X,D')\otimes_K K_\omega,\metr_{\sup,\omega}'),\] 
	i.e. the operator norm $\|f_\omega\|\leq 1$, for any $\omega\in\Omega$. For the dual map
	\[f_{\omega}^\vee\colon ({V'}^\vee\otimes_KK_\omega, {\metr'}^\vee_\omega)\to (V^\vee\otimes_KK_\omega, \metr^\vee_\omega)\]
	which is surjective by \cite[Proposition~1.1.22]{chen2020arakelov}, we have that $\|f_\omega^\vee\|\leq \|f_\omega\|\leq 1$. Then for any $t\in V^\vee$, there is $t'\in{V'}^\vee$ such that $t=f^\vee(t')$, and we have that
	\[\log\|t\|_\omega^\vee=\log\|f^\vee(t')\|_\omega^\vee\leq \log{\|t'\|'}^\vee_\omega\]
	for any $\omega\in\Omega$.
	Note that $(H^0(X,D'),\metr_{\sup}')$ is dominated by \cite[Theorem~6.1.13]{chen2020arakelov}, then
	\begin{equation*}
		\label{equa:modeldominance}
		\tint_{\Omega} \log\|t\|_\omega^\vee\, \nu(d\omega)\leq \tint_{\Omega} \log{\|t'\|'}_\omega^\vee\, \nu(d\omega)<\infty.   
	\end{equation*}
	This completes the proof of Step~2.
	
	\vspace{2mm} \noindent
	Step 3: {\it The graded $K$-algebra $\{(H^0_+(U,m\overline{D}),\metr_{\sup,m})\}_{m\in\N}$ is a graded $K$-algebra of adelic vector bundle.}
	
	\vspace{2mm} \noindent
	For any $m\in\N$, by Step~1 and Step~2 the $S$-normed space $(H^0_+(U,m\overline{D}),\metr_{\sup,m})$ is an adelic vector bundle on $S$. It remains to verify the three conditions in \cref{def:gradedalgerba}. To verify condition~\ref{gradedalgebra i}, note that we can choose a projective model $X$ of $U$ and a very ample divisor $A$ in $X$ such that $X$ is integrally closed in $U$ and $A\ge D$ in $\widetilde{\Div}_\Q(U)_{\cpt}$, then there are inclusions 
	\[H^0_+(U,m\overline D)\xhookrightarrow{} H^0(U,mD)\xhookrightarrow{} S_m\coloneq H^0(U,mA)=H^0(X,mA).\]
	Since $A$ is very ample on $X$, Serre's theorem implies that $\bigoplus\limits_{m\in\N} S_mT^m$ is a $K$-algebra of finite type which clearly shows condition~\ref{gradedalgebra i} ($\bigoplus\limits_{m\in\N} (S_m\otimes_{K}\overline{K})T^m$ is $\overline{K}$-subalgebra of $R[T]$ by restricting at $p$). Condition~\ref{gradedalgebra ii} is clear and condition~\ref{gradedalgebra iii} is clearly satisfied for the sup-norms $\metr_{\sup,\omega}$ due to how tensor product of two metrics are defined. This finishes the proof of Step~3.
\end{proof}
\subsection{Arithmetic volumes of compactified $S$-metrized divisors}
\label{subsection:arithmetic volume of cpt}
In this subsection, we define the arithmetic volumes for compactified $S$-metrized divisors. We first consider the geometric volume. For $\overline{D}=(D,g)\in\widehat{\Div}_{S,\Q}(U)_\cpt$, we use the graded linear series $\{H^0(U,mD)\}_{m\in\N}$ to define the geometric volume $\vol(D)$ and the Okounkov body $\Delta(D)$ in \cref{okounkov body of a graded linear series} (see also \cref{def:volume}). On the other hand, we have the graded linear series $\{H_+^0(U,m\overline{D})\}_{m\in\N}$ defined in \cref{def:auxglobsec}, we denote the corresponding volume and Okounkov body by $\vol_+(\overline{D})$ and $\Delta_+(\overline{D})$ (see \cref{okounkov body of a graded linear series}), respectively. We have the following result.
\begin{lemma}
	\label{lemma:gooddef}
	Let $\overline{D}=(D,g)\in\widehat{\Div}_{S,\Q}(U)_\cpt$ with $D$ big. Then the graded linear series $\{H^0_+(U,m\overline{D})\}_{m\in\N}$ contains an ample series (see \cref{def:ampleser}). Furthermore, we have that
	\[\vol_+(\overline{D})={\vol}(D)\]
	In particular, we have that
	\[\Delta_+(\overline{D})^\circ=\Delta(D)^\circ.\]
\end{lemma}
\begin{proof}
	Let $\overline{B}=(B,g_B)$ be a weak boundary divisor of $U$, $(\overline{D_j})_{j\in\N_{\geq 1}}=(D_j,g_j)_{j\in\N_{\geq 1}}\subset \widehat{\Div}_{S,\Q}(U)_\CM$ and $(\varepsilon_j)_{j\in\N_{\geq 1}}\subset \Q_{>0}$ such that $\lim\limits_{j\to\infty}\varepsilon_j=0$ and 
	\[-\varepsilon_j\overline{B}\leq \overline{D}-\overline{D_j}\leq \varepsilon_j\overline{B}.\]
 For any $j\in\N_{\geq 1}$, let $X_j$ be a projective model of $U$ such that $\overline{D_j}\in\widehat{\Div}_{S,\Q}(X_j)$. We claim that there is an inclusion 
	\begin{align}\label{containment of sections}
		H^0(U,mD_j')\xhookrightarrow{}H^0_+(U,m\overline{D})\subseteq H^0(U,mD)
	\end{align}
	for all $j$ and $m$ where $\overline{D_j'}=(D_j',g'_j)\coloneq\overline{D_j}-\varepsilon_j\overline{B}$. If the claim holds, then by \cite[Theorem~5.2.1]{yuan2021adelic} we can deduce that $D_j'$ is big for large enough $j$ since $D$ is assumed to be big. We fix such a $j_0$. Since $D_{j_0}'$ is big, we have that $H^0(X_{j_0}, mD_{j_0}')\not=0$ for large $m$, so $H^0(U,m\overline{D})\not=0$. By Kodaira lemma (\cite[Proposition~2.2.6]{lazarsfeld2017positivity}), we can find a positive integer $n$ and an ample divisor $A$ on $X_{j_0}$ such that $nD_{j_0}'-A\ge 0$ which in turn will induce injections 
	\[H^0(U,mA)\xhookrightarrow{\otimes s^{\otimes m}} H^0(U,mnD_{j_0}')\xhookrightarrow{}H^0_+(U,mn\overline{D})\]
	for a non-zero section $s\in H^0(U,nD_{j_0}'-A)$. This clearly shows that $\{H^0_+(U,n\overline{D})\}_{n\in\N}$ contains an ample series. Moreover from \eqref{containment of sections}, we have inequalities 
	\[\vol(D_j')\le\vol_+(\overline{D})\le \vol(D)\]
	which also shows the equality of volumes by \cref{geometric volume}\ref{continuity of geometric volume}. Then we can deduce the claim about the interiors by noting that two compact convex subsets of $\mathbb{R}^d$, one contained in another, can have the same volume if and only if their difference has empty interior. Hence it is enough to show our claim. To show the claim, for any $j\in\N_{\geq1}$, note that the effectivity $\overline{D_j'}\le \overline{D}$ induces an inequality 
	\begin{equation}
		\label{equa:mukhene}
		-g_{\omega}+\log|s|_\omega\leq -g'_{j,\omega}+\log|s|_\omega
	\end{equation}
	on $U_\omega^{\an}$ for all $s\in H^0(U,D_j')$ and $\omega\in\Omega$. Furthermore since $\overline{D_j'}\in\widehat{\Div}_{S,\Q}(X_j)$, by \cite[Theorem~6.1.13]{chen2020arakelov} we have that
	\[\log\|s\|'_{j,\sup,\omega}\coloneq \sup\{-g_{j,\omega}'(x)+\log|s(x)|_\omega\mid x\in X_{j,\omega}^\an\}<\infty,\]
	\[\tint_{\Omega}\log \|s\|'_{j,\sup,\omega}\, \nu(d\omega)<\infty\]
	for all $s\in H^0(U,D_j')$. Now we can easily deduce the claim noting the inequality \eqref{equa:mukhene} and \cite[Proposition~A.4.2~(2)]{chen2020arakelov}.
\end{proof}

Recall the arithmetic $\chi$-volume $\widehat{\vol}_\chi(\overline{E_\bullet})$ and arithmetic volume $\widehat{\vol}(\overline{E_\bullet})$ of a graded algebra of adelic vector bundles $\overline{E_\bullet}=\{\overline{E_m}\}_{m\in\N}$ in \cref{def:arithvol}.

\begin{definition}
	\label{def:volume}
	Let $\overline{D}\in\widehat{\Div}_{S,\Q}(U)_\cpt$. We call   
	\[\overline{V_{\overline{D},\bullet}}=\left\{\overline{V_{\overline{D},m}}\right\}_{m\in\N}\coloneq\left\{(H^0_+(U,m\overline{D}),\metr_{\sup,m})\right\}_{m\in\N}\] the \emph{graded $K$-algebra of adelic vector bundles associated to $\overline{D}$} and define the \emph{arithmetic volume} and  \emph{arithmetic $\chi$-volume} of $\overline{D}$ as
	\[\widehat{\vol}(\overline{D})\coloneq \widehat{\vol}(\overline{V_{\overline{D},\bullet}})=\limsup_{m\to\infty} \frac{\widehat{\deg_+}(\overline{V_{\overline{D},m}})}{m^{d+1}/(d+1)!}\]    
	and 
	\[\widehat{\vol}_\chi(\overline{D})\coloneq \widehat{\vol}_\chi(\overline{V_{\overline{D},\bullet}})=\limsup_{m\to\infty} \frac{\widehat{\deg}(\overline{V_{\overline{D},m}})}{m^{d+1}/(d+1)!},\]
	respectively.
	We furthermore set 
	\[\widehat{\mu}^{\mathrm{asy}}_{\max}(\overline{D})\coloneq \widehat{\mu}^{\mathrm{asy}}_{\max}(\overline{V_{\overline{D},\bullet}})\ \text{ and }\ \widehat{\mu}^{\mathrm{asy}}_{\min}(\overline{D})\coloneq \widehat{\mu}^{\mathrm{asy}}_{\min}(\overline{V_{\overline{D},\bullet}}).\]
	We say that $\overline{D}$ is \emph{big} if $\widehat{\vol}(\overline{D})>0$.
\end{definition}

\begin{remark}
	Even if we can define the arithmetic volume of any compactified $S$-metrized divisor with any assumption on the generic fiber thanks to the formalism of graded $K$-algebras, we will mostly stick to the case where $D$ is big. In that case \cref{lemma:gooddef} shows that the graded $K$-algebra $\{H_+^0(U,\overline{D})\}_{m\in\N}$ has Kodaira dimension $d$. Furthermore the construction of Okounkov bodies in \cite{nijeribanra} coincides with the more general construction sketched in \cite[Section 6.3.2]{chen2020arakelov}. The main difference is that in the big case, all the Okounkov bodies live in the same ambient Euclidean space $\R^d$ even if we perturb it in small enough directions whereas the Okounkov bodies in \cite{chen2020arakelov} may live in different ambient spaces of different dimension. Hence we almost always stick to the big case to be able to use properties shown in \cite{nijeribanra}.
\end{remark}
\begin{remark}\label{remark:compare asy min of divisors}
	When $U=X$ is projective, and $\overline{D}=(D,g)\in N_{S,\Q}(X)$ a relatively nef element, recall $N_{S,\Q}(X)$ in \cref{global adelic Green functions for proper varieties}, then Chen-Moriwaki also define the asymptotic minimal slope for $\overline{D}$ in \cite[Definition~6.4.2]{chen2022hilbert}. 
	Since $X$ is normal, if $K$ is perfect, 
	$D$ is ample and with $\Z$-coefficients, then our definition above is the same as their definition in \cite[Definition~6.4.2]{chen2022hilbert}.
\end{remark}
\begin{lemma}
	\label{lemma:slopebounded}
	Let $\overline{D}\in\widehat{\Div}_{S,\Q}(U)_\cpt$.  Then \[\widehat{\mu}^{\mathrm{asy}}_{\max}(\overline{D})<\infty.\]
\end{lemma}
\begin{proof}
	The idea will be to deduce it from {the projective case} as usual. Let $\overline{D'}=(D',g)\in \widehat{\Div}_{S,\Q}(U)_\CM$  such that $\overline{D'}\ge \overline{D}$. For any $m\in\N$, write $\overline{V_m}\coloneq (H^0(U,mD),\metr_{m,\sup})$ (resp. $\overline{V_m'}\coloneq (H^0(U,mD'),\metr_{m,\sup}')$), where $\metr_{m,\sup}$ (resp. $\metr_{m,\sup}'$) is the corresponding sup-norm. Assume that $\overline{D'}\in\widehat{\Div}_{S,\Q}(X)$ for some projective $K$-model of $U$ which is integrally closed in $U$. Then $H^0_+(U,m\overline{D'})=H^0(X,mD')$ for any $m\in\N$ by Remark~\ref{remark:auxiliariy for divisors from projective model}. Moreover, \cite[Proposition~6.4.4]{chen2020arakelov} and \cite[Proposition~6.2.7]{chen2020arakelov} together show that $\widehat{\mu}^{\mathrm{asy}}_{\max}(\overline{D'})<\infty$. Furthermore the effectivity relation $\overline{D'}\ge \overline{D}$ implies that there is an injective $K$-linear map 
	\[f_m\colon H^0_+(U,m\overline{D})\xhookrightarrow{} H^0_+(U,m\overline{D'})\]
	which is furthermore norm-contractive at every place $\omega$, i.e. $\|f_m\|_\omega\le 1$ for all $\omega\in\Omega$ and $m\in\N$. Consequently we have that $h(f_m)\le 0$ for all $m\in\N$. Then \cite[Proposition~4.3.31~(1)]{chen2020arakelov} shows that 
	\[\widehat{\mu}_{\max}(\overline{V_m})\le \widehat{\mu}_{\max}(\overline{V_m'})\]
	for any $m\in\N$ which clearly implies 
	\[\widehat{\mu}^{\mathrm{asy}}_{\max}(\overline{D})\le \widehat{\mu}^{\mathrm{asy}}_{\max}(\overline{D'})\]
	and finishes the proof.
\end{proof}

\subsection{Concave transforms of compactified $S$-metrized divisors}

\begin{art}\label{construction of concave transform}
Let $\overline{V_\bullet}=\{\overline{V_m}\}_{m\in\N}=\{(V_m,\metr_m)\}_{m\in\N}$ be a graded $K$-algebra of adelic vector bundles over $S$, and  $\mathcal{H}_\bullet=\{\mathcal{H}_m\}_{m\in\N}$ be the the Harder-Narasimhan filtration of $V_\bullet=\{V_m\}_{m\in\N}$ defined in \cref{def:harder filtration}. We assume that
	\[\N(V_\bullet)\coloneq \{n\in\N\mid V_n\not=0\}\not= \{0\}.
	\] 
Write $V_{\overline{K},\bullet} = \{V_{\overline{K},m}\}_{m\in\N}\coloneq \{V_{m}\otimes_K\overline{K}\}_{m\in\N}$ $\mathcal{H}_{\overline{K},\bullet}\coloneq \{\mathcal{H}_m\otimes_K\overline{K}\}_{m\in\N}$, notice that $\mathcal{H}_{\overline{K},\bullet}$ may not be the Harder-Narasimhan filtration of $V_{\overline{K},\bullet}$. However, by \cite[Proposition~6.3.25]{chen2020arakelov}(notice that the canonical extension $S_{\overline{K}}$ of $S$ to $\overline{K}$ (see \cite[\S 3.4.2]{chen2020arakelov}) satisfies the tensorial minimal slope property in sense of \cref{def:tensorialproperty}), $\mathcal{H}_{\overline{K},\bullet}$ is a $\delta$-superadditive $\R$-filtration on $V_{\overline{K},\bullet}$ in the sense of \cite[Definition~6.3.19~(c)]{chen2020arakelov}, where $\delta\colon \N\mapsto \R_{\geq 0}, \ \ n\mapsto C\log(\dim_K(V_n))$ (although $\overline{V_\bullet}$ may not be a graded $K$-algebra of adelic vector bundles in the sense of \cite[Definition~6.4.24]{chen2020arakelov}, but the proof of \cite[Proposition~6.3.25]{chen2020arakelov} works without the assumption that $V_\bullet$ is a $K$-subalgebra of $\Frac(K[[T_1,\cdots, T_d]])[T]$). Notice that $\dim_K(V_n)\leq C_1n^\kappa$ for some $C_1\in \R_{>0}$, where $\kappa$ is the Kodaira dimension of $V_{\bullet}$, we can replace $\delta$ by the map $n\mapsto C_2\log n$ for some suitable $C_2\in\R_{>0}$. It is obvious that $\delta$ is increasing, and 
\begin{align*}
	\sum\limits_{a\in\N}\frac{\delta(2^a)}{2^a}&=C_2\log2\cdot\sum\limits_{a\in\N}\frac{a}{2^a}=C_2\log2<\infty,
\end{align*}
i.e. the condition in \cite[Theorem~6.3.20]{chen2020arakelov} holds (similarly, the condition for $\dim_K(W_n)$ in \cite[Proposition~6.3.28]{chen2020arakelov} automatically holds).
Then we can construct a \emph{concave transform} $G_{\mathcal{H}_\bullet}\colon \Delta(V_{\bullet})^\circ\to \R\cup\{+\infty\}$ associated to $\mathcal{H}_{\bullet}$ as follows (see \cref{okounkov body of a graded linear series} for the definitions of $\Gamma(V_{\bullet})$ and $\Delta(V_{\bullet})$).
 We set
\begin{align*}\label{eq:g function}
	\rho\colon \Gamma(V_{\bullet})&\to \R\cup\{\infty\},\\
	(m,\gamma)&\mapsto \sup\{t\in\R\mid \text{there is $x\in\mathcal{H}_{\overline{K},m}^t(V_{\overline{K},m})$ such that $v_Y(x)=\gamma$}\},
\end{align*} and construct the auxiliary function $\widetilde{\rho}(m,\gamma)\coloneq \limsup\limits_{n\to\infty}\frac{\rho(nm,n\gamma)}{n}$, where $v_Y$ is the valuation defined by the admissible flag $Y_\bullet$, see \cref{flag of a quasi-projective variety}.
We furthermore define 
\[\Gamma^t(V_{\bullet})\coloneq\{(m,\gamma)\in \Gamma(V_{\bullet})\mid \widetilde{\rho}(m,\gamma)\ge mt\},\]
\[\Gamma^t(V_\bullet)_m\coloneq \Gamma^t(V_\bullet)\cap (\{m\}\times\Z^d),\]	\[\Delta(\Gamma^t(V_{\bullet}))\coloneq \text{ the closure of $\bigcup\limits_{m\in\N_{\geq 1}}\frac{1}{m}\Gamma^t(V_\bullet)_m$ in $\{1\}\times\R^d$ }.\]
For any $\lambda\in \Delta(V_\bullet)^\circ$, 
\[G_{\mathcal{H}_\bullet}(\lambda)\coloneq\sup\{t\in\R\mid \lambda\in \Delta(\Gamma^t(V_{\bullet}))\}.\]
Since $G_{\mathcal{H}_\bullet}$ is convex, we have that $G_{\mathcal{H}_\bullet}\equiv +\infty$ or $G_{\mathcal{H}_\bullet}(\Delta(V_\bullet)^\circ)\subset \R$.
On the other hand, for $m\in\N(V_\bullet)$, let $\metr_{\mathcal{H}_{m}}$ (resp. $\metr_{\mathcal{H}_{\overline{K},m}}$) be the norm on $V_m$ (resp. $V_{\overline{K},m}$) associated to $\mathcal{H}_{m}$ (resp.  $\mathcal{H}_{\overline{K},m}$) in \cite[Remark~1.1.40]{chen2020arakelov}, by \cite[Remark~4.3.63]{chen2020arakelov}, we have that
\begin{align}\label{eq:slops for filtrations}
\widehat{\mu}_i(V_{\overline{K},m},\metr_{\mathcal{H}_{\overline{K},m}}) = \widehat{\mu}_i(V_{m},\metr_{\mathcal{H}_{m}}) = \widehat{\mu}_i(\overline{V_m}),
\end{align}
where $\widehat{\mu}_i$ is the $i$-th jumping number of the Harder-Narasimhan filtration defined in \cref{def:harder filtration} (we view $(V_m, \metr_{\mathcal{H}_{m}})$ (resp. $(V_{\overline{K},m},\metr_{\mathcal{H}_{\overline{K},m}})$) as adelic vector bundle over $K$ (resp. $\overline{K}$) with the trivial absolute value). Write $\nu_m$ the Borel probability measure on $\R$ such that for any Borel function $f$ on $\R$, we have that
\[\int_\R f(t)\ \nu_m(dt)\coloneq \frac{1}{\dim_K(V_m)}\sum_{i=1}^{\dim_K(V_m)}f\left(\frac{1}{m}\widehat{\mu}_i(\overline{V_m})\right).\]
From our discussion above and apply \cite[Theorem~6.3.20]{chen2020arakelov} to $\mathcal{H}_{\overline{K},\bullet}$, we have that $(\nu_m)_{m\in\N(V_\bullet)}$ vaguely converge to the zero measure or to the pushforward $\frac{1}{\vol_{\R^\kappa}(\Delta(V_{\bullet}))}G_{\mathcal{H}_\bullet,*}(d\lambda)$, where $d\lambda$ is the standard Lebesgue measure on $\Delta(V_{\bullet})^\circ\subset \R^\kappa$. More precisely, from Step~6 of the proof of \cite[Theorem~6.3.12]{chen2020arakelov}, if $G_{\mathcal{H}_\bullet}\equiv +\infty$, then $(\nu_m)_{m\in\N(V_\bullet)}$ vaguely converge to the zero measure; if $G_{\mathcal{H}_\bullet}(\Delta(V_\bullet)^\circ)\subset \R$, then  $(\nu_m)_{m\in\N(V_\bullet)}$ weakly converge to $\frac{1}{\vol_{\R^\kappa}(\Delta(V_{\bullet}))}G_{\mathcal{H}_\bullet,*}(d\lambda)$. By \cite[Remark~4.3.48, Remark~6.3.21]{chen2020arakelov} and \eqref{eq:slops for filtrations}, if $\widehat{\mu}_{\max}^{\mathrm{asy}}(\overline{V_\bullet})<+\infty$, then
\begin{align}\label{eq:sup of concave transform}
\sup\limits_{\lambda\in \Delta(V_{\bullet})^\circ}G_{\mathcal{H}_\bullet}(\lambda) = \lim\limits_{m\in\N(V_\bullet), m\to+\infty}\frac{1}{n}\widehat{\mu}_1(\overline{V_m})=\widehat{\mu}_{\max}^{\mathrm{asy}}(\overline{V_\bullet})<+\infty,
\end{align}
hence in this case, $(\nu_m)_{m\in\N(V_\bullet)}$ converge weakly to $\frac{1}{\vol_{\R^\kappa}(\Delta(V_{\bullet}))}G_{\mathcal{H}_\bullet,*}(d\lambda)$.

Let $\overline{D}\in\widehat{\Div}_{S,\Q}(U)_\cpt$ such that there is $n\in\N_{\geq 1}$ with $\dim_K(H^0_+(U,m\overline{D}))>0$. We consider the case where $\overline{V_\bullet}$ is the graded $K$-algebra of adelic vector bundles associated to $D$ defined in \cref{def:auxglobsec}, then all above assumptions for $\overline{V_\bullet}$ are satisfied (notice that $\widehat{\mu}_{\max}^{\mathrm{asy}}(\overline{D})<\infty$ by \cref{lemma:slopebounded}). The corresponding concave transform $G_{\overline{D}}\colon \Delta_+(\overline{D})^\circ\to \R$  is called the \emph{concave transform} of $\overline{D}$.
\end{art}


We have the following theorem which generalizes a result of Chen-Moriwaki \cite[Theorem~6.4.6]{chen2020arakelov}.

\begin{theorem}
	\label{theorem:concavemain}
 Let $\overline{D}=(D,g)\in\widehat{\Div}_{S,\Q}(U)_\cpt$, and $\overline{V_\bullet}=\{\overline{V_m}\}_{m\in\N}$ its associated graded $K$-algebra of adelic vector bundles defined in \cref{def:auxglobsec}. Assume that $\N(V_\bullet)\not=0$. Then the associated Borel probability measures $(\nu_m)_{m\in\N(V_\bullet)}$ in \cref{construction of concave transform} converge weakly to the probability measure $\frac{1}{\vol_{\R^\kappa}(\Delta_+(\overline{D}))}G_{\overline{D},*}(d\lambda)$, where  $d\lambda$ is the standard Lebesgue measure on $\Delta(D)\subseteq \R^\kappa$ and $\kappa$ is the Kodaira dimension $V_\bullet$ ($\kappa=d$ if $D$ is big). In particular, if $\overline{D}$ is big, then $D$ is big. Moreover, in the case where $D$ is big (note that $\Delta(D)^\circ=\Delta_+(\overline{D})^\circ$ by \cref{lemma:gooddef}), we have that
 \begin{align}\label{eq:main theorem concave 1}
\widehat{\vol}(\overline{D})=\lim_{ m\to\infty}\frac{\widehat{\deg}_+(\overline{V_m})}{m^{d+1}/(d+1)!}={(d+1)!} \int_{\Delta(D)^{\circ}} \max\{G_{\overline{D}}(\lambda),0\} \, d\lambda.
 \end{align}
and	\begin{align}\label{eq:main theorem concave 2}
\widehat{\vol}_{\chi}(\overline{D})=\limsup\limits_{ m\to\infty}\frac{\widehat{\deg}(\overline{V_m})}{m^{d+1}/(d+1)!}\leq(d+1)!\int_{\Delta(D)^{\circ}} G_{\overline{D}}(\lambda) \, d\lambda.
	\end{align}
with equality if $\widehat{\mu}_{\min}^{\mathrm{asy}}(\overline{D})>-\infty$ (in this case the supremum limit in \eqref{eq:main theorem concave 2} is a limit).  
\end{theorem}
\begin{proof}
     By our discussion in \cref{construction of concave transform}, $(\nu_m)_{m\in\N(V_\bullet)}$ converge weakly to the probability measure $\frac{1}{\vol_{\R^\kappa}(\Delta_+(\overline{D}))}G_{\overline{D},*}(d\lambda)$.  In particular, by \cite[Remark~6.3.27]{chen2020arakelov},	\begin{align}
     	\begin{split}     	\label{eq:concave 2}
     	\frac{1}{\vol_{\R^\kappa}(\Delta_+(\overline{D}))}\int_\R \max\{x,0\}\, ((G_{\overline{D}})_*d\lambda)(dx)=&\frac{1}{\vol_{\R^\kappa}(\Delta_+(\overline{D}))}\int_{\Delta_+(D)^{\circ}} \max\{G_{\overline{D}}(\lambda),0\} \, d\lambda\\
     	=&\lim_{m\in\N(V_\bullet), m\to\infty} \frac{\widehat{\deg}_+(\overline{V_m})}{m\dim_K(V_m)}.
     	\end{split}
     \end{align}
Since $\N(V_\bullet)\not=0$, there are infinitely many $m\in \N$ such that $V_m\not=0$. So 
 \begin{align}\label{eq:volume 1}
\vol_+(\overline{D})=	\limsup\limits_{m\to\infty}\frac{\dim_K(V_m)}{m^d/d!}= 	\limsup\limits_{m\in\N(V_\bullet), m\to\infty}\frac{\dim_K(V_m)}{m^d/d!},
 \end{align}
  \begin{align}
  	\begin{split}\label{eq:arith volume 1}
\widehat{\vol}(\overline{D})=&	\limsup\limits_{m\to\infty}\frac{\widehat{\deg}_+(\overline{V_m})}{m^{d+1}/(d+1)!}= 	\limsup\limits_{m\in\N(V_\bullet), m\to\infty}\frac{\widehat{\deg}_+(\overline{V_m})}{m^{d+1}/(d+1)!}\\
=&(d+1)\cdot\limsup\limits_{m\in\N(V_\bullet), m\to\infty}\left(\frac{\dim_K(V_m)}{m^d/d!}\cdot\frac{\widehat{\deg}_+(\overline{V_m})}{m\dim_K(V_m)}\right)
  	\end{split}
  \end{align}
Combine \eqref{eq:arith volume 1}~\eqref{eq:concave 2}~\eqref{eq:volume 1}, we have that
\begin{align}\label{eq:arith volume 2}
\widehat{\vol}(\overline{D})= (d+1)\cdot \vol_+(\overline{D})\cdot \frac{1}{\vol_{\R^\kappa}(\Delta_+(\overline{D}))}\int_{\Delta_+(D)^{\circ}} \max\{G_{\overline{D}}(\lambda),0\} \, d\lambda.
	\end{align}
      If $\overline{D}$ is big, i.e. $\widehat{\vol}(\overline{D})
      >0$, then by \eqref{eq:arith volume 2}, $\vol_+(\overline{D})>0$ which implies that $D$ is big since $\vol(D)\geq \vol_+(\overline{D})>0$. 
     
     In the case where $D$ is big, by \cref{lemma:gooddef}, $V_\bullet$ contains an ample series, $\vol_+(\overline{D}) =\vol(D)$, $\Delta_+(\overline{D})^\circ=\Delta(D)^\circ$ and $\kappa=d$ (see \cref{okounkov body of a graded linear series}). In particular, $V_m\not=0$ for $m$ large enough, so \eqref{eq:concave 2} becomes
     \begin{align}
     	\begin{split}\label{eq:concave 3}
\frac{1}{\vol_{\R^d}(\Delta_+(\overline{D}))}\int_{\Delta_+(D)^{\circ}} \max\{G_{\overline{D}}(\lambda),0\} \, d\lambda=&\lim_{m\in\N(V_\bullet), m\to\infty} \frac{\widehat{\deg}_+(\overline{V_m})}{m\dim_K(V_m)}\\
=&\lim_{m\to\infty} \frac{\widehat{\deg}_+(\overline{V_m})}{m\dim_K(V_m)}.
     	\end{split}
     \end{align} 
     We can deduce from \cref{lemma:refokoun} that 
	\begin{align}\label{eq:concave 1}
		\vol_+(\overline{D})=\lim\limits_{m\to\infty}\frac{\dim_K(V_m)}{m^d/d!}=d!\cdot\vol_{\R^d}(\Delta_+(\overline{D})).
	\end{align} 
	Combine \eqref{eq:arith volume 2} and  \eqref{eq:concave 1}, we have that
	\begin{align*}
		\widehat{\vol}(\overline{D}) =&(d+1)\cdot \vol_+(\overline{D}) \cdot \frac{1}{\vol_{\R^d}(\Delta_+(\overline{D}))}\int_{\Delta_+(\overline{D})^{\circ}} \max\{G_{\overline{D}}(\lambda),0\} \, d\lambda\\
		=&(d+1)!\int_{\Delta(D)^{\circ}} \max\{G_{\overline{D}}(\lambda),0\} \, d\lambda
		\end{align*}
On the other hand, 
combine \eqref{eq:arith volume 2} \eqref{eq:concave 3} \eqref{eq:concave 1}, we have that 
	\begin{align}\label{eq:arith vol is a limit}
\widehat{\vol}(\overline{D}) = 
(d+1)\cdot\lim\limits_{m\to\infty}\frac{\dim_K(V_m)}{m^d/d!}\cdot\lim_{m\to\infty} \frac{\widehat{\deg}_+(\overline{V_m})}{m\dim_K(V_m)}= \lim\limits_{m\to\infty}\frac{\widehat{\deg}_+(\overline{V_m})}{m^{d+1}/(d+1)!},
	\end{align}
i.e. the supremum limit in the definition of $\widehat{\vol}(\overline{D})$ is actually a limit when $D$ is big.
Hence we have completed the proof of \eqref{eq:main theorem concave 1}.

It remains to consider \eqref{eq:main theorem concave 2}. For any large $m\in\N$ (notice that $m\in\N(V_\bullet)$), $\omega\in\Omega$, we recall the invariant $\delta_\omega(\overline{V_m})$ of $\overline{V_m}$ defined in \cite[Definition~4.3.9]{chen2020arakelov} and set $$\delta(\overline{V_m})\coloneq\int_{\omega\in\Omega}\log\delta_\omega(\overline{V_m})\, \nu(d\omega)$$
as in \cite[Definition~4.3.12]{chen2020arakelov}. Then by \cite[Proposition~4.3.10]{chen2020arakelov}
\[0\leq \delta(\overline{V_m})\leq \frac{1}{2}\dim_K(V_m)\log(\dim_K(V_m))\nu(\Omega_\infty),\]
i.e.
\begin{align*}
0\leq \frac{\delta(\overline{V_m})}{m\dim_K(V_m)}\leq \frac{1}{2}\cdot\frac{\log(\dim_K(V_m))}{m}\cdot\nu(\Omega_\infty).
\end{align*}
By \cref{remark:kodaira dimension}, we have that $\log(\dim_K(V_m))\leq C'\log(m)$ for some $C'\in\R_{>0}$, then
\begin{align}\label{eq:delta}
	0\leq \frac{\delta(\overline{V_m})}{m\dim_K(V_m)}\leq \frac{C'}{2}\cdot\frac{\log(m)}{m}\cdot\nu(\Omega_\infty).
\end{align}
By our assumption that $\nu(\Omega_\infty)<\infty$, after taking the limit on \eqref{eq:delta}, we have that
\begin{align}\label{eq:limit of delta}
\lim\limits_{m\to\infty}\frac{\delta(\overline{V_m})}{m\dim_K(V_m)}=0.
\end{align}
 By \cite[Proposition~4.3.50, Proposition~4.3.51]{chen2020arakelov},
\begin{align*}
\sum\limits_{i=1}^{\dim_K(V_m)}\widehat{\mu}_i(\overline{V_m})\leq \widehat{\deg}(\overline{V_m})\leq \sum\limits_{i=1}^{\dim_K(V_m)}\widehat{\mu}_i(\overline{V_m})+\delta(\overline{V_m}).
\end{align*}
i.e. 
\begin{align}\label{eq:arith volume}
	\frac{\sum_{i=1}^{\dim_K(V_m)}\widehat{\mu}_i(\overline{V_m})}{m\dim_K(V_m)}\leq \frac{\widehat{\deg}(\overline{V_m})}{m\dim_K(V_m)}\leq \frac{\sum_{i=1}^{\dim_K(V_m)}\widehat{\mu}_i(\overline{V_m})+\delta(\overline{V_m})}{m\dim_K(V_m)}.
\end{align}
After taking the supremum limit on \eqref{eq:arith volume}, by \cref{def:volume} and \eqref{eq:limit of delta}, we have that
\begin{align}\label{eq:arith degree mu}
\limsup\limits_{m\to\infty}\frac{\widehat{\deg}(\overline{V_m})}{m\dim_K(V_m)}=\limsup\limits_{m\to\infty}\frac{\sum_{i=1}^{\dim_K(V_m)}\widehat{\mu}_i(\overline{V_m})}{m\dim_K(V_m)}.
\end{align}
Then  
\begin{align}
	\begin{split}\label{eq:arith volume mu}
	\widehat{\vol}_\chi(\overline{D})=&\limsup\limits_{m\to\infty}\frac{\widehat{\deg}(\overline{V_m})}{m^{d+1}/(d+1)!}\\=&(d+1)\cdot\limsup\limits_{m\to\infty}\left(\frac{\dim_K(V_m)}{m^{d}/d!}\cdot\frac{\widehat{\deg}(\overline{V_m})}{m\dim_K(V_m)}\right)\\=&(d+1)\cdot \vol_+(\overline{D})\cdot\limsup\limits_{m\to\infty}\frac{\sum_{i=1}^{\dim_K(V_m)}\widehat{\mu}_i(\overline{V_m})}{m\dim_K(V_m)}.
	\end{split}
\end{align}
where the last equality is from \eqref{eq:concave 1} and from \eqref{eq:arith degree mu}, notice that the supremum limit in the definition of $\vol_+(\overline{D})$ (see \eqref{eq:volume 1}) is actually a limit in \eqref{eq:concave 1} when $D$ is big.
On the other hand, notice that $\sup\limits_{\lambda\in\Delta(D)^\circ}\{G_{\overline{D}}(\lambda)\}=\widehat{\mu}^{\mathrm{asy}}_{\max}(\overline{D})=\limsup\limits_{m\to\infty}\frac{\widehat{\mu}_{\max}(V_m)}{m}<\infty$, so there exists $c>\max\{0, \widehat{\mu}^{\mathrm{asy}}_{\max}(\overline{D})\}$ such that $\frac{\widehat{\mu}_i(V_m)}{m}\leq \frac{\widehat{\mu}_{\max}(V_m)}{m}\leq c$ for any $m\in \N(V_\bullet)$ and $1\leq i\leq \dim_K(V_m)$. Hence for any $\alpha\in\R_{>0}$,  
\begin{align*}
\frac{1}{\vol_{\R^d}(\Delta_+(\overline{D}))}\int_{-\alpha}^{\infty}x\,  ((G_{\overline{D}})_*d\lambda)(dx) =&\frac{1}{\vol_{\R^d}(\Delta_+(\overline{D}))}\int_{-\alpha}^{c}x\,  ((G_{\overline{D}})_*d\lambda)(dx)\\ =&\lim\limits_{m\to\infty}\int_{-\alpha}^{c}x\, \nu_m(dx) \\
=&\lim\limits_{m\to\infty}\int_{-\alpha}^{\infty}x\, \nu_m(dx)\\
\geq &\limsup\limits_{m\to\infty}\int_{-\infty}^{\infty}x\, \nu_m(dx)\\
=&\limsup\limits_{m\to\infty}\frac{\sum_{i=1}^{\dim_K(V_m)}\widehat{\mu}_i(\overline{V_m})}{m\dim_K(V_m)}.
\end{align*}
Let $\alpha\to\infty$ in the above inequality, we get
\begin{align}\label{eq: measure mu ineq}
\frac{1}{\vol_{\R^d}(\Delta_+(\overline{D}))}\int_{-\infty}^{\infty}x\,  ((G_{\overline{D}})_*d\lambda)(dx) = \frac{d!}{\vol_+(\overline{D})}\int_{\Delta(D)^\circ}G_{\overline{D}}(\lambda)\,d\lambda\geq \limsup\limits_{m\to\infty}\frac{\sum_{i=1}^{\dim_K(V_m)}\widehat{\mu}_i(\overline{V_m})}{m\dim_K(V_m)}.
\end{align}
Combine \eqref{eq:arith volume mu} and \eqref{eq: measure mu ineq}, we have that
\[\widehat{\vol}_\chi(\overline{D})\leq (d+1)\cdot \vol_+(\overline{D})\cdot \frac{d!}{\vol_+(\overline{D})}\int_{\Delta(D)^\circ}G_{\overline{D}}(\lambda)\,d\lambda=(d+1)!\cdot\int_{\Delta(D)^\circ}G_{\overline{D}}(\lambda)\,d\lambda.\]
If  $\widehat{\mu}_{\min}^{\mathrm{asy}}(\overline{D}) = \liminf\limits_{m\to\infty}\frac{\widehat{\mu}_{\min}(m\overline{D})}{m}>-\infty$, 
by \cite[Remark~6.3.27]{chen2020arakelov}, we have that
\[\frac{1}{\vol_{\R^d}(\Delta_+(\overline{D}))}\int_{-\infty}^{\infty}x\,  ((G_{\overline{D}})_*d\lambda)(dx) = \frac{d!}{\vol_+(\overline{D})}\int_{\Delta(D)^\circ}G_{\overline{D}}(\lambda)\,d\lambda= \lim\limits_{m\to\infty}\frac{\sum_{i=1}^{\dim_K(V_m)}\widehat{\mu}_i(\overline{V_m})}{m\dim_K(V_m)},\]
in particular, the supremum limits in \eqref{eq:arith volume mu} are limits. Combine \eqref{eq:arith volume mu} and \eqref{eq: measure mu ineq}, we have that
\[\widehat{\vol}_\chi(\overline{D})=(d+1)\cdot \vol_+(\overline{D})\cdot \frac{d!}{\vol_+(\overline{D})}\int_{\Delta(D)^\circ}G_{\overline{D}}(\lambda)\,d\lambda=(d+1)!\cdot\int_{\Delta(D)^\circ}G_{\overline{D}}(\lambda)\,d\lambda.\]
This completed the proof of \eqref{eq:main theorem concave 2}.
\end{proof}

For any $c=(c(\omega))_{\omega\in\Omega}\in \mathscr{L}^1(\Omega,\mathcal{A},\nu)$, we denote $(0, c)\in \widehat{\Div}_{S,\Q}(U)_\cpt$ the compactified $S$-metrized divisor such that the underlying compactified divisor of $(0,c)$ on a projective model of $U$ is $0$, and the Green functions at $\omega\in\Omega$ are $c(\omega)$. For any $\overline{D}\in\widehat{\Div}_{S,\Q}(U)_\cpt$,  we write $\overline{D}(c)\coloneq \overline{D}+(0,c)$.

We list the following properties of the concave transforms.

\begin{proposition} \label{properties of concave transforms}
 Let $\overline{D}=(D,g), \overline{D'}=(D',g')\in \widehat{\Div}_{S,\Q}(U)_\cpt$ with $D, D'$ big. Then the following statements hold.
	\begin{enumerate1}
		\item \label{mutiple} For any $\alpha\in \Q_{>0}$, we have that $\Delta(\alpha\overline{D})=\alpha\Delta(\overline{D})$ and 
		\[G_{\alpha\overline{D}}(\alpha\lambda) = \alpha\cdot G_{\overline{D}}(\lambda)\]
		for any $\lambda\in\Delta(\overline{D})^\circ$.
		\item \label{shifting} 	Let $c=(c(\omega))_{\omega\in\Omega}\in \mathscr{L}^1(\Omega,\mathcal{A},\nu)$. Then 
		\[G_{\overline{D}(c)}=G_{\overline{D}}+\int_{\Omega}c(\omega)\,\nu(d\omega)\]
		on $\Delta(\overline{D}(c))^{\circ}=\Delta(D)^{\circ}$.
		\item \label{monotone}  If $\overline{D}\le \overline{D'}$, then $G_{\overline{D}}\le G_{\overline{D'}}$ on $\Delta(D)^{\circ}\subseteq \Delta(D')^{\circ}$.
		\item \label{inf and sup of G} We have that	\[\widehat{\mu}^{\mathrm{asy}}_{\min}(\overline{D})\le G_{\overline{D}}\le \widehat{\mu}^{\mathrm{asy}}_{\max}(\overline{D}).\]
		Moreover,  $\sup\limits_{\lambda\in\Delta(D)^\circ}G_{\overline{D}}(\lambda)=\widehat{\mu}_{\max}^{\mathrm{asy}}(\overline D)$.
		\item \label{additivity} For any $\lambda\in \Delta(\overline{D})^\circ$ and $\lambda'\in \Delta(\overline{D'})^\circ$, we have that $\lambda+\lambda'\in \Delta(\overline{D}+\overline{D'})^\circ$ and
		\[G_{\overline{D}+\overline{D'}}(\lambda+\lambda')\ge G_{\overline{D}}(\lambda)+G_{\overline{D'}}(\lambda').\] 
		\item \label{birational}	Let $\varphi\colon V\to U$ a birational morphism of normal quasi-projective varieties such that $\varphi^*Y_\bullet$ is an admissible flag of $V$. Then $\widehat{\vol}(\varphi^*\overline{D}) = \widehat{\vol}(\overline{D})$, $\Delta(\varphi^*D)^\circ=\Delta(D)^\circ$ and $G_{\varphi^*\overline{D}}=G_{\overline{D}}$ on $\Delta(D)^\circ$. 
	\end{enumerate1}
\end{proposition}
\begin{proof}
	Let $\overline{V}_{\bullet}=\{\overline{V}_{m}\}_{m\in\N} \coloneq \{(H_+^0(U,m\overline{D}), \metr_{m,\sup})\}_{m\in \N}$ (resp. $\overline{V_{\bullet}'}=\{\overline{V_{m}'}\}_{m\in\N} \coloneq \{(H_+^0(U,m\overline{D'}), \metr_{m,\sup}')\}_{m\in \N}$)  be the graded $K$-algebra of adelic vector bundles corresponding to $\overline{D}$ (resp. $\overline{D'}$), $\mathcal{H}$ the Harder-Narasimhan filtration on $V_{\bullet}$ (resp. $V_{\bullet}'$) induced by $\overline{D}$ (resp. $\overline{D}'$) and $\rho, \widetilde{\rho}, \Gamma^t(V_{\bullet})$ (resp. $\rho', \widetilde{\rho}'$, $\Gamma^t(V_{\bullet}')$) the corresponding functions and semi-group given in \cref{construction of concave transform}. 
	
	\ref{mutiple} This is from \cite[Remark~6.3.21~(2)]{chen2020arakelov}.
	
	\ref{shifting} For any $m\in\N_{\geq 1}$, let $\metr_{m,c,\sup}$ be the sup-norm on $H^0(U,m\overline{D}(c))$. Then
	\[\metr_{m,c,\sup,\omega}=e^{-c(\omega)}\metr_{m,\sup,\omega}\]
	on $H_+^0(U,m\overline{D})=H_+^0(U,m\overline{D}(c))$. This equality above of sup-norms now easily implies \ref{shifting} by the construction of concave transforms in \cref{construction of concave transform}.
	
	\ref{monotone}   The effectivity relation $\overline{D}\le \overline{D'}$ implies that there is an injective $K$-linear map 
	\[f_m\colon V_{m}\hookrightarrow V_{m}'\]
	for every $m\in\N$ which is furthermore norm-contractive at every $\omega\in\Omega$, i.e. $\|f_m\|_\omega\le 1$ for all $\omega\in\Omega$ (see \cref{basic invariant of adelic vector bundle} for the notation). This implies that $h(f_m)\le 0$. By \cite[Proposition 4.3.49]{chen2020arakelov} we have that $\mathcal{H}^t(V_{m})\subseteq \mathcal{H}^t(V_{m}')$ since $h(f_m)\le 0$ and the filtrations are non-increasing. This implies that $\rho\leq \rho', \widetilde{\rho}\leq \widetilde{\rho}'$, hence $\Gamma^t(V_{\bullet})\subset \Gamma^t(V_{\bullet}')$, for any $t$. By the construction of $G_{\overline{D}}, G_{\overline{D'}}$, we have $G_{\overline{D}}\leq G_{\overline{D'}}$, this completes the proof of \ref{monotone}. 
	
	\ref{inf and sup of G} 	From the discussion in \cref{construction of concave transform}, we have that $\sup\limits_{\lambda\in\Delta(D)^\circ}G_{\overline{D}}(\lambda)=\widehat{\mu}_{\max}^{\mathrm{asy}}(\overline D)$. It remains to show that $\widehat{\mu}_{\min}^{\mathrm{asy}}(\overline{D})\leq G_{\overline{D}}$. Notice that $\rho(m,\gamma)\geq \widehat{\mu}_{\min}(\overline{V_m})$ for any $(m,\gamma)\in \Gamma(V_{\bullet})$ (see \cref{construction of concave transform}). For any $t<\widehat{\mu}^{\mathrm{asy}}_{\min}(\overline{D})$, by definition we have that $\widehat{\mu}_{\min}(\overline{V_m})>mt$ for all large enough $m$. Then for any  large enough $m$ with $(m,\gamma)\in\Gamma(V_\bullet)$, we conclude that $\rho(m,\gamma)\geq \widehat{\mu}_{\min}(\overline{V_m})>mt$ which in turn implies $\widetilde{\rho}(m,\gamma)>mt$ for all $(m,\gamma)\in \Gamma(V_\bullet)$ by definition. In particular, we conclude that for any $t<\widehat{\mu}^{\mathrm{asy}}_{\min}(\overline{D})$, $\Gamma^t(V_\bullet)=\Gamma(V_\bullet)$, $\Delta(\Gamma^t(V_\bullet))=\Delta(V_\bullet)$ which clearly shows that $G_{\overline{D}}(\lambda)\ge t$ for all $\lambda\in \Delta(D)^\circ$. As $t<\widehat{\mu}^{\mathrm{asy}}_{\min}(\overline{D})$ was arbitrary we deduce that $\widehat{\mu}_{\min}^{\mathrm{asy}}(\overline{D})\leq G_{\overline{D}}$.

	
    \ref{additivity} This is from \cite[Proposition~6.3.28]{chen2020arakelov}, notice that the condition for $\dim_K(W_n)$ in \cite[Proposition~6.3.28]{chen2020arakelov} holds, see \cref{construction of concave transform}.
	
	\ref{birational} 
Since the sup-norm is stable under birational morphism, it suffices to show that $H^0(V,\varphi^*D)=H^0(U,D)$ which will imply the corresponding graded algebras of adelic vector bundles associated to $\varphi^*\overline{D}, \overline{D}$, respectively, are isomorphic, hence proves \ref{birational}. We have a homomorphism $H^0(U,{D})\to H^0(V,\varphi^*{D})$ which is injective since $\varphi$ is birational, it remains to show that if $f\in K(V)\simeq K(U)$ such that $\widehat{\mathrm{div}}(f)+\varphi^*\overline{D}\geq 0$, then $\widehat{\mathrm{div}}(f)+D\geq 0$. It suffices to show that $\varphi^*\overline{D}\geq 0$ implies $\overline{D}\geq 0$. If $\varphi^*\overline{D}\geq 0$, let $(\overline{D_j})_{j\in\N_{\geq 1}}$ be a sequence of $S$-metrized divisors in $\widehat{\Div}_{S,\Q}(U)_\CM$ decreasingly converging to $\overline{D}$ with respect to the $\overline{B}$-boundary topology for some weak boundary divisor $\overline{B}$ of $U$. Let $X_j$ be a normal projective $K$-model of $U$ such that $\overline{D_j}=(D_j,g_j)\in \widehat{\Div}_{S,\Q}(X_j)$. Let $X_j'$ be a projective $K$-model of $V$ such that the extension $\varphi_j\colon X_j'\to X_j$ of $\varphi$ exists. Moreover, since $\varphi^*\overline{D_j}\geq \varphi^*\overline D\geq 0$ in $\widehat{\Div}_{S,\Q}(U)_\cpt$, we can assume that $\varphi_j^*\overline{D_j}\geq 0$ in $\widehat{\Div}_{S,\Q}(X_j')$. This implies that $g_j\circ\varphi_j\geq0$ and the Weil divisor corresponding to $\varphi_j^*{D}_j$ is non-negative, so $g_j\geq 0$ and the Weil divisor corresponding to $\overline{D_j}$ is non-negative. Since $X_j$ is normal, so $\overline{D_j}\geq 0$. This completes the proof of \ref{birational}. 
\end{proof}
\begin{remark}\label{rmk:monotone of volume}
		Note that if $\overline{D}\leq \overline{D'}$ in $\widehat{\Div}_{S,\Q}(U)_\cpt$, as in the proof of \cref{properties of concave transforms}~\ref{monotone}, we have that $\widehat{\deg}_+(H_+^0(U,\overline{D}),\metr_{\sup}) \leq \widehat{\deg}_+(H_+^0(U,\overline{D'}),\metr_{\sup})$. In particular, $\widehat{\vol}(\overline{D})\leq \widehat{\vol}(\overline{D'})$.
\end{remark}

\begin{corollary}\label{cor:bigness}
	Let $\overline{D}=(D,g)\in\widehat{\Div}_{S,\Q}(U)_\cpt$. Then $\overline{D}$ is big if and only if $D$ is big and $\widehat{\mu}_{\max}^{\mathrm{asy}}(\overline{D})>0$. 
\end{corollary}
\begin{proof}
	As argued in the proof of \cite[Proposition~6.4.18]{chen2020arakelov}, the corollary is from \cref{properties of concave transforms}~\ref{inf and sup of G} and \cref{theorem:concavemain}.  For the convenience of readers, we give a complete proof. 
	Let $\overline{V}_{\bullet}=\{\overline{V}_{m}\}_{m\in\N} \coloneq \{(H_+^0(U,m\overline{D}), \metr_{m,\sup})\}_{m\in \N}$ be the graded $K$-algebra of adelic vector bundle corresponding to $\overline{D}$.
	
	If $\overline{D}$ is big, then for big enough $m\in \N$, $\widehat{\deg}_+(\overline{V_m})>0$. By \cite[Remark 4.3.53]{chen2020arakelov} we have that 
	\[\widehat{\deg}_+(\overline{V_m})\le \dim_K(V_m)\cdot \widehat{\mu}_1(\overline{V_m}).\]
	Then we can see that 
	\[0<\widehat{\vol}(\overline{D})=\limsup_{m\to\infty}\frac{\deg_+(\overline{V_m})}{m^{d+1}}\le \limsup_{m\to\infty}\frac{\dim_K(V_m)}{m^d}\cdot \frac{\widehat{\mu}_1(\overline{V_m})}{m}=\vol(D)\cdot \widehat{\mu}^{\text{asy}}_{\max}(\overline{D}).\]
	This then clearly implies that $\vol(D)>0$ and $\widehat{\mu}_{\max}^{\mathrm{asy}}(\overline{D})>0$.
	
	Conversely, assume that $D$ is big and $\widehat{\mu}_{\max}^{\mathrm{asy}}(\overline{D})>0$. Since $D$ is big, we have that $\Delta(D)^\circ\not=\emptyset$. Since $\widehat{\mu}_{\max}^{\mathrm{asy}}(\overline{D})>0$ and $G_{\overline{D}}$ is continuous on $\Delta(D)^\circ$, there is a non-empty open subset $W$ of $\Delta(D)^\circ$ such that $G_{\overline{D}}>0$ on $W$. By \cref{theorem:concavemain}, we have that
	\[\widehat{\vol}(\overline{D})=(d+1)!\int_{\Delta(D)^\circ}\max\{G_{\overline{D}}(\lambda), 0\}\, d\lambda\geq (d+1)!\int_{W}G_{\overline{D}}(\lambda)\, d\lambda>0,\]
	which implies that $\overline{D}$ is big.
\end{proof}

\begin{corollary}
   \label{cor:sumofbigisbig}
Let $\overline{D_1}$, $\overline{D_2}\in \widetilde{\Div}_{S,\Q}(U)_\cpt$. If $\overline{D_1}$, $\overline{D_2}$ are big. Then 
\[\widehat{\vol}(\overline{D_1}+\overline{D_2})^{\frac{1}{d+1}}\ge \widehat{\vol}(\overline{D_1})^{\frac{1}{d+1}}+\widehat{\vol}(\overline{D_2})^{\frac{1}{d+1}}.\]
In particular, $\overline{D_1}+\overline{D_2}$ is big.
\end{corollary}
\begin{proof}
Arguing as in the proof of \cite[Theorem 6.4.7]{chen2020arakelov}, we easily deduce the corollary from \cref{properties of concave transforms}~\ref{additivity}, \cref{theorem:concavemain} and the classical Brunn-Minkowski inequality.
\end{proof}
\begin{remark}\label{remark:volme inequality by adding big}
	An easy observation from the corollary is that, if $\overline{D}\in\widehat{\Div}_{S,\Q}(U)_\cpt$ is big, then for any $\overline{M}\in \widehat{\Div}_{S,\Q}(U)_\cpt$, we have that
	\[\widehat{\vol}(\overline{M})\leq \widehat{\vol}(\overline{M}+\overline{D}).\]
	Indeed, it is trivial if $\widehat{\vol}(\overline{M})=0$, otherwise the inequality is from the corollary.
	\end{remark}

For further application, we give the following definition.
\begin{definition}
	Let $\overline{D}=(D,g)\in\widehat{\Div}_{S,\Q}(U)_\cpt$ with $D$ big and $G_{\overline{D}}$ the concave transform of $\overline{D}$ given in \cref{construction of concave transform}. We define the \emph{numerical $\chi$-volume} of $\overline{D}$ as 
	\[\widehat{\vol}_{\chi}^{\mathrm{num}}(\overline{D})\coloneq (d+1)!\int_{\Delta(D)^\circ}G_{\overline{D}}(\lambda) \ d\lambda.\]
	By \cref{theorem:concavemain}, we have that $\widehat{\vol}_{\chi}^{\mathrm{num}}(\overline{D})\geq\widehat{\vol}_{\chi}(\overline{D})$ with equality if $\widehat{\mu}_{\min}^{\mathrm{asy}}(\overline{D})>-\infty$. 
\end{definition}
Next we record an easy corollary of the previous theorem which will be useful later.
\begin{corollary}
	\label{corol:volsame}
	Let $\overline{D}=(D,g)\in\widehat{\Div}_{S,\Q}(U)_\cpt$ with $D$ big. If $G_{\overline{D}}\ge 0$ on $\Delta(D)^{\circ}$, then $\widehat{\vol}(\overline{D})=\widehat{\vol}_\chi^{\num}(\overline{D}) 
	$. More generally, if $\inf\limits_{\lambda\in\Delta(D)^\circ}G_{\overline{D}}(\lambda)>-\infty$, then there exists $c\in \mathscr{L}^1(\Omega,\mathcal{A},\nu)$ such that
	$\widehat{\vol}(\overline{D}(c))=\widehat{\vol}_\chi^\num(\overline{D}({c}))$. 
\end{corollary}
\begin{proof}
	The proof is immediate from the two statements of \cref{theorem:concavemain} and \cref{properties of concave transforms}~\ref{shifting}.
\end{proof}

\subsection{Continuity of arithmetic volumes}

\begin{lemma}
    \label{lemma:twistbig}
    Let $\overline{D}=(D,g)\in\widehat{\Div}_{S,\Q}(U)_{\cpt}$ with $D$ big. Then there is a non-negative integrable function $c\colon\Omega\rightarrow \R_{\ge 0}$ such that the twist $\overline{D}(c)\coloneqq \overline{D}+(0,c)$ is 
    big.
\end{lemma}
\begin{proof}
    Let $\alpha\in\Delta(D)^\circ$ and $r\coloneq G_{\overline{D}}(\alpha)>-\infty$. There is a function $c\in\mathscr{L}^1(\Omega,\mathcal{A},\nu)_{\geq 0}$ such that $\int_{\Omega}c(\omega)\, \nu(d\omega)>-r$. Then clearly from \cref{properties of concave transforms}~\ref{shifting} we have that $$G_{\overline{D}(c)}(\alpha)=G_{\overline{D}}(\alpha)+ \int_{\Omega} c(\omega)\, \nu(d\omega)={r+\int_{\Omega}c(\omega)\, \nu(d\omega)}>0$$ and thus we have $\widehat{\vol}(\overline{D}(c))>0$ by \cref{theorem:concavemain}.
\end{proof}
We next show that the big cone is open along arbitrary directions. This is the quasi-projective analogue of \cite[Proposition 6.4.23]{chen2020arakelov}.
\begin{lemma}
    \label{lemma:bigtwists}
    Let $\overline{D}=(D,g), \overline{M}=(M,f)\in\widehat{\Div}_{S,\Q}(U)_{\cpt}$. If $\overline{D}$ is big,  then there is an $n_0\in \N_{\geq 1}$ such that $n\overline{D}+\overline{M}$ is big for all $n\ge n_0$.
\end{lemma}
\begin{proof}
    For any $\overline{D'}=(D',g')\in \widehat{\Div}_{S,\Q}(U)_\cpt$ with $D'$ big, we set $\widetilde{G}_{\overline{D'}}\coloneqq \max\{G_{\overline{D'}},0\}$ for simplicity. By bigness of $\overline{D}$, we know that the underlying compactified (geometric) divisor $D$ is big from \cref{cor:bigness}. Hence by continuity of geometric volumes, there is an integer $m$ such that $mD+M$ is big. On the other hand, by \cref{properties of concave transforms}~\ref{shifting}, for any integrable function $c\colon \Omega\to \R_{\ge 0}$, we have that
    \[G_{\overline{D}(c)}=G_{\overline{D}}+\int_{\Omega} c(\omega)\, \nu(d\omega)\] which clearly implies \[\widetilde{G}_{\overline{D}}(\alpha)\le \widetilde{G}_{\overline{D}(c)}(\alpha)\le \widetilde{G}_{\overline{D}}(\alpha)+\int_{\Omega}c(\omega)\, \nu(d\omega)\]
    for all $\alpha\in\Delta(D)^{\circ}$. Similarly we have 
    \[\widetilde{G}_{\overline{D}}(\alpha)\ge \widetilde{G}_{\overline{D}(-c)}(\alpha)\ge \widetilde{G}_{\overline{D}}(\alpha)-\int_{\Omega}c(\omega)\, \nu(d\omega),\]
     for all $\alpha\in\Delta(D)^{\circ}$.
    By \cref{theorem:concavemain} the above then clearly implies
    \begin{equation}
    \label{equa:lagbe}
    \widehat{\vol}(\overline{D})\le\widehat{\vol}(\overline{D}(c))\le \widehat{\vol}(\overline{D})+(d+1)\cdot\vol(D)\cdot \int_{\Omega} c(\omega)\, \nu(d\omega).
    \end{equation}
    \begin{equation}
        \label{equa:lagbeagain}
           \widehat{\vol}(\overline{D})\ge\widehat{\vol}(\overline{D}(-c))\ge \widehat{\vol}(\overline{D})-(d+1)\cdot\vol(D)\cdot \int_{\Omega} c(\omega)\, \nu(d\omega).
    \end{equation}
      By \cref{lemma:twistbig}, there is a non-negative integrable function $c\colon \Omega\to\R_{\geq 0}$ such that $(m\overline{D}+\overline{M})(c)$ is big. Applying our observation in \eqref{equa:lagbeagain} to the non-negative integrable function $c/N$ for any positive integer $N\in\N_{\geq 1}$, we get
    \[\widehat{\vol}\left(\overline{D}\left(c/N\right)\right)\ge \widehat{\vol}(\overline{D})-\frac{(d+1)\cdot\vol(D)}{N}\int_{\Omega}c(\omega)\, \nu(d\omega).\] Since $\widehat{\vol}(\overline{D})>0$ and $c$ is an integrable function, we can choose a large integer $N$ such that $\widehat{\vol}((\overline{D}(c/N))>0$. Now we write 
    \[(m+N)\overline{D}+\overline{M}=(m\overline{D}+\overline{M})(c)+N\overline{D}(-c/N)\] which is clearly big by choice of $c$ and $N$ and \cref{cor:sumofbigisbig}. Now clearly $n_0=m+N$ satisfies the claim again by \cref{cor:sumofbigisbig}.
\end{proof}
We now state the main theorem which is the quasi-projective analogue of \cite[Theorem 6.4.24]{chen2020arakelov}.
\begin{theorem}
    \label{theorem:continuityofvolumes}
   {Let $(\overline{D_j})_{j\in\N_{\geq 1}}$ be a sequence in $\widehat{\Div}_{S,\Q}({U})_{\cpt}$ converging to $\overline{D}=(D,g)$ with respect to the $\overline{B}$-boundary topology for some weak boundary divisor $\overline{B}=(B,g_B)\in\widehat{\Div}_{S,\Q}(U)_\CM$. Then
    \[\widehat{\vol}(\overline{D})=\lim_{j\to \infty}\widehat{\vol}(\overline{D_j}).\]}
\end{theorem}
\begin{proof}
	{There is a sequence $(\varepsilon_{j})_{j\geq 0}\subset \Q_{>0}$ converging to $0$ such that
	\[\overline{D}-\varepsilon_j\overline{B}\leq \overline{D_j}\leq \overline{D}+\varepsilon_j\overline{B},\]
	Then by \cref{rmk:monotone of volume}, 
	\[\widehat{\vol}(\overline{D}-\varepsilon_j\overline{B})\leq \widehat{\vol}(\overline{D_j})\leq \widehat{\vol}(\overline{D}+\varepsilon_j\overline{B}).\]
	So it suffices to show that
	\begin{align}\label{eq:continuity of volume1}
		\widehat{\vol}(\overline{D}) = \lim_{\Q\ni t\to 0}\widehat{\vol}(\overline{D}+t\overline{B}).
		\end{align}}
	
	{We firstly consider the case where $D\in \widetilde{\Div}_{\Q}(U)_\cpt$ is not big, i.e. ${\vol}(D)=0$. Notice that in this case, $\widehat{\vol}(\overline{D})=0$ by \cref{cor:bigness}.
		  \begin{itemize}
			\item For any $t\in\Q_{\leq 0}$, $\vol(D+tB)=0$, so $\widehat{\vol}(\overline{D}+t\overline{B})=0$ by \cref{cor:bigness} which implies that
				\begin{align}\label{eq:continuity of volume2}
				\widehat{\vol}(\overline{D}) = \lim_{\Q_{\leq 0}\ni t\to 0}\widehat{\vol}(\overline{D}+t\overline{B}).
			\end{align}
			\item If ${\vol}(D+tB) = 0$ for $t\in\Q_{>0}$ small enough, then $\widehat{\vol}(\overline{D}+t\overline{B}) = 0$ when $t\in \Q_{>0}$ is small by \cref{cor:bigness}, so $$\lim\limits_{\Q_{>0}\ni t\to 0}=\widehat{\vol}(\overline{D}+t\overline{B})=0.$$ If ${\vol}(D+tB) > 0$, i.e. $D+tB$ is big, for any $t\in\Q_{>0}$, by \cref{properties of concave transforms}~\ref{monotone}~\ref{inf and sup of G}, \cref{lemma:slopebounded}, we have that
			\[G_{\overline{D}+t\overline{B}}\leq G_{\overline{D}+\overline{B}}\leq \widehat{\mu}_{\max}^{\mathrm{asy}}(\overline{D}+\overline{B})<\infty\]
			on $\Delta(D+tB)$ for any $t\in\Q_{<1}$. By \cref{theorem:concavemain} and \cref{lemma:refokoun},
			\begin{align*}
\widehat{\vol}(\overline{D}+t\overline{B})=&(d+1)!\cdot \int_{\Delta(D+tB)}\max\{G_{\overline{D}+t\overline{B}}(\lambda),0\}\, d\lambda\\ \leq& (d+1)!\cdot\max\{\mu^{\mathrm{asy}}_{\max}(\overline{D}+\overline{B}),0\}\cdot \frac{1}{d!}\cdot \vol(D+tB)\\
=&(d+1)\cdot\max\{\mu^{\mathrm{asy}}_{\max}(\overline{D}+\overline{B}),0\}\cdot \vol(D+tB)
			\end{align*}
			for any $t\in\Q_{<1}$. So 
			\[\lim\limits_{\Q_{>0}\ni t\to0}\widehat{\vol}(\overline{D}+t\overline{B})=0\]
			by the continuity of geometric volumes, see \cref{geometric volume}~\ref{continuity of geometric volume}.
		\end{itemize}
	From our discussion above, \eqref{eq:continuity of volume2} holds when $D$ is not big.}
	
    If $D$ is big,  by \cref{lemma:twistbig}, we can find a non-negative integrable function $c\colon \Omega\rightarrow \R$  such that $\overline{D}(c)=(D,g+c)$ is big. Then by \cref{lemma:bigtwists}, we can choose a positive integer $n$ such that $n\overline{D}(c)\pm \overline{B}$ is big. Let $t\in \Q_{>0}$. Then by \cref{remark:volme inequality by adding big} we have that
    \[\widehat{\vol}(\overline{D}-nt\cdot\overline{D}(c))\le \widehat{\vol}(\overline{D}-t\overline{B})\le \widehat{\vol}(\overline{D}+nt\cdot\overline{D}(c))\] and  
    \[\widehat{\vol}(\overline{D}-nt\cdot\overline{D}(c))\le \widehat{\vol}(\overline{D}+t\overline{B})\le \widehat{\vol}(\overline{D}+nt\cdot\overline{D}(c)),\]
    since $nt\overline{D}(c)\pm t\overline{B}$ is big.
    But \[\overline{D}-nt\cdot \overline{D}(c)=(1-nt)\overline{D}-nt\cdot(0,c)=(1-nt)\cdot\left(\overline{D}-\frac{nt}{1-nt}(0,c)\right).\]
    Similarly,
    \[\overline{D}+nt\cdot \overline{D}(c)=(1+nt)\cdot\left(\overline{D}+\frac{nt}{1+nt}(0,c)\right).\]
    Thus it reduces to showing 
    \[\widehat{\vol}(\overline{D})=\lim_{\Q\ni t\to 0}\widehat{\vol}\left((1+nt)\cdot\left(\overline{D}+\frac{nt}{1+nt}(0,c)\right)\right).\]
    Using homogeneity of arithmetic volumes and the fact that {$\lim\limits_{t\to 0}(1+nt)=1$}, 
    it is enough to show that 
    \[\widehat{\vol}(\overline{D})=
    \lim_{\Q\ni t\to 0}\widehat{\vol}\left(\overline{D}+\frac{nt}{1+nt}(0,c)\right).\]
    Since $\lim\limits_{t\to 0}\frac{nt}{1-nt}= 0$, it amounts showing that for any non-negative integrable function $c\colon\Omega\rightarrow \R$, we have 
    \[\lim_{\Q\ni t\to 0}\widehat{\vol}(\overline{D}+t(0,c))=\widehat{\vol}(\overline{D})).\]
    However from \eqref{equa:lagbe}, we have that for $t>0$,
    \[\widehat{\vol}(\overline{D})\le \widehat{\vol}(\overline{D}(tc))\le \widehat{\vol}(\overline{D})+t(d+1)\cdot\vol(D)\cdot\int_{\Omega}c(\omega)\, \nu(d\omega).\]
    Similarly from \eqref{equa:lagbeagain} on the other side, we have
    \[\widehat{\vol}(\overline{D})\ge \widehat{\vol}(\overline{D}(-tc))\ge \widehat{\vol}(\overline{D})-t(d+1)\cdot\vol(D)\cdot\int_{\Omega}c(\omega)\, \nu(d\omega).\]
    The claim now clearly follows by letting $t\to 0$ as $c$ is an integrable function. This completed the proof of \cref{theorem:continuityofvolumes}.
\end{proof}

{
\begin{corollary}\label{cor:continuity of arithmetic volume2}
	Let $\overline{D}, \overline{M_1}, \dots, \overline{M_r}\in\widehat{\Div}_{S,\Q}({U})_{\cpt}$. Then
	\[\lim\limits_{\Q^{r}\ni(t_1,\dots, t_r)\to0}\widehat{\vol}(\overline{D}+t_1\overline{M_1}+ \cdots+t_r\overline{M_r}) = \widehat{\vol}(\overline{D}).\]
\end{corollary}}
\begin{proof}
	For any $1\leq i\leq r$, let $\overline{M_i'}\in \widehat{\Div}_{S,\Q}(U)_{\CM}$ be an effective $S$-metric divisor on some projective model of $U$ such that $\overline{M_i'}\pm \overline{M_i}\geq 0$. Denote $\overline{M} \coloneq \overline{M_1'}+\cdots+\overline{M_r'}$. It suffices to show that
	\begin{align}\label{eq:corollary continuity}
\widehat{\vol}(\overline{D}) = \lim_{\Q\ni t\to 0}\widehat{\vol}(\overline{D}+t\overline{M}).
	\end{align}
	By \cref{properties of concave transforms}~\ref{birational}, after shrinking $U$, we may assume that $\overline{M}$ is a weak boundary divisor of $U$. Then \eqref{eq:corollary continuity} is exactly \eqref{eq:continuity of volume1} which is proved. This proved the corollary.
\end{proof}
\subsection{Essential minimum}
\begin{definition}
	Let $\overline{D}\in\widehat{\Div}_{S,\Q}(U)_\cpt$. We define the \emph{essential minimum} of $\overline{D}$ as
	\[\zeta_{\mathrm{ess}}(\overline{D})\coloneq\sup\limits_{\text{$V\subset U$ open}}\inf\limits_{x\in V(\overline{K})}h_{\overline{D}}(x).\]
\end{definition}

We next show \cref{thm:essential minimum} which gives an analogue for compactified divisors of the fundamental inequality of Zhang. To prove this theorem, we need the following definition from \cite[\S 1.1.10, Remark~1.1.67]{chen2020arakelov}.
\begin{art}\label{psi-direct sum}
	Let $\overline{V}=(V,\metr_V), \overline{W}=(W,\metr_W)$ be $S$-norm vector spaces and $\psi\colon [0,1]\to[0,1]$ such that $\max\{t,1-t\}\leq \psi(t)$ for any $t\in[0,1]$. For any $(x,y)\in V_{\omega}\oplus W_{\omega}$, set 
	\[\|(x,y)\|_{\psi,\omega}\coloneq(\|x\|_{V,\omega}+\|y\|_{W,\omega})\cdot\psi\left(\frac{\|x\|_{V,\omega}}{\|x\|_{V,\omega}+\|y\|_{W,\omega}}\right).\]
	We call $\metr_\psi=\{\metr_{\psi,\omega}\}_{\omega\in\Omega}$ the \emph{$\psi$-direct sum} of $\metr_V$ and $\metr_W$.
	
	A basis $\{e_1,\dots, e_r\}\subset V$ is a \emph{Hadamard basis} if 
	\[\|e_1\wedge\cdots\wedge e_r\|_{V,\det,\omega} = \|e_1\|_{V,\omega}\cdots\|e_r\|_{V,\omega},\]
	for every $\omega\in\Omega$ where $\metr_{V,\det}=\{\metr_{V,\det,\omega}\}_{\omega\in\Omega}$ is the determinant norm of $\metr_V$ defined in \cref{basic invariant of adelic vector bundle}.
\end{art}

\begin{theorem}\label{thm:essential minimum}
	Let $\overline{D}=(D,g)\in \widehat{\Div}_{S,\Q}(U)_\cpt$. Then \begin{align} \label{essential minimum1}
		\zeta_{\mathrm{ess}}(\overline{D})\geq \widehat{\mu}_{\max}(\overline{D}) \ \text{ and } \ \zeta_{\mathrm{ess}}(\overline{D})\geq \widehat{\mu}_{\max}^{\mathrm{asy}}(\overline{D}).
	\end{align}
	In particular, if $D$ is big, then
	\begin{align} \label{essential minimum2}
		\zeta_{\mathrm{ess}}(\overline{D})\geq \widehat{\mu}_{\max}^{\mathrm{asy}}(\overline{D})\geq \frac{\widehat{\vol}_\chi^{\mathrm{num}}(\overline{D})}{(d+1)\cdot\vol(D)}\geq \frac{\widehat{\vol}_\chi(\overline{D})}{(d+1)\cdot\vol(D)}.
	\end{align}
\end{theorem}
\begin{proof}
	Notice that $\zeta_{\mathrm{ess}}(m\overline{D}) = m\zeta_{\mathrm{ess}}(\overline{D})$, then the second inequality of \eqref{essential minimum1} is a consequence of the first inequality of \eqref{essential minimum1}. If $D$ is big, let $\Delta(D)$ be the Okounkov body of $D$ defined in \cref{okounkov body of a graded linear series}. By \cref{properties of concave transforms}~\ref{inf and sup of G} and \cref{lemma:refokoun}, 
	we have that
	\[\widehat{\mu}_{\max}^{\mathrm{asy}}(\overline{D})\geq \frac{1}{\vol_{\R^d}(\Delta(D))}\int_{\Delta(D)^\circ}G_{\overline{D}}(\lambda)\ d\lambda=\frac{d!}{\vol(D)}\cdot\frac{\widehat{\vol}_\chi^{\mathrm{num}}(\overline{D})}{(d+1)!}=\frac{\widehat{\vol}_\chi^{\mathrm{num}}(\overline{D})}{(d+1)\cdot\vol(D)}.\]
	This shows the second inequality of \eqref{essential minimum2}. Hence it suffices to show the first inequality of \eqref{essential minimum1}.
	
	We follow the idea of the proof of \cite[Proposition~6.4.4]{chen2020arakelov} to show that $\zeta_{\mathrm{ess}}(\overline{D})\geq \widehat{\mu}_{\max}(\overline{D})$. 
	Let $S_{\overline{K}}=(\overline{K}, \Omega_{\overline{K}}, \mathcal{A}_{\overline{K}}, \nu_{\overline{K}})$ be the adelic curve associated to the algebraic closure $\overline{K}$ induced by $S$ and let $\overline{V}=(V,\metr)=(H_+^0(U,\overline{D}),\metr_{\sup})$ be the adelic vector bundle of auxiliary global section associated to $\overline{D}$ defined in \cref{def:auxglobsec}. Let $t\in\R$ such that $\zeta_{\mathrm{ess}}(\overline{D})<t$. Then there is an infinite subset $\Lambda$ of $U(\overline{K})$ such that $\Lambda$ is dense in $U$ and $h_{\overline{D}}(x)<t$ for all $x\in \Lambda$. Let $W\subset V$ be an arbitrary subspace.  Notice that $V$ is a subspace of $H^0(U,D|_U)$ and hence a subspace of $K(U)$ since $U$ is integral. We consider the evaluation map
	\[W\otimes_K\overline{K}\to \prod\limits_{x\in \Lambda}\kappa(x)\]
	which is injective since $\Lambda$ is dense in $U$, where $\kappa(x)$ is the residue field of $x\in\Lambda$. Then there are points $x_1,\dots, x_{\dim_K(W)}\in\Lambda$ such that the evaluation map $\varphi\colon W\otimes_K\overline{K}\to \bigoplus\limits_{i=1}^{\dim_K(W)}\kappa(x_i)$ is bijective. For any $\omega\in\Omega_{\overline{K}}$, let $\metr_{i,\omega}$ be a norm of $\kappa(x_i)_\omega\coloneq\kappa(x_i)\otimes_{\overline{K}}\overline{K}_\omega$ given by $\|1\|_{i,\omega}\coloneq \exp(g_{\pi(\omega)}(\sigma(x_i)))$ where $\pi\colon \Omega_{\overline{K}}\to\Omega_K$ is the canonical map and $\sigma\colon (U_{\overline{K}})^\an_{\omega}\to U_{\pi(\omega)}^\an$ is the canonical morphism as analytic spaces. We equip $\bigoplus\limits_{i=1}^{\dim_K(W)}\kappa(x_i)$ with the $\psi$-direct sum $\metr_{\psi}=\{\metr_{\psi,\omega}\}_{\omega\in\Omega}$, where $\psi\colon [0,1]\to[0,1], \ \ t\mapsto \max\{t,1-t\}$, i.e. for any $(s_1,\dots, s_{\dim_K(W)})\in \bigoplus\limits_{i=1}^{\dim_K(W)}\kappa(x_i)_\omega$, we have that
	\[\|(s_1,\dots, s_{\dim_K(W)})\|_{\psi,\omega}=\max\{\|s_1\|_{1,\omega},\dots,\|s_{\dim_K(W)}\|_{\dim_K(W),\omega}\}.\] If we denote $\{e_i\}_{i=1}^{\dim_K(W)}$ a basis of $\bigoplus\limits_{i=1}^{\dim_K(W)}\kappa(x_i)$ such that $e_i\in \kappa(x_i)$, then this basis is orthogonal with respect to $\metr_{\psi,\omega}$ for any $\omega\in\Omega_{\overline{K}}$. By \cite[Proposition~1.2.23]{chen2020arakelov}, $\{e_i\}_{i=1}^{\dim_K(W)}$ is a Hadamard basis for $\metr_{\psi,\omega}$ for any $\omega\in \Omega_{\overline{K}}$. In particular,
	\[\widehat{\deg}\left(\bigoplus\limits_{i=1}^{\dim_K(W)}\left(\kappa(x_i),\metr_\psi\right)\right) = \sum\limits_{i=1}^{\dim_K(W)}h_{\overline{D}}(x_i)\leq \dim_K(W)\cdot t.\]
	Moreover, for any $\omega\in\Omega_{\overline{K}}$ the operator norm $\|\varphi\|_\omega\leq 1$. By \cite[Proposition~4.3.18]{chen2020arakelov}, one has
	\[\widehat{\mu}(\overline{W})=\widehat{\mu}(\overline{W}\otimes_K\overline{K})\leq \frac{1}{\dim_K(W)}\widehat{\deg}\left(\bigoplus\limits_{i=1}^{\dim_K(W)}\left(\kappa(x_i),\metr_\psi\right)\right)\leq t.\]
	This completes the proof of \eqref{essential minimum1}.
\end{proof}
\subsection{Height Inequality}
We end this section by obtaining a height inequality which is a generalization of the height inequalities obtained in \cite[Theorem 5.3.5~{(1)}(3)]{yuan2021adelic}. Note that this height inequality has been used later as one of the tools to obtain a uniform Bogomolov type result, see \cite[\S 4.5]{yuan2021arithmetic}. We generalize this in the setting of adelic curves.
\begin{theorem}
	\label{theorem:heightinequality}
	Let $\overline{D}=(D,g), \overline{M}\in \widehat{\Div}_{S,\Q}(U)_{\cpt}$.
    \begin{enumerate1}
			\item \label{theorem:heightinequality1} If $\overline{D}$ is big, then there exists $\epsilon>0$ and a non-empty Zariski open subset $V$ of $U$ such that 
			\[h_{\overline{D}}(x)\ge \epsilon\cdot h_{\overline{M}}(x).\]
			\item \label{theorem:heightinequality2} If $D$ is big, then there exists $\epsilon>0$, $c\in \R$ and a non-empty Zariski open subset $V$ of $U$ such that 
			\[h_{\overline{D}}(x)\ge \epsilon\cdot h_{\overline{M}}(x)-c \text{ for all } x\in V(\overline{K}).\]
	\end{enumerate1}.
\end{theorem}
\begin{proof}
	We see that \ref{theorem:heightinequality1} implies that \ref{theorem:heightinequality2}. In fact, if $D$ is big, then there exists $c_0\in \R$ and $\lambda\in \Delta(D)^{\circ}$ such that $G_{\overline{D}}(\lambda)>-c_0$. We take a function ${c}\in \mathscr{L}^1(\Omega, \mathcal{A}, \nu)$ with $c_0= \int_{\Omega}\widetilde{c}(\omega)\, \nu(d\omega)$. By \cref{properties of concave transforms}~\ref{shifting}~\ref{inf and sup of G}, \[\widehat{\mu}_{\max}^{\mathrm{asy}}(\overline{D}(c))\geq G_{\overline{D}({c})}(\lambda) = G_{\overline{D}}(\lambda)+c_0>0.\]
Then $\overline{D}({c})$ is big by \cref{cor:bigness}. We can apply \ref{theorem:heightinequality1} to $\overline{D}({c})$. This gives \ref{theorem:heightinequality2} by the simple relation
\[h_{\overline{D}({c})}(x)=h_{\overline{D}}(x)+\int_{\Omega}{c}(\omega)\, \nu(d\omega) = h_{\overline{D}}(x)+c_0.\]
	
	It remains to show \ref{theorem:heightinequality1}. By \cref{cor:continuity of arithmetic volume2}, there is $\epsilon\in \Q_{>0}$ such that $\widehat{\vol}(\overline{D}-\epsilon \overline{M})>0$, i.e. $\overline{D}-\epsilon \overline{M}$ is big. By \cref{thm:essential minimum} and \cref{cor:bigness}, we then have that
	\[\zeta_{\mathrm{ess}}(\overline{D}-\epsilon\cdot\overline{M})\geq \widehat{\mu}_{\max}^{\mathrm{asy}}(\overline{D}-\epsilon\cdot \overline{M})>0.\]
	By definition of the essential minima, we then have a non-empty Zariski open subset $V$ of $U$ such that 
	\[h_{\overline{D}-\epsilon\cdot\overline{M}}(x)=h_{\overline{D}}(x)-\epsilon\cdot h_{\overline{M}}(x)>0  \text{ for all } x\in V(\overline{K})\]
	which readily proves \ref{theorem:heightinequality1}.
\end{proof}
The generalization of \cite[Theorem 5.3.5~(2)]{yuan2021adelic} will be seen in \cref{prop:height inequality 2} as an application of arithmetic Hilbert-Samuel formula.

\section{Arithmetic Hilbert-Samuel for singular metrics}

\label{section: arithmetic hilbert}
In this section, we will build the relation between the numerical $\chi$-volume and arithmetic auto-intersection number of a given relatively nef compactified $S$-metrized divisor, i.e. the arithmetic Hilbert-Samuel formula. We fix an adelic curve $S=(K,\Omega,\mathcal{A},\nu)$ satisfying the condition \ref{finite measure for archimedean} \ref{exists finite measure for subset} \ref{K countable or A discrete} in \S \ref{subsection:Global theory: Compactified (model) divisors}, and assume that $S$ has either the strong tensorial minimal slope property of level $\ge C_0$ or the Minkowski property of level $\ge C_0$ for some $C\in\R_{\geq0}$ (see \cref{def:tensorialproperty}). We also fix a $d$-dimensional normal quasi-projective variety $U$ over $K$ and an admissible flag $Y_\bullet$ of $U_{\overline{K}}$, see \cref{flag of a quasi-projective variety}.

\subsection{Behaviour of the concave transform under perturbations}
Recall that for $\overline{D}=(D,g)\in\widehat{\Div}_{S,\Q}(U)_\cpt$ with $D\in\widetilde{\Div}_{\Q}(U)_\cpt$ big, we denote by $\Delta(D)^{\circ}$ the topological interior of the Okounkov body of $D$ calculated with respect to the fixed flag $Y_\bullet$ in $U$, see \cref{okounkov body of a graded linear series}. Furthermore we have the concave transform $G_{\overline{D}}$ on $\Delta(D)^{\circ}$ constructed in \cref{construction of concave transform} corresponding to $\overline{D}$. 

\begin{art} \label{limits of concave transforms}
Let $(\overline{D_j})_{j\in\N_{\geq 1}} = (D_j, g_j)_{j\in\N_{\geq 1}}$ be a Cauchy sequence in $\widehat{\Div}_{S,\Q}({U})_{\cpt}$ converging to $\overline{D}=(D,g)$ with respect to the $\overline{B}$-boundary topology for some weak boundary divisor $\overline{B}\in\widehat{\Div}_{S,\Q}(U)_\CM$. If ${D}$ is big, by the continuity of (geometric) volume, see \cref{geometric volume}~\ref{continuity of geometric volume}, $D_j$ is big for $j$ large enough. Hence $G_{\overline{D_j}}$ is well-defined for $j$ large enough. By \cref{lemma:interior point convergence of concave}, we can conclude that every $\lambda\in \Delta(D)^{\circ}$ is in $\Delta(D_j)^{\circ}$ for $j$ large enough. Hence for any $\lambda\in  \Delta(D)^{\circ}$, it makes sense to consider the limits $\limsup\limits_{j\to\infty}G_{\overline{D_j}}(\lambda)$ and $\liminf\limits_{j\to\infty}G_{\overline{D_j}}(\lambda)$.
\end{art}

\begin{theorem}
\label{thm:convergence of concave transforms}
Let $(\overline{D_j})_{j\in\N_{\geq 1}}$ be a sequence in $\widehat{\Div}_{S,\Q}({U})_{\cpt}$ converging to $\overline{D}=(D,g)$ with respect to the $\overline{B}$-boundary topology for some weak boundary divisor $\overline{B}\in\widehat{\Div}_{S,\Q}(U)_\CM$. If $D$ is big, then  
\begin{align}\label{eq:lim of sequence of concave}
\lim_{j\to\infty}G_{\overline{D_j}}=G_{\overline{D}}
\end{align}
pointwise on $\Delta(D)^{\circ}$ (as we discussed in \cref{limits of concave transforms}, the limit makes sense for $j$ large enough).
\end{theorem}
\begin{proof}
We prove the theorem in the following steps.

\vspace{2mm} \noindent
Step 1: {\it If the theorem holds when $\overline{D_j}$ is decreasing or $\overline{D_j}$ is increasing, then it holds in general.} 

\vspace{2mm} \noindent
Let $\overline{B}$ be a weak boundary divisor in $\widehat{\Div}_{S,\Q}(U)_\CM$, and $(\varepsilon_{j})_{j\in\N_{\geq1}}$ a sequence in $\Q_{>0}$ such that $\overline{D}-\varepsilon_j\overline{B}\leq \overline{D_j}\leq \overline{D}+\varepsilon_j\overline{B}$ for all $j\in\N_{\geq 1}$ and $\lim\limits_{j\to\infty}\varepsilon_j=0$. Let $\lambda\in \Delta(D)^\circ$. By \cref{properties of concave transforms}~\ref{monotone}, we have that
\begin{align}\label{inequality of concave by cauchy}
G_{\overline{D}-\varepsilon_j\overline{B}}(\lambda)\leq G_{\overline{D_j}}(\lambda)\leq G_{\overline{D}+\varepsilon_j\overline{B}}(\lambda).
\end{align}
It suffices to show that
\begin{align}\label{eq:limsup of sequence of concave}
G_{\overline{D}}(\lambda)= \limsup\limits_{j\to\infty}G_{\overline{D}+\varepsilon_j\overline{B}}(\lambda)=\liminf\limits_{j\to\infty}G_{\overline{D}-\varepsilon_j\overline{B}}(\lambda).
\end{align}
We only consider the the sequence $(\overline{D}+\varepsilon_j\overline{B})_{j\in\N_{\geq 1}}$, it is similar for $(\overline{D}-\varepsilon_j\overline{B})_{j\in\N_{\geq 1}}$.
Let $(\varepsilon_{j_i})_{i\in\N_{\geq 1}}$ be a subsequence of $(\varepsilon_{j})_{j\in \N_{\geq 1}}$ such that 
\[\lim\limits_{i\to\infty}G_{\overline{D}+\varepsilon_{j_i}\overline{B}}(\lambda)=\limsup\limits_{j\to\infty}G_{\overline{D}+\varepsilon_j\overline{B}}(\lambda).\]
Furthermore, we can assume that $\varepsilon_{j_i} \leq 1/i$ for any $i\in\N_{\geq 1}$. Since the theorem holds for $(\overline{D}+\frac{1}{i}\overline{B})_{j\in\N_{\geq 1}}$ by our assumption, from \cref{properties of concave transforms}~\ref{monotone}, we have that
\[G_{\overline{D}}(\lambda)\leq \limsup\limits_{j\to\infty}G_{\overline{D}+\varepsilon_j\overline{B}}(\lambda)= \lim\limits_{i\to\infty}G_{\overline{D}+\varepsilon_{j_i}\overline{B}}(\lambda) \leq \lim\limits_{i\to\infty}G_{\overline{D}+\frac{1}{i}\cdot\overline{B}}(\lambda)=G_{\overline{D}}(\lambda).\]
This shows the first equality of \eqref{eq:limsup of sequence of concave}. The proof for the second equality of \eqref{eq:limsup of sequence of concave} is similar. Since $\lambda\in \Delta(D)^\circ$ is arbitrary, we have that Step~1 holds.

\vspace{2mm} \noindent

In the following, we fix an integrable function $c\in \mathscr{L}^1(\Omega,\mathcal{A},\nu)$ such that $\int_\Omega c(\omega)\, \nu(d\omega) = 1$. For any $\alpha>0$, we set \[\Delta(D)^\alpha\coloneq\{\lambda\in\Delta(D)\mid G_{\overline{D}}(\lambda)\ge -\alpha\},\]
and denote  $\overline{D}(\alpha)\coloneq \overline{D}+(0,\alpha c)$, similarly for $\overline{D_j}(\alpha)$.

\vspace{2mm} \noindent
Step 2: {\it If $\overline{D_j}$ is decreasing, then the theorem holds.}

\vspace{2mm} \noindent
If $\overline{D_j}$ is decreasing, then for any $\lambda\in\Delta(D)^\circ$, $G_{\overline{D_j}}(\lambda)\geq G_{\overline{D}}(\lambda)$ and $\lim\limits_{j\to\infty}G_{\overline{D_j}}(\lambda)$ exists by \cref{properties of concave transforms}~\ref{monotone}. So it suffices to show that $\lim\limits_{j\to\infty}G_{\overline{D_j}}\leq G_{\overline{D}}$ almost everywhere on $\Delta(D)^\circ$. For any $r\in \Q_{>0}$ and $\alpha\in\R_{>0}$, we set 	
\[{K_{\mathrm{Sing},r}^+}\coloneq\{\lambda\in\Delta(D)^{\circ}\mid \lim_{j\to\infty}G_{\overline{D_j}}(\lambda)>G_{\overline{D}}(\lambda)+r\},\] 
\[\Delta(D)^{\alpha,r}_+\coloneq\Delta(D)^\alpha\cap K_{\mathrm{Sing},r}^+\subset \Delta(D)^\circ.\]
We claim that the measure $\vol_{\R^d}(\Delta(D)^{\alpha,r}_+)$ of $\Delta(D)^{\alpha,r}_+$ is $0$ for any $r\in\Q_{>0}, \alpha\in\R_{>0}$. Indeed, given $r\in\Q_{>0}, \alpha\in\R_{>0}$, for any $\lambda\in\Delta(D)_+^{\alpha,r}$ we have 
\begin{equation}\label{equa:uffvai1}
	G_{\overline{D_j}(\alpha)}(\lambda)\geq \lim_{i\to\infty}G_{\overline{D_i}(\alpha)}(\lambda)>G_{\overline{D}(\alpha)}(\lambda)+r\geq r>0
\end{equation} 
thanks to \cref{properties of concave transforms}~\ref{shifting}. Write $f_{\overline{D}(\alpha)}\coloneq \max\{G_{\overline{D}(\alpha)},0\}$ on $\Delta(D)$ (resp. $f_{\overline{D}_j(\alpha)}\coloneq \max\{G_{\overline{D_j}(\alpha)},0\}$ on $\Delta(D_j)$), then $f_{{\overline{D_j}}(\alpha)}\geq f_{\overline{D}(\alpha)}$ on $\Delta(D)^\circ$. 
For all $\lambda\in\Delta(D)^{\alpha,r}_+\subseteq \Delta(D)^\circ$ and $j\in\N_{\geq 1}$, by \eqref{equa:uffvai1} we have $$f_{\overline{D_j}(\alpha)}(\lambda)=G_{\overline{D_j}(\alpha)}(\lambda), \ \  f_{\overline{D}(\alpha)}(\lambda)=G_{\overline{D}(\alpha)}(\lambda),$$ then
\begin{equation}
	\label{equa:uffvai2}
	f_{\overline{D_j}(\alpha)}(\lambda)-f_{\overline{D}(\alpha)}(\lambda)=G_{\overline{D_j}(\alpha)}(\lambda)-G_{\overline{D}(\alpha)}(\lambda)=G_{\overline{D_j}}(\lambda)-G_{\overline{D}}(\lambda)>r,
\end{equation}
the last inequality holds since $\lambda\in {K_{\mathrm{Sing},r}^+}$.
Using the first claim of \cref{theorem:concavemain}, we have
\begin{align*}
	&\frac{1}{(d+1)!}(\widehat{\vol}(\overline{D_j}(\alpha))-\widehat{\vol}(\overline{D}(\alpha)))\\
	=&\int_{\Delta(D)^{\circ}} \left(f_{\overline{D_j}(\alpha)}(\lambda)-f_{\overline{D}(\alpha)}(\lambda)\right) \, d\lambda+\int_{\Delta(D_j)^{\circ}\backslash \Delta(D)^{\circ}}f_{\overline{D_j}(\alpha)}(\lambda) \, d\lambda \\
	\geq& \int_{\Delta(D)^{\alpha,r}_+} \left(f_{\overline{D_j}(\alpha)}(\lambda)-f_{\overline{D}(\alpha)}(\lambda)\right) \, d\lambda \\
	\geq& r\cdot \vol_{\R^d}(\Delta(D)^{\alpha,r}_+) \geq 0.
\end{align*}
By the continuity of arithmetic volumes, see \cref{theorem:continuityofvolumes}, we have that the left-hand side of the equation above goes to $0$ as $j\to \infty$. Hence the $\vol_{\R^d}(\Delta(D)^{\alpha,r}_+)=0$. This proves our claim. So $K_{\mathrm{Sing},r}^+ = \bigcup\limits_{\alpha\in\Q_{>0}}\Delta(D)^{\alpha,r}_+$ has measure $0$. Note the set of points where the mentioned pointwise convergence fails to happen is exactly given by $\bigcup\limits_{r\in\mathbb{Q}} K_{\mathrm{Sing},r}^+$. So \eqref{eq:lim of sequence of concave} holds almost everywhere on $\Delta(D)^\circ$. Since $\lim\limits_{j\to\infty}G_{\overline{D_j}}$ and $G_{\overline{D}}$ are concave, hence continuous on $\Delta(D)^\circ$, we have that \eqref{eq:lim of sequence of concave} holds pointwise on $\Delta(D)^\circ$, which proves Step~2.

\vspace{2mm} \noindent
Step 3: {\it If $\overline{D_j}$ is increasing, then the theorem holds.}

\vspace{2mm} \noindent	
If $\overline{D_j}$ is increasing, as before, for any $\lambda\in \Delta(D)^\circ$, the limit $\lim\limits_{j\to\infty}G_{\overline{D_j}}(\lambda)$ exists. For any $r\in\Q_{>0}, \alpha\in\R_{>0}$ and $j\in\N_{\geq 1}$, we set
\[ K_{\mathrm{Sing},r}^-\coloneq\{\lambda\in\Delta(D)^{\circ}\mid \lim_{j\to\infty}G_{\overline{D_j}}(\lambda)<G_{\overline{D}}(\lambda)-r\},\]
\[\Delta(D)^{\alpha,r}_-\coloneq\Delta(D)^\alpha\cap K_{\mathrm{Sing},r}^-\subset \Delta(D)^\circ,\]
\[\Delta(D_j)^{\alpha,r}_-\coloneq\Delta(D)^{\alpha,r}_-\cap \Delta(D_j)^\circ\subset \Delta(D_j)^\circ.\]
As before, to show the theorem in this case, it suffices to show that the measure $\vol_{\R^d}(\Delta(D)^{\alpha,r}_-)=0$ for any $r\in\Q_{>0}, \alpha\in\R_{>0}$. 
Given $r\in\Q_{>0}, \alpha\in\R_{>0}$, we can assume that $\alpha\geq r$ (notice that $\Delta(D)_-^{\alpha,r}$ is increasing at $\alpha$). By \cref{lemma:convbodagain} we have that $\Delta(D)^{\alpha,r}_-=\bigcup\limits_{j=1}^\infty \Delta(D_j)^{\alpha,r}_-$.  Notice that $\Delta(D_i)^{\alpha,r} \subset \Delta(D_j)^{\alpha,r}$ if $i\leq j$, then
\begin{align} \label{limits of measure at minus}
	\lim\limits_{j\to\infty}\vol_{\R^d}(\Delta(D_j)^{\alpha,r}_-)=\vol_{\R^d}(\Delta(D)^{\alpha,r}_-).
\end{align}  For any $\lambda\in\Delta(D_j)_-^{\alpha,r}$ we have 
\begin{equation}
	\label{equa:uffvai3}
	G_{\overline{D_j}(2\alpha)}(\lambda)\leq \lim_{i\to\infty}G_{\overline{D_i}(2\alpha)}(\lambda)<G_{\overline{D}(2\alpha)}(\lambda)-r,
\end{equation} 
\begin{align}\label{equa:uffvai4}
	G_{\overline{D}(2\alpha)}(\lambda)-r\ge \alpha-r\geq 0
\end{align}
thanks to \cref{properties of concave transforms}~\ref{shifting}.  Write $f_{\overline{D}(2\alpha)}\coloneq \max\{G_{\overline{D}(2\alpha)},0\}$ on $\Delta(D)$ (resp. $f_{\overline{D_j}(2\alpha)}\coloneq \max\{G_{\overline{D_j}(2\alpha)},0\}$ on $\Delta(D_j)$), then $f_{\overline{D_j}(2\alpha)}\leq f_{\overline{D}(2\alpha)}$ on $\Delta(D_j)^\circ$. 
For all $\lambda\in\Delta(D_j)^{\alpha,r}_-$ and $j\in\N_{\geq 1}$, by \eqref{equa:uffvai3} and \eqref{equa:uffvai4}, we have 
\begin{equation}
	\label{equa:hawal}
	f_{\overline{D_j}(2\alpha)}(\lambda)=\max\{G_{\overline{D_j}(2\alpha)}(\lambda),0\}\le G_{\overline{D}(2\alpha)}(\lambda)-r\leq f_{\overline{D}(2\alpha)}(\lambda)-r
\end{equation}
Now we will apply a similar argument as the decreasing case, more precisely using \cref{theorem:concavemain}, we have
\begin{align*}
	&\frac{1}{(d+1)!}(\widehat{\vol}(\overline{D_j}(2\alpha))-\widehat{\vol}(\overline{D}(2\alpha))\\
	=&\int_{\Delta(D_j)^{\circ}} (f_{\overline{D_j}(2\alpha)}(\lambda)-f_{\overline{D}(2\alpha)}(\lambda)) \,  d\lambda-\int_{\Delta(D)^{\circ}\setminus \Delta(D_j)^{\circ}}f_{\overline{D}(2\alpha)}(\lambda) \, d\lambda\\
	\le& \int_{\Delta(D_j)^{\circ}} (f_{\overline{D_j}(2\alpha)}(\lambda)-f_{\overline{D}(2\alpha)}(\lambda)) \, d\lambda\\
	\le & \int_{\Delta(D_j)^{\alpha,r}_-} (f_{\overline{D_j}(2\alpha)}(\lambda)-f_{\overline{D}(2\alpha)}(\lambda)) \, d\lambda\\
	\leq & -r\cdot \vol_{\R^d}(\Delta(D_j)^{\alpha,r}_-) \leq 0.
\end{align*} 
By the continuity of arithmetic volumes, see \cref{theorem:continuityofvolumes}, and \eqref{limits of measure at minus}, taking $j\to\infty$, we have that $\vol_{\R^d}(\Delta(D)^{\alpha,r}_-) =0$. This implies Step~3. Hence we complete the proof.
\end{proof}

The following proposition is based on the fact that $S$ has the strong tensorial minimal slope property or the Minkowski property (of level $\geq C_0$ for some $C_0\in \R_{\geq0}$). 

\begin{proposition} \label{prop:the volume and concave of arith nef}
Let $\overline{D}=(D,g)\in\widehat{\Div}_{S,\Q}(U)_{\arnef}$. Then the following statements hold:
\begin{enumerate}
	\item \label{prop:hilbert-samuel for arith nef} we have that 
	$$\widehat{\vol}(\overline{D}) = (\overline{D}^{d+1}\mid U)_S.$$
	\item \label{prop:the concave of arith nef is non-negative} if $D$ is big,  $G_{\overline{D}}\geq 0$ on $\Delta(D)^\circ$.
\end{enumerate}
\end{proposition}
\begin{proof} 
We prove the proposition in the following steps.

\vspace{2mm} \noindent
Step 1: {\it If  $S$ has the strong tensorial minimal slope property of level $\geq C_0$ for some $C_0\in \R_{\geq0}$, $U=X$ is projective and $\overline{D}$ is $S$-ample in sense of \cref{global adelic Green functions for proper varieties}, then $\widehat{\mu}_{\min}^{\mathrm{asy}}(\overline{D})> 0$.}

\vspace{2mm} \noindent
Notice that \cite[Theorem~8.8.3]{chen2022hilbert} holds for adelic curves with the strong tensorial property. Indeed, in the proof of \cite[Theorem~8.8.3]{chen2022hilbert}, the condition that $K$ is perfect is only required in the application of \cite[Theorem~7.2.4]{chen2020arakelov} which holds for adelic curves with the strong tensorial property. Since \cite[Theorem~8.8.3]{chen2022hilbert} implies that \cite[Proposition~9.1.2]{chen2022hilbert}, Step~1 holds.

\vspace{2mm} \noindent
Step 2: {\it If  $S$ has the Minkowsi property of level $C_0$ for some $C_0\in \R_{\geq0}$, $U=X$ is projective and $\overline{D}$ is $S$-ample, then then $\widehat{\mu}_{\min}^{\mathrm{asy}}(\overline{D})> 0$.}

\vspace{2mm} \noindent
If $S$ has the Minkowski property for level $C_0$ for some $C_0\in\R_{\geq 0}$, by \cite[Theorem~7.2.1]{chen2020arakelov} and the fact that restriction of $S$-ample on a closed subvariety is again $S$-ample, to show Step~2, it suffices to show the claim: there is $n\in \N_{>0}$ and $0\not\in s\in H^0(X,D)$ such that $\widehat{\deg}(s)> 0$, where $H^0(X,nD)$ is endowed with the sup-norm, see \cref{sup-norm}. By definition of $S$-ampleness, there is $\varepsilon$ such that $$(\overline{D}^{d+1}\mid X)_S\geq \varepsilon D^d (\dim(X)+1)>0.$$ Then \cite[Theorem~B]{chen2022hilbert} implies that $$\widehat{\vol}(\overline{D})\geq \widehat{\vol}_\chi(\overline{D})=(\overline{D}^{d+1}\mid X)_S>0.$$ 
So $\overline{D}$ is big. Hence $\overline{D}$ is \emph{strongly big} in the sense of \cite[Definition~6.4.19]{chen2020arakelov}, i.e. our claim holds. This completed the proof of Step~2.

\vspace{2mm} \noindent
Step 3: {\it If $U=X$ is projective and $\overline{D}\in\widehat{\Div}_{S,\Q}(X)_{\arnef}$, then $\widehat{\vol}(\overline{D}) = (\overline{D}^{d+1}\mid X)_S$. If, furthermore, $D$ is big, then $G_{\overline{D}}\geq 0$.}

\vspace{2mm} \noindent
We first consider the case where $D$ is not big. Then $\overline{D}$ is not big, i.e. $\widehat{\vol}(\overline{D})=0$, by \cite[Proposition~6.4.18]{chen2020arakelov}. On the other hand, by \cite[Proposition~7.10~(v), Proposition~6.34~(1)]{cai2024abstract}, we have that $(\overline{D}^{d+1}\mid X)_S\geq 0$.
Hence we have the inequalities
\[0=\widehat{\vol}(\overline{D})\geq \widehat{\vol}_\chi(\overline{D})=(\overline{D}^{d+1}\mid X)_S\geq0\]
which implies that $\overline{D} =(\overline{D}^{d+1}\mid X)_S =0.$

Assume that $D$ is big. Let $\overline{A}\in N_{S,\Q}^+(X)$ be an $S$-ample divisor on $X$ such that the point corresponding to the admissible flag $Y_\bullet$ is not in $|A|$. Then for any $n\in\N_{\geq 1}$, $\overline{D}+\frac{1}{n}\overline{A}$ is $S$-ample. This implies that $G_{\overline{D}+\frac{1}{n}\overline{A}}\geq 0$ by Step~1 and Step~2. Let $V=X\setminus|A|$, and $i\colon V\to X$ the corresponding open immersion. Notice that $i^*\left(\overline{D}+\frac{1}{n}\overline{A}\right)$ converges to $\overline{D}$ in $\widehat{\Div}_{S,\Q}(V)_\cpt$ with respect to the $\overline{A}$-boundary topology, by \cref{properties of concave transforms}~\ref{birational} and \cref{thm:convergence of concave transforms}, we have that
\[G_{\overline{D}}=G_{i^*\overline{D}}=\lim\limits_{n\to\infty}G_{i^*\left(\overline{D}+\frac{1}{n}\overline{A}\right)}=\lim\limits_{n\to\infty}G_{\overline{D}+\frac{1}{n}\overline{A}}\geq0\]
almost everywhere on $\Delta(D)^\circ$. Since $G_{\overline{D}}$ is continuous on $\Delta(D)^\circ$, so the statement holds in this case.

\vspace{2mm} \noindent
Step 4: {\it If $U$ is quasi-projective and $\overline{D}\in\widehat{\Div}_{S,\Q}(U)_{\arnef}$, then $\widehat{\vol}(\overline{D}) = (\overline{D}^{d+1}\mid U)_S$. If, furthermore, $D$ is big, then $G_{\overline{D}}\geq 0$.}

\vspace{2mm} \noindent
For general case where $U$ is quasi-projective and $\overline{D}\in \widehat{\Div}_{S,\Q}(U)_\arnef$, by \cref{nef become strongly nef after shrinking U} and \cref{properties of concave transforms}~\ref{birational}, after shrinking $U$, we can assume that $\overline{D}$ is the limit of a sequence $(\overline{D_j})_{j\in \N_{\geq 1}}=(D_j,g_j)_{j\in\N_{\geq 1}}\subset N_{S,\Q}'(U)$ with respect to the $\overline{B}$-boundary topology for some  weakly boundary divisor $\overline{B}\in N_{S,\Q}'(U)$. Since $\overline{D_j}\in N_{S,\Q}'(U)$, we can deduce from Step~3 and \cref{properties of concave transforms}~\ref{birational} that $\widehat{\vol}(\overline{D_j})=(\overline{D_j}^{d+1}\mid U)_S$ and $G_{\overline{D_j}}\ge 0$ on $\Delta(D_j)^\circ$. By \cref{theorem:continuityofvolumes}, we have that
\[\widehat{\vol}(\overline{D})=\lim\limits_{j\to\infty}\widehat{\vol}(\overline{D_j})=\lim\limits_{j\to\infty} (\overline{D_j}^{d+1}\mid U)_S =  (\overline{D}^{d+1}\mid U)_S.\]
It $D$ is big, then $D_j$ is big for $j$ large enough. So $G_{\overline{D_j}}\geq 0$ on $\Delta(D_j)^\circ$ by Step~3. By \cref{thm:convergence of concave transforms}, we have that $G_{\overline{D}}\ge 0$ on $\Delta(D)^{\circ}$. 
\end{proof}

\subsection{Arithmetic Hilbert-Samuel formula}
\begin{proposition} \label{proposition:limits property for chi volume}
Let $\overline{D}=(D,g)\in\widehat{\Div}_{S,\Q}({U})_{\cpt}$ and  $(\overline{D_j})_{j\in\N_{\geq 1}}=(D_j, g_j)_{j\in\N_{\geq 1}}$ a sequence in $\widehat{\Div}_{S,\Q}({U})_{\cpt}$ converging to $\overline{{D}}$ with respect to the $\overline{B}$-boundary topology for some weak boundary divisor $\overline{B}\in\widehat{\Div}_{S,\Q}(U)_\CM$ of $U$. Assume that $D_j, D\in \widetilde\Div_{\Q}(U)_{\cpt}$ are big and one of the following condition holds:
\begin{enumerate}
\item \label{lower bound finite for limits property for chi volume} there is $c_0\in\R$ such that $\inf\limits_{\lambda\in\Delta(D_j)^\circ}G_{\overline{D_j}}(\lambda)$, $\inf\limits_{\lambda\in\Delta(D)^\circ}G_{\overline{D}}(\lambda) >c_0$ for any $j\in\N_{\geq 1}$;
\item \label{decreasing for limits property for chi volume} $\overline{D_j}$ is decreasing and $D_j=D$ in $\widetilde\Div_\Q(U)_{\cpt}$ for arbitrary large $j\in\N_{\geq 1}$.
\end{enumerate}
Then
\[\lim\limits_{j\to\infty}\widehat{\vol}^\num_\chi(\overline{D_j})=\widehat{\vol}^\num_\chi(\overline{D}).\]
\end{proposition}
\begin{proof}
\ref{lower bound finite for limits property for chi volume} Let $c_0\in\R$ such that  $\inf\limits_{\lambda\in\Delta(D_j)^\circ}G_{\overline{D_j}}(\lambda)$, $\inf\limits_{\lambda\in\Delta(D)^\circ}G_{\overline{D}}(\lambda) >c_0$ for any $n\geq 1$. We fix an integrable function $c\in\mathscr{L}^1(\Omega,\mathcal{A},\nu)$ such that $c_0=\int_{\Omega}c(\omega)\, \nu(d\omega)$. By \cref{properties of concave transforms}~\ref{shifting}, we have that $G_{\overline{D_j}(-{c})}>0$ on $\Delta(D_j)^\circ$ (resp. $G_{\overline{D}(-{c})}>0$ on $\Delta(D)^\circ$). By  \cref{corol:volsame} and \cref{lemma:refokoun}, we have that
\[\widehat{\vol}_{\chi}^\num(\overline{D_j}) = \widehat{\vol}_{\chi}^\num(\overline{D_j}(-{c}))+c_0(d+1)\vol(D_j)= \widehat{\vol}(\overline{D_j}(-{c}))+c_0(d+1)\vol(D_j).\] Similarly, we have $\widehat{\vol}_\chi^\num(\overline{D}) = \widehat{\vol}(\overline{D}(-{c}))+c_0(d+1)\vol(D)$. By \cref{geometric volume}~\ref{continuity of geometric volume} and \cref{theorem:continuityofvolumes}, we have that 
\begin{align*}
\lim\limits_{j\to\infty}\widehat{\vol}_\chi^\num(\overline{D_j})&=\lim\limits_{j\to\infty}\left(\widehat{\vol}(\overline{D_j}(-c))+c_0(d+1)\vol(D_j)\right)\\
&=\widehat{\vol}(\overline{D}(-c))+c_0(d+1)\vol(D)\\
&=\widehat{\vol}_\chi^\num(\overline{{D}}).
\end{align*}

\ref{decreasing for limits property for chi volume} By \cref{thm:convergence of concave transforms}, the concave transforms $G_{\overline{D_j}}$ decreasingly converges to $G_{\overline{D}}$ pointwise on $\Delta(D)^\circ$. Hence, by the monotone convergence theorem, we have that
\[\lim\limits_{n\to\infty}\widehat{\vol}_\chi^\num(\overline{D_j})=\lim\limits_{n\to\infty}(d+1)!\int_{\Delta(D)^\circ}G_{\overline{D_j}}(\lambda)\, d\lambda=(d+1)!\int_{\Delta(D)^\circ}G_{\overline{D}}(\lambda)\, d\lambda=\widehat{\vol}_\chi^\num(\overline{D}).\]
\end{proof}

We are ready to prove the main result of this section.
\begin{theorem} \label{thm:Hilbert-Samuel formula}
Let $\overline{D}=(D,g)\in \widehat{\Div}_{S,\Q}({U})_{\relnef}^{\arnef}$ with $D$ big. Then
\[\widehat{\vol}_\chi^{\num}(\overline{D})=(\overline{D}^{d+1} \mid U)_S.\]
In particular, if $\widehat{\mu}_{\min}^{\mathrm{asy}}(\overline{D})>-\infty$, then by \cref{theorem:concavemain}, we have that \[\widehat{\vol}_\chi(\overline{D})=(\overline{D}^{d+1} \mid U)_S.\]
\end{theorem}
\begin{proof}
We prove the theorem in the following steps.

\vspace{2mm} \noindent
Step 1: {\it If $\overline{D}\in \widehat{\Div}_{S,\Q}(U)_\arnef$, then $\widehat{\vol}_\chi^\num(\overline{D})=(\overline{D}^{d+1}\mid U)_S$.} 

\vspace{2mm} \noindent
By \cref{prop:the volume and concave of arith nef}~\ref{prop:the concave of arith nef is non-negative} and \cref{theorem:concavemain}, we have that $$\widehat{\vol}_\chi^\num(\overline D)=(d+1)!\int_{\Delta(\overline{D})^\circ}G_{\overline D}(\lambda) \, d\lambda=(d+1)!\int_{\Delta(\overline{D})^\circ}\max\{G_{\overline D}(\lambda),0\} \, d\lambda=\widehat{\vol}(\overline{D}).$$
By \cref{prop:the volume and concave of arith nef}~\ref{prop:hilbert-samuel for arith nef}, we have that $\widehat{\vol}(\overline{D})=(\overline{D}^{d+1}\mid U)_S$. This proves Step~1. 

\vspace{2mm} \noindent
Step 2: {\it If there is ${c}\in\mathscr{L}^1(\Omega,\mathcal{A},\nu)$ such that $\overline{D}(c)\in\widehat{\Div}_{S,\Q}(U)_\arnef$, then  $\widehat{\vol}_\chi^\num(\overline{D})=\overline{D}^{d+1}$.}

\vspace{2mm} \noindent

Let ${c}\in\mathscr{L}^1(\Omega,\mathcal{A},\nu)$ be the one in the assumption of Step~2. By \cref{properties of concave transforms}~\ref{shifting} and \cref{lemma:refokoun}, we have that $$\widehat{\vol}_{\chi}^\num(\overline D(c))=\widehat{\vol}_{\chi}^\num(\overline D) + (d+1)\int_{\Omega}c(\omega)\,\nu(d\omega)\cdot\vol(D).$$ On the other hand, by the properties of arithmetic intersection number on quasi-projective varieties (see \cite[Theorem~7.22]{cai2024abstract}), we can induce that  $$(\overline{D}({c})^{d+1}\mid U)_S = (\overline{D}^{d+1}\mid U)_S+(d+1)\int_{\Omega}c(\omega)\,\nu(d\omega)\cdot D^d.$$ By \cite[Theorem~5.2.2]{yuan2021adelic}, we have $\vol(D)=D^{d}$. Combining these with Step~1, we easily see that Step~2 holds.

\vspace{2mm} \noindent 
Step 3: {\it Theorem holds.} 	

\vspace{2mm} \noindent
Let $\overline{D'}= ({D},g')\in\widehat{\Div}_{S,\Q}({U})_{\arnef}$. By \cref{nef become strongly S-nef after shrinking U} and \cref{properties of concave transforms}~\ref{birational}, we can shrink $U$ and assume that there is a $\overline{B}\in N_{S,\Q}(U)$ such that $\overline{D}\in N_{S,\Q}(U)^{d_{\overline B}}$ (resp. $\overline{D'}\in N_{S,\Q}'(U)^{d_{\overline B}}$). We can assume that $g'\geq g$, otherwise, by \cite[Lemma~10.3~(ii)]{cai2024abstract}, we can replace $g'$ by $\max\{g',g\}$. By \cite[Lemma~10.6]{cai2024abstract}, there is an increasing sequence of $\nu$-integrable functions $(c_m)_{m\in \N_{\geq 1}}\subset \mathscr{L}^1(\Omega,\mathcal{A},\nu)$ such that $\overline{D_m}=(D,g_m')\coloneq (D,\max\{g,g'-c_m\})\in \widehat{\Div}_{S,\Q}({U})_{\relsnef}$ and $\overline{D_m}$ converges decreasingly to $\overline{D}$ with respect to the $\overline{B}$-boundary topology. By \cite[Theorem~10.9~(v)]{cai2024abstract}, we have that
\begin{align}\label{proof of main 1}
	\lim\limits_{m\to\infty}(\overline{D_m}^{d+1}\mid U)_S=(\overline{D}^{d+1}\mid U)_S.
\end{align}
On the other hand, by \cref{proposition:limits property for chi volume}~\ref{decreasing for limits property for chi volume}, we have that
\begin{align} \label{proof of main 2}
	\lim\limits_{m\to\infty}\widehat{\vol}_{\chi}^\num(\overline{D_m}) = \widehat{\vol}_{\chi}^\num(\overline{D}).
\end{align}
Notice that $\overline{D_m}(c_m)\geq \overline{D'}$ and $\overline{D'}\in \widehat{\Div}_{S,\Q}({U})_{\arnef}$, so $\overline{D_m}(c_m)\in \widehat{\Div}_{S,\Q}({U})_{\arnef}$ by \cite[Lemma~10.3]{cai2024abstract}.
 By Step~2, we have that
\[\widehat{\vol}_\chi^\num(\overline{D_m})=(\overline{D_m}^{d+1}\mid U)_S.\]
Take $m\to\infty$ on both sides and by \eqref{proof of main 1}, \eqref{proof of main 2}, we prove the theorem.
\end{proof}
\begin{remark}
\cref{thm:Hilbert-Samuel formula} gives another definition of arithmetic auto-intersection of relatively nef compactified line bundle in \cite{cai2024abstract}  when the underlying geometry line bundles are big by realising it as the integral of a concave function in analogy to the toric setting. In particular, the Hodge bundles on Shimura varieties and the Jacobi line bundles on the universal elliptic surface are \emph{relatively nef compactified line bundles}.
\end{remark}
The following corollary follows directly.

\begin{corollary}\label{cor:big if intersection number>0}
Let $\overline{D}=(D,g)\in\widehat{\Div}_{S,\Q}(U)_{\relnef}^{\arnef}$ with $D$ big. Then 
\[\widehat{\vol}(\overline{D})\geq (\overline{D}^{d+1}\mid U)_S.\]
In particular, if $\overline{D}^{d+1}>0$, then $\overline{D}$ is big.
\end{corollary}

We have the following height inequality generalizing \cite[Theorem~5.3.5~(2)]{yuan2021adelic}.
\begin{proposition}\label{prop:height inequality 2}
	Let $\overline{D}=(D,g), \overline{M}\in \widehat{\Div}_{S,\Q}(U)_\cpt$. If $\overline{D}\in \widehat{\Div}_{S,\Q}(U)_\relnef^\arnef$ and $D\in\widetilde{\Div}_{\Q}(U)_\cpt$ is big, then for any $c>0$, there is $\epsilon\in \Q_{>0}$ and a Zariski open, dense subset $V$ of $U$ such that
	\[h_{\overline{D}}(x)\geq \epsilon h_{\overline{M}}(x)+\frac{(\overline{D}^{d+1}\mid U)_S}{(d+1)D^{d}}-c, \ \ \text{ for any $x\in V(\overline{K})$.}\]
\end{proposition}
\begin{proof}
	We can assume that $(\overline{D}^{d+1}\mid U)_S>-\infty$. Let $c$ be an arbitrary positive number and $\widetilde{c}\in \mathscr{L}^1(\Omega,\mathcal{A},\nu)$ such that
	\[\int_\Omega\widetilde{c}(\omega)\, \nu(d\omega) = -\frac{(\overline{D}^{d+1}\mid U)_S}{(d+1)D^{d}}+c.\]
	Since $D$ is nef and big, we have that $\vol(D)=D^d>0$ by \cite[Theorem~5.2.2]{yuan2021adelic} and
	\[(\overline{D}(\widetilde{c})^{d+1}\mid U)_S= (\overline{D}^{d+1}\mid U)_S+(d+1)\cdot\int_\Omega\widetilde{c}(\omega)\, \nu(d\omega)\cdot D^{d} = (d+1)cD^{d}>0.\]
	Then $\overline{D}(\widetilde{c})$ is big by \cref{cor:big if intersection number>0}. By \cref{theorem:heightinequality}~\ref{theorem:heightinequality1}, there is $\epsilon>0$ and a Zariski open, dense subvariety $V$ of $U$ such that
	\[h_{\overline{D}(\widetilde{c})}=h_{\overline{D}}(x)+\int_\Omega\widetilde{c}(\omega)\, \nu(d\omega)\geq \epsilon h_{\overline{M}}(x), \ \ \text{ for any $x\in V(\overline{K})$.}\]
	This completes the proof.
	\end{proof}
\section{Equidistribution on quasi-projective varieties over classical adelic curves and comparison with classical volumes}
\label{section: equidistribtion}
{In our final section, we obtain an equidistribution result for compactified metrized divisors on quasi-projective varieties over the classical global fields. 
We fix $C$ to be the spectrum of the ring of integers of a number field $K$ or a regular projective curve over a field $k$ with function field $K$, and denote the corresponding adelic curve given in \cref{adelic curve defined by number field and projective curves} by $S=(K,\Omega,\mathcal{A},\nu)$. Notice that $S$ has the strong tensorial minimal slope property of level $\geq C_0$ for some $C_0\in\R_{\geq 0}$ when $K$ is a number field; $S$ has the Minkowski property of level $\geq C_0$ for some $C_0\in\R_{\geq 0}$ when $K$ is the function field of a curve, see \cref{def:tensorialproperty}. 
We also fix a $d$-dimensional normal quasi-projective variety $U$ over $K$ and an admissible flag $Y_\bullet$ of $U_{\overline{K}}$.}

\subsection{Compactified $S$-metrized $\YZ$-divisors}
Recall the space of compactified $S$-metrized divisors $\widehat{\Div}_{S,\Q}(U)_\cpt$, the relatively nef cone $\widehat{\Div}_{S,\Q}(U)_{\relnef}$ and the arithmetically nef cone $\widehat{\Div}_{S,\Q}(U)_{\arnef}$ defined in \cref{def:CM divisors on non-proper}.

We will prove the arithmetic Hilbert-Samuel formula for the relatively nef compactified $S$-metrized $\YZ$-divisors which form a subset $\widehat{\Div}_{S,\Q}(U)_\relnef^\YZ$ of $\widehat{\Div}_{S,\Q}(U)_\relnef$. To introduce such divisors, we follow the process similar as \cref{def:CM divisors on non-proper}.

\begin{art} \label{model metrics and model Green functions}
	Let $X$ be a projective variety over $K$. A \emph{projective $C$-model} of $X$ is a flat, integral scheme $\mathcal{X}$ projective over $C$ with generic fiber $X$. We say $(D,g)\in \widehat{\Div}_{S,\Q}(X)$ is a \emph{model $S$-metrized} (\emph{$\Q$-})\emph{divisor} if it is induced by a pairing $(\mathcal{D},g_{\infty})$ on a projective $C$-model $\mathcal{X}$ of $X$ where $\mathcal{D}$ is a $\Q$-Cartier divisor on $\mathcal{X}$ and $g_{\infty}=\{g_\omega\}_{\omega\in\Omega_\infty}$ is a family of smooth Green functions for $D$ on $X^\an_\omega$ for each $\omega\in\Omega_\infty$. Set $\widehat{\Div}_{S,\Q}({X})_\mo$ to be the subspace of model $S$-metrized divisors in $\widehat{\Div}_{S,\Q}({X})$.
	We denote  \[N_{S,\Q}(X)_\mo\coloneq \widehat{\Div}_{S,\Q}({X})_\mo \cap N_{S,\Q}(X), \ \  N_{S,\Q}'(X)_\mo\coloneq \widehat{\Div}_{S,\Q}({X})_\mo \cap N_{S,\Q}'(X)\] (see \cref{global adelic Green functions for proper varieties} for $N_{S,\Q}(X), N_{S,\Q}'(X)$). 
\end{art}

\begin{art} \label{definition of compactified YZ-divisors}
	The space of \emph{model $S$-metrized} (\emph{$\Q$-})\emph{divisors} on $U$ is defined as the limit 
	\[ \widehat{\Div}_{S,\Q}(U)_{\mathrm{mo}}\coloneq\varinjlim_{X}\widehat{\Div}_{S,\Q}(X)_{\mathrm{mo}},\]
	where $X$ runs through all projective $K$-models of $U$. Similarly, we also denote
	\[N_{S,\Q}(U)_\mo\coloneq \varinjlim_{X}N_{S,\Q}(X)_\mo, \ \ N_{S,\Q}'(U)_\mo\coloneq \varinjlim_{X}N_{S,\Q}'(X)_\mo\]
	Let $T$ be the subset of weak boundary divisors in $\widehat{\Div}_{S,\Q}(U)_\mo$. For a weak boundary divisor $\overline{B}\in T$, we have a \emph{$\overline{B}$-boundary topology} on $\widehat{\Div}_{S,\Q}(U)_\mo$, see \cref{def:CM divisors on non-proper},
	and denote by $\widehat{\Div}_{S,\Q}(U)^{d_{\overline{B}}}_\mo$ (resp. $N_{S,\Q}(U)_\mo^{d_{\overline{B}}}$, resp. $N_{S,\Q}'(U)_\mo^{d_{\overline{B}}}$) the completion of $\widehat{\Div}_{S,\Q}(U)_\mo$ (resp. $N_{S,\Q}(U)_\mo$, resp. $N_{S,\Q}'(U)_\mo$) with respect to the $\overline{B}$-boundary topology. The space of \emph{compactified $S$-metrized $\YZ$-divisors} is defined as 
	\[\widehat{\Div}_{S,\Q}(U)_{\cpt}^{\YZ}\coloneq\varinjlim_{\overline{B}\in T}\widehat{\Div}_{S,\Q}(U)^{d_{\overline{B}}}_\mo.\]
	We also set
	\[\widehat{\Div}_{S,\Q}(U)_{\relsnef}^\YZ\coloneq\varinjlim_{\overline{B}\in T}N_{S,\Q}(U)_\mo^{d_{\overline{B}}}, \ \ \widehat{\Div}_{S,\Q}(U)_{\arsnef}^\YZ\coloneq\varinjlim_{\overline{B}\in T}N_{S,\Q}'(U)_\mo^{d_{\overline{B}}}\]
	\[\widehat{\Div}_{S,\Q}(U)_\relint^\YZ\coloneq \widehat{\Div}_{S,\Q}(U)^\YZ_\relsnef-\widehat{\Div}_{S,\Q}(U)^\YZ_\relsnef,\] 
	\[\widehat{\Div}_{S,\Q}(U)_\arint^\YZ\coloneq \widehat{\Div}_{S,\Q}(U)^\YZ_\arsnef-\widehat{\Div}_{S,\Q}(U)^\YZ_\arsnef,\]
	and denote by $\widehat{\Div}_{S,\Q}(U)_\relnef^\YZ$ (resp. $\widehat{\Div}_{S,\Q}(U)_\arnef^\YZ$)
	the closure of $\widehat{\Div}_{S,\Q}(U)_\relsnef^\YZ$ (resp. $\widehat{\Div}_{S,\Q}(U)_\arsnef^\YZ$) in $\widehat{\Div}_{S,\Q}(U)_\relint^\YZ$ (resp. $\widehat{\Div}_{S,\Q}(U)_\arint^\YZ$) with respect to the finite subspace topology. The elements of $\widehat{\Div}_{S,\Q}(U)_\relint^\YZ$ (resp. $\widehat{\Div}_{S,\Q}(U)_\relsnef^\YZ$, $\widehat{\Div}_{S,\Q}(U)_\relnef^\YZ$, $\widehat{\Div}_{S,\Q}(U)_\arint^\YZ$, $\widehat{\Div}_{S,\Q}(U)_\arsnef^\YZ$, $\widehat{\Div}_{S,\Q}(U)_\arnef^\YZ$) are called \emph{relatively integrable} (resp. \emph{strongly relatively nef}, \emph{relatively nef}, \emph{arithmetically integrable}, \emph{strongly arithmetically nef}, \emph{arithmetically nef}) compactified $S$-metrized $\YZ$-divisors of $U$.
	We have that $$\widehat{\Div}_{S,\Q}(U)^\YZ_\cpt\subset \widehat{\Div}_{S,\Q}(U)_\cpt\subset \widehat{\Div}_{S,\Q}(U)$$ and a forgetting homomorphism
	\[\widehat{\Div}_{S,\Q}(U)_\cpt^\YZ\to \widetilde{\Div}_{\Q}(U)_\cpt,\]
	where $\widetilde{\Div}_{\Q}(U)_\cpt$ is the space of compactified (geometric) divisors of $U$ defined in \cref{geometric intersection number}. As \cref{def:CM divisors on non-proper}, from now on we write an element $\overline{D}\in\widehat{\Div}_{S,\Q}(U)_\cpt^\YZ$ as $(D,g)$ with $D\in\widetilde{\Div}_{\Q}(U)_\cpt$ and $g$ an $S$-measurable, locally $S$-bounded $S$-Green function for $D|_U$.
\end{art}
\begin{remark}\label{volume results for YZ-divisors}
	Our definition above is slightly different from Yuan-Zhang's definition given in \cite[\S 2.4, \S 2.5.3]{yuan2021adelic} when $K$ is a number field, their space of adelic (compactified) divisors on $U$ is a subspace of $\widehat{\Div}_{S,\Q}(U)_\cpt^\YZ$, see \cite[Section 8]{cai2024abstract}. However, it is not hard to show that the results for (classical) arithmetic volumes (see \cref{definition:yuanzhangseries} below) in \cite[Section 5]{yuan2021adelic} can be generalized to $\widehat{\Div}_{S,\Q}(U)_\cpt^\YZ$.
\end{remark}

\begin{remark}
	Obviously, we have that $\widehat{\Div}_{S,\Q}(U)^\YZ_{\relnef}\subset \widehat{\Div}_{S,\Q}(U)_\cpt^\YZ\cap\widehat{\Div}_{S,\Q}(U)_{\relnef}$, it is natural to ask if the converse holds or not, i.e. if $\overline{D}$ is a relative nef compactified $S$-metrized divisor which is also a compactified $S$-metrized $\YZ$-divisor, is $\overline{D}$ a compactified $S$-metrized $\YZ$-divisor?
\end{remark}

\subsection{Classical arithmetic volume}
In this subsection, we compare the arithmetic volume of a compactified $S$-metrized $\YZ$-divisor $\overline{D}\in\widehat{\Div}_{S,\Q}(U)$ defined in \cref{def:volume} and the classical one defined in \cite[Definition~5.1.3]{yuan2021adelic}. 
The comparison relies on the results given in \S \ref{over global field}. Recall the definition of purification in \cref{purification}, and the notion of generically trivial $S$-normed vector spaces, coherent $S$-normed vector space in \cref{def:adelicvectorspaces}.
The space of global sections of an adelic line bundle over a projective variety with sup-norms in \cite{chen2020arakelov} is not necessarily generically trivial. It is similar for quasi-projective case. Recall the space of auxiliary sections $H^0_+(U,\overline{D})$ of a compactified $S$-metrized divisor defined in \cref{def:volume}, it is an adelic vector bundle on $S$ by \cref{H+ is adelic}. We introduce a subspace of $H^0_+(U,\overline{D})$ which will naturally give rise to coherent adelic vector bundles.

\begin{definition}
	\label{definition:yuanzhangseries}
	Let $\overline{D}\in {\widehat{\Div}_{S,\Q}(U)_\cpt}$. We define
	\[H^0_\YZ(U,\overline{D})\coloneq \{s\in H^0_+(U,\overline{D})\mid \|s\|_{\sup,\omega}\le 1\ \text{for almost all}\ \omega\in\Omega\},\]
	where $\metr_{\sup}=(\metr_{\sup, \omega})_{\omega\in\Omega}$ is the family of sup-norms on $H^0_+(U,\overline{D})$, see \cref{sup-norm}.
	Since $(H^0_+(U,\overline{D}), \metr_{\sup})$ is an adelic vector bundle, we have that $(H^0_\YZ(U,\overline{D}), \metr_{\sup})$ is a coherent adelic vector bundle. Set $\overline{V_{\YZ,m}}\coloneq (H^0_\YZ(U,m\overline{D}), \metr_{\sup,m})$, where $\metr_{\sup,m}$ is the family of sup-norms on $H^0_\YZ(U,m\overline{D})$. We get a graded $K$-algebra of coherent adelic vector bundles $\overline{V_{\YZ,\bullet}}=\{\overline{V_{\YZ,m}}\}_{m\in\N}$. We define the \emph{classical arithmetic volume} of $\overline{D}$ as 
	\[\widehat{\vol}^\YZ(\overline{D})\coloneq \widehat{\vol}^\YZ(\overline{V_{\YZ,\bullet}})=\limsup_{m\to\infty} \frac{\widehat{h}^0(\overline{V_{m,\YZ}})}{m^{d+1}/(d+1)!}\]
	(see \cref{def:classical arithvol} for the notion $\widehat{\vol}^\YZ(\cdot)$).
	By \cref{lemma:adelicvectorspaceisvectorbundle}~\ref{purification is adelic}, the purification $\overline{V_{\YZ,\pur,m}}$ of  $\overline{V_{\YZ,m}}$ is a pure generically trivial adelic vector bundle. Set $\overline{V_{\YZ,\pur,\bullet}}\coloneq\{\overline{V_{\YZ,\pur,m}}\}_{m\in\N}$. We define the \emph{classical arithmetic $\chi$-volume} of $\overline{D}$ as 
	\[\widehat{\vol}_\chi^\YZ(\overline{D})\coloneq \limsup_{m\to\infty} \frac{\chi(\overline{V_{m,\YZ,\pur}})}{m^{d+1}/(d+1)!}\]
	(see \cref{def:classical arithvol} for the notion $\widehat{\vol}^\YZ_\chi(\cdot)$).
	Notice that $\widehat{\vol}^\YZ(\overline{D})=\widehat{\vol}^\YZ(\overline{V_{\YZ,\pur,\bullet}})$ by \cref{rmk: small section of purification}.
\end{definition}
\begin{remark} \label{generize YZ limits of vol}
	Let $\overline{D}\in\widehat{\Div}_{S,\Q}(U)_\cpt^\YZ$. Notice that our definition of $\widehat{\vol}^\YZ(\overline{D})$ coincide with the one in \cite[Definition~5.1.3~(2)]{yuan2021adelic}. Following a similar statement as \cite[Theorem~5.2.1]{yuan2021adelic}, we know that the limit $\widehat{\vol}^\YZ(\overline{D})=\lim\limits_{m\to\infty} \frac{\widehat{h}^0(\overline{V_{m,\YZ}})}{m^{d+1}/(d+1)!}$ exists and if a sequence of model $S$-metrized divisor $(\overline{D_j})_{j\in \N_{\geq 1}}\subset \widehat{\Div}_{S,\Q}(U)_\cpt^\YZ$ converges to $\overline{D}$ with respect to the $\overline{B}$-boundary topology for some weak boundary divisor $\overline{B}$ in $\widehat{\Div}_{S,\Q}(U)_\mo$, then
	\[\widehat{\vol}^\YZ(\overline{D})=\lim\limits_{j\to\infty}\widehat{\vol}^\YZ(\overline{D_j}).\] 
\end{remark}
\begin{example} \label{example:classicalcase}
	Assume that $K$ is a number field. Let $X$ be a projective variety over $K$, $\mathcal{X}$ a projective $C$-model of $X$ with the structure morphism $\varphi\colon \mathcal{X}\to C$, and $\overline{\mathcal{D}}=(\mathcal{D}, g_{\infty})$ a metrized divisor on $\mathcal{X}$ corresponding to a line bundle with smooth metrics. We can assume that $\mathcal{X}$ is integrally closed in $X$. Let $H^0(\mathcal{X},\mathcal{D})$
	be the space of global sections of $\mathcal{L}\coloneq\OO_{\mathcal{X}}(\mathcal{D})$ which is an $\OO_K$-module. Then $H^0(X,D)=H^0(\mathcal{X},\mathcal{D})\otimes_{\OO_K}K$. We consider an $S$-normed $\metr_{\overline{\mathcal{D}}}$ on $H^0(X,D)$ defined as follows: 
	\begin{itemize}
		\item For any $\omega\in \Omega_\infty$ and $s\in H^0(X,D)\otimes_KK_\omega$, we set
		\[\|s\|_{\overline{\mathcal D},\omega}\coloneq \|s\|_{\sup,\omega}=\sup\limits_{x\in X_\omega^\an}\|s(x)\|_{\omega}.\]
		\item For any $\omega\in\Omega_{\mathrm{fin}}$ and $s\in H^0(X,D)\otimes_KK_\omega$, we set 
		\[\|s\|_{\overline{\mathcal D},\omega}\coloneq\inf\{|\alpha|_\omega\mid \alpha\in K_\omega^\times, s\in \alpha (H^{0}(\mathcal{X},\mathcal{D})\otimes_{\OO_K}K_\omega^\circ)\}.\]
	\end{itemize} 	
	On the other hand, $\overline{\mathcal{D}}$ determines an element $\overline{D}=(D,g)\in \widehat{\Div}_{S,\Q}(X)_\mo$, and we have a family of sup-norms $\metr_{\sup}=\{\metr_{\sup,\omega}\}_{\omega\in\Omega}$ on $H^0(X,D)$, see \cref{sup-norm}. Then the following statements hold:
	\begin{enumerate1}
		\item \label{ex:pure} The $S$-norm vector space $(H^0(X, D), \metr_{\overline{\mathcal{D}}})$ is a pure generically trivial adelic vector bundle.
		\item \label{ex:sup eq pure almost} The purification of  $\metr_{\sup}$ is $\metr_{\overline{\mathcal D}}$, and $\metr_{\sup,\omega}=\metr_{\overline{\mathcal{D}},\omega}$ for all but finitely many $\omega\in\Omega$ (in fact, $\metr_{\sup,\omega} = \metr_{\overline{\mathcal{D}},\omega}$ if the special fiber of $\mathcal{X}$ over $\omega\in \Omega_{\mathrm{fin}}$ is reduced). In particular,  $H^0_\YZ(X,\overline{D})=H^0_+(X,\overline{D}) = H^0(X,D)$, and $(H^0(X, D), \metr_{\sup})$ is a generically trivial adelic vector bundle.
		\item \label{ex:volume and yz volume} Recall the arithmetic volume $\widehat{\vol}(\overline{D})$ and arithmetic $\chi$-volume $\widehat{\vol}_\chi(\overline{D})$ of $\overline{D}$ defined in \cref{def:arithvol}. Write $\overline{V_{\overline{\mathcal{D}},\bullet}}=\{(H^0(X,mD),\metr_{m\overline{\mathcal D}})\}_{m\in\N}$. Then
		\[\widehat{\vol}(\overline{D})=\widehat{\vol}^\YZ(\overline{D})=\widehat{\vol}^\YZ(\overline{V_{\overline{\mathcal D},\bullet}})= \widehat{\vol}(\overline{V_{\overline{\mathcal D},\bullet}}),\]
		\[\widehat{\vol}_\chi(\overline{D})=\widehat{\vol}_\chi^\YZ(\overline{D})=\widehat{\vol}^\YZ_\chi(\overline{V_{\overline{\mathcal D},\bullet}})= \widehat{\vol}_\chi(\overline{V_{\overline{\mathcal D},\bullet}}).\]
		Notice that $\widehat{\vol}^\YZ(\overline{V_{\overline{\mathcal D},\bullet}})$ (resp. $\widehat{\vol}^\YZ_\chi(\overline{V_{\overline{\mathcal D},\bullet}})$) is exactly the arithmetic volume (resp. arithmetic $\chi$-volume) of $\overline{\mathcal{D}}$ defined \cite[Section~4]{moriwaki2008continuity} (resp. \cite[\S 1.3]{yuan2009volumes}), see also \cite[Definition~3.14]{burgos2016arithmetic}. 
	\end{enumerate1}
	When $K$ is a function field of a curve over a field $k$, with the same notation, the module $H^0(\mathcal{X}, \mathcal{D})$ is replaced by $\varphi_*\OO_{\mathcal{X}}(\mathcal{D})$  on $C$ and we can define $\metr_{\overline{D}}$ on $H^0(X,D)$ similarly. Then the statements \ref{ex:pure}, \ref{ex:sup eq pure almost}, \ref{ex:volume and yz volume} also hold.
	
	We will show \ref{ex:pure}, \ref{ex:sup eq pure almost}, \ref{ex:volume and yz volume}. For \ref{ex:pure}, notice that $H^{0}(\mathcal{X},\mathcal{D})\otimes_{\OO_K}K_\omega^\circ$ is a lattice of $H^{0}({X},{D})\otimes_{K}K_\omega$ in the sense of \cite[Definition~1.1.23]{chen2020arakelov}, then $(H^0(X, D), \metr_{\overline{\mathcal{D}}})$ is pure by \cref{rmk: pure at non-archimedean}. 
	Write $\varphi\colon \mathcal{X}\rightarrow C$ the structure morphism. Since $\varphi_*\OO_{\mathcal{X}}(\mathcal{D})$ is free locally around the generic point of $C$, by choosing an open subset $W\subset C$ and an $\OO_W$-basis $\mathbf{e}$ of $\varphi_*\OO_{\mathcal{X}}(\mathcal{D})$ on $W$, it can be seen easily that $(H^0(X, D), \metr_{\overline{\mathcal{D}}})$ is generically trivial given by the basis $\mathbf{e}$. 
	
	For \ref{ex:sup eq pure almost}, by \cite[Proposition~6.3]{boucksom2021spaces}, for any $\omega\in\Omega$, if the special fiber of $\mathcal{X}$ over $\omega\in \Omega_{\mathrm{fin}}$ is reduced, then $\metr_{\overline{\mathcal{D}},\omega}=\metr_{\sup,\omega}$. Hence $\metr_{\overline{\mathcal{D}},\omega}=\metr_{\sup,\omega}$ for all but finitely many $\omega\in\Omega$, this implies that $H^0_\YZ(X,\overline{D})=H^0_+(X,\overline{D}) = H^0(X,D)$, and $(H^0(X,D), \metr_{\sup})$ is a generically trivial adelic vector bundle. Moreover, since $\mathcal{X}$ is assumed to be integrally closed in $X$, 	by \cite[Lemma~3.3.3]{yuan2021adelic}, we have that
	\[H^0(\mathcal{X},{\mathcal D})=\{s\in H^0(X,D)\mid \|s\|_{\sup,\omega}\leq 1 \text{ for any $\omega\in\Omega_{\mathrm{fin}}$}\}.\]
	Since $(H^0(X,D), \metr_{\sup})$ is coherent, by \cite[Proposition~4.4.2]{chen2020arakelov}, for any $\omega\in\Omega_{\mathrm{fin}}$, we have that
	\[H^0(\mathcal{X},{\mathcal D})\otimes _{\OO_K}K^\circ_\omega=\{s\in H^0(X,D)\otimes_KK_\omega\mid \|s\|_{\sup,\omega}\leq 1\},\]
	this implies that the purification of $\metr_{\sup}$ is $\metr_{\overline{\mathcal{D}}}$ by the definitions.
	
	For \ref{ex:volume and yz volume}, by \ref{ex:sup eq pure almost} and \cref{rmk: small section of purification}, we have that $$\widehat{\vol}^\YZ(\overline{D})=\widehat{\vol}^\YZ(\overline{V_{\overline{\mathcal D},\bullet}}), \ \  \widehat{\vol}_\chi^\YZ(\overline{D})=\widehat{\vol}_\chi^\YZ(\overline{V_{\overline{\mathcal D},\bullet}}).$$ By \ref{ex:pure} and \cref{arith volume and classical arith volume}, we have that $$\widehat{\vol}^\YZ(\overline{V_{\overline{\mathcal D},\bullet}})= \widehat{\vol}(\overline{V_{\overline{\mathcal D},\bullet}}), \ \  \widehat{\vol}^\YZ_\chi(\overline{V_{\overline{\mathcal D},\bullet}})= \widehat{\vol}_\chi(\overline{V_{\overline{\mathcal D},\bullet}}).$$ By \cite[Theorem~7.5.9]{chen2020arakelov}, the graded $K$-algebra of generically trivial adelic vector bundle $\{(H^0(X,m\overline{D}), \metr_{m,\sup})\}_{m\in\N}$ is asymptotically pure, where $\metr_{m,\sup}$ is the sup-norm on $H^0(X,m\overline{D})$. By \ref{ex:sup eq pure almost} and \cref{arith volume and classical arith volume} again, we have that \[\widehat{\vol}(\overline{D})=\widehat{\vol}^\YZ(\overline{D}), \ \ \widehat{\vol}_\chi(\overline{D})=\widehat{\vol}^\YZ_\chi(\overline{D}).\] This completes the proof of \ref{ex:volume and yz volume}.
\end{example}
Finally we can state the main result of this subsection which shows that the arithmetic volumes and the classical arithmetic volumes coincide.
\begin{proposition}
	\label{theorem:consistentvolumedef}
	Let $\overline{D}\in \widehat{\Div}_{S,\Q}(U)_\cpt$. Then 
	\[\widehat{\vol}^\YZ(\overline{D})\leq \widehat{\vol}(\overline{D}).\]
	The equality holds if $\overline{D}\in \widehat{\Div}_{S,\Q}(U)^\YZ_\cpt$.
\end{proposition}
\begin{proof}
	Let $\overline{V_\bullet}=\{\overline{V_m}\}_{m\in\N}\coloneq \{(H^0_+(X,\overline{D}),\metr_{\sup,m})\}_{m\in\N}$ (resp. $\overline{V_{\YZ,\bullet}}=\{\overline{V_{\YZ,m}}\}_{m\in\N}\coloneq \{(H^0_\YZ(X,\overline{D}),\metr_{\sup,m})\}_{m\in\N}$), where $\metr_{\sup,m}$ is the family of sup-norms on $H^0_+(U,m\overline{D})$. Let $\overline{V_{\YZ,\pur,m}}$ be the purification of $\overline{V_{\YZ,m}}$, and $\overline{V_{\YZ,\pur,\bullet}}\coloneq \{\overline{V_{\YZ,\pur,m}}\}_{m\in\N}$. By Remark~\ref{rmk: small section of purification} and \cref{arith volume and classical arith volume}, we have that $$\widehat{\vol}^{\YZ}(\overline{D})=\widehat{\vol}^{\YZ}(\overline{V_{\YZ,\pur,\bullet}})=\widehat{\vol}(\overline{V_{\YZ,\pur,\bullet}}).$$ Note that there is a canonical norm-contractive injection of adelic vector bundles
	\[f_m\colon\overline{V_{\YZ,\pur,m}}\xhookrightarrow{} \overline{V_m},\]
	i.e. the operator norms $\|f_m\|_\omega\le 1$ at all $\omega\in\Omega$ since $\metr_\omega\le \metr_{\omega,\pur}$. It is easy to deduce then from \cite[Proposition~4.3.18]{chen2020arakelov} that 
	\[\widehat{\deg}_+(\overline{V_{\YZ,\pur,m}})\le \widehat{\deg}_+(\overline{V_m})\]
	which clearly shows that $\widehat{\vol}(\overline{V_{\YZ,\pur,\bullet}})\le \widehat{\vol}(\overline{D})$. This proves the inequality. 
	
	If $\overline{D}\in \widehat{\Div}_{S,\Q}(U)^\YZ_\cpt$, we choose any model $S$-metrized divisor $\overline{D'}$ on some projective model of $U$ with $\overline{D}\leq \overline{D'}$. Then by \cref{example:classicalcase}~\ref{ex:volume and yz volume}, we have that $\widehat{\vol}(\overline{D'})=\widehat{\vol}^\YZ(´\overline{D'})$.  However the effectivity relation $\overline{D}\le \overline{D'}$ shows by arguments just like before that 
	\[\widehat{\vol}^\YZ(\overline{D})\le\widehat{\vol}(\overline{D})\le \widehat{\vol}(\overline{D'})=\widehat{\vol}^{\YZ}(\overline{D'}).\]
	Now choosing a sequence of such model $S$-metrized divisors $\overline{D'}$ converging to $\overline{D}$ with respect to the $\overline{B}$-boundary topology for some weak boundary divisor $\overline{B} \in \widehat{\Div}_{S,\Q}(U)_\mo$, by the generalization of \cite[Theorem 5.2.1]{yuan2021adelic}, see Remark~\ref{generize YZ limits of vol}, one has $\widehat{\vol}^\YZ(\overline{D})= \widehat{\vol}(\overline{D})$.
\end{proof}
\begin{remark}
	Similarly, for $\overline{D}\in \widehat{\Div}_{S,\Q}(U)_\cpt$, we have that	$\widehat{\vol}_\chi^\YZ(\overline{D})\leq \widehat{\vol}_\chi(\overline{D})$. We may expect that the equality holds if $\overline{D}\in \widehat{\Div}_{S,\Q}(U)^\YZ_\cpt$, but it is unknown. 
\end{remark}

\cref{theorem:consistentvolumedef} allows us to translate the results for in \cite[\S~5.2]{yuan2021adelic} in the language of our paper (although $\widehat{\Div}_{S,\Q}(U)_\cpt^\YZ$ is slightly large than the group of adelic divisors defined in \cite{yuan2021adelic}).
\subsection{Suitably approximating sequences}
In this subsection, we construct suitable Cauchy sequences of compactified $S$-metrized divisors which will be crucial for obtaining arithmetic Hilbert-Samuel formula later. 

\begin{lemma} \label{lemma: max is YZ divisor}
	Let $\overline{{D}}=(D,g), \overline{D'}=(D',g')\in \widehat{\Div}_{S,\Q}({U})_\relsnef$ be given as a limit of a sequence in $N_{S,\Q}(U)$ with respect to the $\overline{B}$-boundary topology for some boundary divisor $\overline{B}\in \widehat{\Div}_{S,\Q}(U)$. Assume that ${D} \geq {D}'$ in $\widetilde{\Div}_\Q(U)_\cpt$ and
	write 
	$h\coloneq \max\{g,g'\}$. 	
	\begin{enumerate1}
		\item \label{max is strongly relatively nef} If we can choose $\overline B$ in $N_{S,\Q}(U)_\mo$, then $h$ is an $S$-Green function for $D\in \widetilde{\Div}_{\Q}(U)_\cpt$, and $(D,h) \in \widehat{\Div}_{S,\Q}(U)_\relsnef^\YZ$. 
		\item \label{max is strongly S-nef} If we can choose $\overline B$ in $N_{S,\Q}(U)_\mo$  and if additionally $(D,g)\in N'_{S,\Q}(U)_\mo^{d_{\overline B}}$, then  $(D,h)\in \widehat{\Div}_{S,\Q}(U)_\arsnef^\YZ$.
	\end{enumerate1}
\end{lemma}
\begin{proof}
	The proof is similar as the one of \cite[Lemma~10.3]{cai2024abstract}. By \cite[Lemma~10.3]{cai2024abstract}, $h$ is a $S$-Green function for $D\in\widetilde{\Div}_\Q(U)_\cpt$.
	
	\vspace{2mm} \noindent
	Step 1: {\it If $U=X$ is projective and $\overline{D}, \overline{D'}\in N_{S,\Q}(X)_\mo$, let $c\in \mathscr{L}^1(\Omega,\mathcal{A},\nu)$ such that $c_\omega=0$  for any $\omega\in\Omega_{\mathrm{fin}}$ and $c_\omega>0$ for any $\omega\in\Omega_{\infty}$, then $(D,h)$ is the limit of a decreasing sequence $(D,h_n)$ in $N_{S,\Q}(X)_\mo$ with respect to $(0,c)$-boundary topology with $h_{n,\omega}=h_\omega$ for any $\omega\in\Omega_{\mathrm{fin}}$.}
	
	\vspace{2mm} \noindent
	Let $\overline{\mathcal{D}}=(\mathcal{D}, g_{\infty}), \overline{\mathcal{D}'}=(\mathcal{D}', g_{\infty}')$ be pairs of Cartier divisors with Green functions on a projective $C$-model $\mathcal{X}$ of $X$ inducing $\overline{D}, \overline{D'}$, respectively. Since $D\geq D'$, there is a finite subset $T$ of $C$ such that $\mathcal{D}|_{C\setminus T}\geq \mathcal{D}'|_{C\setminus T}$, where $\mathcal{D}|_{C\setminus T}$ (resp. $\mathcal{D}|_{C\setminus T})$ is the restriction of $\mathcal{D}$ (resp. $\mathcal{D}'$) on the $\mathcal{X}\times_C(C\setminus T)$. Then when $\omega\not\in T$, we have that $h_\omega=g_\omega$ and when $\omega\in T\cap\Omega_{\mathrm{fin}}$, we have that $(D_\omega, h_\omega)\in N_{\mo,\Q}(X_\omega)$ (see proof of \cite[Lemma~9.3]{cai2024abstract}). Then after replacing $\mathcal{X}$ by its blowing up along a subscheme of $\coprod\limits_{\omega\in T\setminus \Omega_\infty}\mathcal{X}_\omega$ where $\mathcal{X}_\omega$ is the special fiber of $\mathcal{X}$ over $\omega\in T\setminus\Omega_\infty$, we can find a divisor $\mathcal{D}''$ on $\mathcal{X}$ inducing $(D_\omega, h_\omega)$ for any $\omega\in \Omega_{\mathrm{fin}}$. When $\omega\in T\cap \Omega_\infty$, we have that $(D_\omega,g_\omega)\in \widehat{\Div}_\Q(X_\omega)_\nef$ by \cite[Lemma~9.3]{cai2024abstract}. Moreover, it is the limit of a decreasing sequence in $N_{\mo,\Q}(X_\omega)$. 
	Then $(D,h)$ is the limit of some decreasing sequence in $N_{S,\Q}(X)_\mo$ with respect to $(0,c)$-boundary topology. This proves Step 1.
	
	\vspace{2mm} \noindent
	Step 2: {\it Under the assumptions from Step 1 and if $(D,g) \in N_{S,\Q}'(X)_\mo$, then $(D,h)  \in \widehat{\Div}_{S,\Q}(X)_\arnef^\YZ$.}
	
	\vspace{2mm} \noindent
	By \cite[Lemma~10.3, Proposition~7.10~(iv)]{cai2024abstract}, we have that $(D,h)\in N'_{S,\Q}(X)$, i.e. $S$-nef on $X$. From Step~1, $(D,h)$ is the limit of a decreasing sequence $(D,h_j)_{j\in\N_{\geq 1}}$ in $N_{S,\Q}(X)_\mo$. Since $(D,h_j)\geq (D,h)$ and $(D,h)\in N'_{S,\Q}(X)$, we have that $(D,h_j)\in N_{S,\Q}'(X)_\mo$. This proves Step~2. 
	
	\vspace{2mm} \noindent
	Step 3: {\it \ref{max is strongly relatively nef} holds.}
	
	\vspace{2mm} \noindent
	Let $\overline{B}=(B,g_B)\in N_{S,\Q}(U)$ be a weak boundary divisor, and $(\overline{D_j})_{j\in \N_{\geq1}}, (\overline{D_j'})_{j\in \N_{\geq1}}$ sequences in $N_{S,\Q}(U)_\mo$ converging to $\overline{D}$, $\overline{D'}$ with respect to the $\overline{B}$-boundary topology, respectively. We can assume that there is a positive constant $c_0$ such that $g_{B,\omega}>c_0$ for any $\omega\in \Omega_\infty$.  Let $X_j$ be a projective model of $U$ such that $\overline{D}_j=(D_j, g_j), \overline{D_j'}=(D_j', g_j')\in N_{S,\Q}(X_j)_\mo$. Since $D\geq D'$, after adding small multiples of $\overline{B}$ to $\overline{D}_j$, we can assume that $D_j\geq D_j'$. By Step~1, we have that $(D_j, \max\{g_j,g_j'\})\in \widehat{\Div}_{S,\Q}(X_j)_\relnef^\YZ$  and $(D_j, \max\{g_j,g_j'\})$ converges to $(D, h)$ in $\widehat{\Div}_{S,\Q}(U)_\cpt$ with respect to $\overline{B}$-boundary topology. By Step~1, we can choose an element $\overline{D_j'}$ in $N_{S,\Q}(X_j)_\mo$ closed to $(D_j, \max\{g_j,g_j'\})$ with respect to the $\overline{B}$-boundary topology (notice that $g_{B,\omega}>c_0>0$ on $U^\an_\omega$ for any $\omega\in\Omega_\infty$). Then $\overline{D_j'}$ converges to $(D, h)$ in $\widehat{\Div}_{S,\Q}(U)_\cpt$ with respect to $\overline{B}$-boundary topology. This proves Step~3.
	
	\vspace{2mm} \noindent
	Step 4: {\it \ref{max is strongly S-nef} holds .}
	
	\vspace{2mm} \noindent
	We can twist a boundary divisor $\overline{B}\in N_{S,\Q}(U)_\mo$ with a positive constant at some place $\omega\in\Omega$, and assume that $\overline{B}\in N_{S,\Q}'(U)$. Then Step 4 follows from the same proof as in Step 3 relying now on Step 2 instead of Step 1.
\end{proof}

\begin{remark} \label{nef become strongly S-nef after shrinking U}
	We can obtain a weak boundary divisor $\overline{B}\in N_{S,\Q}(U)_\mo$ by shrinking $U$. The shrinking procedure is useful in the following situation: as \cref{nef become strongly nef after shrinking U}, by shrinking $U$  we may assume that finitely many given relatively nef (resp. arithmetically nef)  $\YZ$-divisors become strongly relatively nef (resp. strongly arithmetically nef).
\end{remark}
\begin{lemma} 
	\label{lemma:existence of sequence of intersections to nef intersection}
	Let $\overline{D_0}=(D_0,g_0),\dots,\overline{D_d}=({D}_d,g_d)\in\widehat{\Div}_{S,\Q}({U})_{\relsnef}$ and $\overline{D_0'}=({D}_0,g_0'),\break \dots, \overline{D_d'}=({D}_d,g_d')\in \widehat{\Div}_{S,\Q}({U})_{\arsnef}^\YZ$. Assume that for any $j=0,\dots, d$,
	\begin{itemize}
		\item  there is a boundary divisor $\overline{B}\in N_{S,\Q}(U)_\mo$ and a sequence in $N_{S,\Q}(U)_\mo$ (resp. $N_{S,\Q}'(U)_\mo$) converging to $\overline{D_j}$ (resp. $\overline{D_j'}$);
		\item we have that $g_{j,\omega}= g_{j,\omega}'$ for all but finitely many $\omega\in\Omega$.
	\end{itemize}
	For any $m\in\N_{\geq 1}$, set 
	\[h_{j,m}\coloneq\max\{g_{j},g_{j}'-m\}.\] 
	Then the following statements hold.
	\begin{enumerate1}
		\item \label{nefness for constructed divisors} For each $j=0,\dots, d$, the function $h_{j,m}$ is an $S$-Green function for ${D}_j\in\widetilde{\Div}_\Q({U})_\cpt$ and $\overline{D_{j,m}}\coloneq({D}_j,h_{j,m})\in \widehat{\Div}_{S,\Q}(U)^\YZ_\relsnef$.
		\item \label{convergences for constructed divisors} For each $j=0,\dots, d$, the sequence $(\overline{D_{j,m}})_{m\in\N_{\geq 1}}$ decreasingly converges to $\overline{D_j}$ with respect to the $\overline{B}$-boundary topology for some weak boundary divisor $\overline{B}\in N_{S,\Q}(U)_\mo$.
		\item \label{convergences of intersection number for constructed divisors} We have that
		\[\lim\limits_{m\to\infty}(\overline{D_{0,m}}\cdots\overline{D_{d,m}}\mid U)_S=(\overline{D_0}\cdots\overline{D_d}\mid U)_S.\]
		\item \label{convergences of intersection number with integrable divisors} For any $\overline{E}=(E,f)\in \widehat{\Div}_{S,\Q}(U)_\arint^\YZ$ with underlying compactified (geometric) divisor $E=0\in \widehat{\Div}_{S,\Q}(U)_\cpt$, we have that
			\begin{align}\label{eq:lim of intersection0}
				\lim\limits_{m\to\infty} (\overline{D_{1,m}}\cdots\overline{D_{d,m}}\cdot \overline{E}\mid U)_S = \sum_{\omega\in\Omega}\nu(\omega)\int_{U_\omega} f_\omega\, c_1(\overline{D_{1,\omega}})\wedge\cdots \wedge c_1(\overline{D_{d,\omega}}).
		\end{align}
	\end{enumerate1}
\end{lemma}
\begin{proof}
	\ref{nefness for constructed divisors} Let $T\subset\Omega$ be a finite subset such that for any $1\leq j\leq d$ and $\omega\in\Omega$, $g_{j,\omega}\neq g_{j,\omega}$ implies that $\omega\in T$. Then for any $m\in \N_{\geq 1}$ and $\omega\not\in \Omega\setminus T$, we have that $g_{j,\omega}\geq g_{j,\omega}'-m$, and $h_{j,m,\omega} = g_{j,\omega}$. By \cref{lemma: max is YZ divisor}~\ref{max is strongly relatively nef}, we know that $h_{j,m}$ is an $S$-Green function for ${D}_j\in\widetilde{\Div}_\Q({U})_\cpt$, and $\overline{D_{j,m}}\in \widehat{\Div}_{S,\Q}(U)_{\relsnef}^\YZ$. 
	
	\ref{convergences for constructed divisors} Obviously, we have that $(\overline{D_{j,m}})_{m\in \N_{\geq 1}}$ is decreasing, as the proof of \cite[Lemma~10.6]{cai2024abstract}, it is not hard to show \ref{convergences for constructed divisors} holds similarly. 
	
	\ref{convergences of intersection number for constructed divisors}  
	Recall the definition of $E(\mathbf{g}',\mathbf{g})$ in \cref{extension of YZ intersection number}. 
	By \ref{convergences for constructed divisors} and \cite[Theorem~10.9~(v)]{cai2024abstract}, we have that
	\[\lim\limits_{m\to\infty}E(\mathbf{g}', \mathbf{h}_{m})=E(\mathbf{g}',\mathbf{g}).\]
	Hence 
	\[	\lim\limits_{m\to\infty}(\overline{D_{0,m}}\cdots \overline{D_{d,m}}\mid U)_S=(\overline{D_{0}}\cdots \overline{D_d}\mid U)_S.\]
	This proves \ref{convergences of intersection number for constructed divisors}.
	
	\ref{convergences of intersection number with integrable divisors} Write $\overline{E}=\overline{E_1} - \overline{E_2}$ with $\overline{E_1}=(E_1,f_1), \overline{E_2}=(E_2,f_2)\in \widehat{\Div}_{S,\Q}(U)_\arnef^\YZ$. Since $E=0$, we have that $E_1=E_2\in \widetilde{\Div}_\Q(U)_{\nef}$. We may assume that $[f_1]\leq [f_2]$ in the sense of \cref{def:equivalent singularities}. Otherwise, we consider $f_3\coloneq \max\{f_1,f_2\}$ which is still a Green function for $E_1=E_2$ and $(E_1,f_3)\in \widehat{\Div}_{S,\Q}(U)^\YZ_\arnef$ (see \cref{lemma: max is YZ divisor}~\ref{max is strongly S-nef}), $[f_1]\leq [f_3], [f_2]\leq [f_3]$, then \eqref{eq:lim of intersection0} for $\overline{E_1}-\overline{E_3}$ and for $\overline{E_2}-\overline{E_3}$ implies \eqref{eq:lim of intersection0} for $\overline{E}=\overline{E_1}-\overline{E_2}$, i.e. \ref{convergences of intersection number with integrable divisors} holds. After assigning $\overline{D_0} = \overline{E_1}$ and $\overline{D_0}=\overline{E_2}$, notice that $f_1=\max\{f_1, f_1-m\}, f_2=\max\{f_2,f_2-m\}$, by \ref{convergences of intersection number for constructed divisors}, we have that
		\begin{align}\label{eq:lim of intersection1}
			\lim\limits_{m\to\infty}(\overline{D_{1,m}}\cdots \overline{D_{d,m}}\cdot\overline{E_1}\mid U)_S = (\overline{D_1}\cdots \overline{D_d}\cdot\overline{E_1}\mid U)_S.
		\end{align}
		By \cref{extension of YZ intersection number}
		\begin{align}\label{eq:lim of intersection2}
			(\overline{D_1}\cdots \overline{D_d}\cdot\overline{E_1}\mid U)_S = (\overline{D_1'}\cdots \overline{D_d'}\cdot\overline{E_1}\mid U)_S + E((g_1',\cdots, g_d',f_1),(g_1,\cdots, g_d,f_1)).
		\end{align} 
		Similarly, we have that
		\begin{align}\label{eq:lim of intersection3}
			\lim\limits_{m\to\infty}(\overline{D_{1,m}}\cdots \overline{D_{d,m}}\cdot\overline{E_2}\mid U)_S = (\overline{D_1}\cdots \overline{D_d}\cdot\overline{E_2}\mid U)_S.
		\end{align}
		\begin{align}\label{eq:lim of intersection4}
			(\overline{D_1}\cdots \overline{D_d}\cdot\overline{E_2}\mid U)_S = (\overline{D_1'}\cdots \overline{D_d'}\cdot\overline{E_2}\mid U)_S + E((g_1',\cdots, g_d',f_2),(g_1,\cdots, g_d,f_2)).
		\end{align}
		We take \eqref{eq:lim of intersection2}$-$\eqref{eq:lim of intersection4} and substitute it into \eqref{eq:lim of intersection1}$-$\eqref{eq:lim of intersection3}, then
		\begin{align*}
			&\lim\limits_{m\to\infty}(\overline{D_{1,m}}\cdots \overline{D_{d,m}}\cdot\overline{E}\mid U)_S \\=& (\overline{D_1}\cdots \overline{D_d}\cdot\overline{E_1}\mid U)_S-(\overline{D_1}\cdots \overline{D_d}\cdot\overline{E_2}\mid U)_S\\
			=&(\overline{D_1'}\cdots \overline{D_d'}\cdot\overline{E_1}\mid U)_S-(\overline{D_1'}\cdots \overline{D_d'}\cdot\overline{E_2}\mid U)_S+E((g_1',\cdots, g_d',f_1),(g_1,\cdots, g_d,f_1))\\
			&-E((g_1',\cdots, g_d',f_2),(g_1,\cdots, g_d,f_2)).
		\end{align*}
		By \cite[Theorem~11.2]{cai2024abstract} and our assumption that $[f_1]\leq[f_2]$, we have that
		\[(\overline{D_1'}\cdots \overline{D_d'}\cdot\overline{E_1}\mid U)_S-(\overline{D_1'}\cdots \overline{D_d'}\cdot\overline{E_2}\mid U)_S+E((g_1',\cdots, g_d',f_2),(g_1',\cdots, g_d',f_1)).\]
		Then 
		\begin{align*}
			&\lim\limits_{m\to\infty}(\overline{D_{1,m}}\cdots \overline{D_{d,m}}\cdot\overline{E}\mid U)_S \\=& E((g_1',\cdots, g_d',f_2),(g_1',\cdots, g_d',f_1))+E((g_1',\cdots, g_d',f_1),(g_1,\cdots, g_d,f_1))\\
			&-E((g_1',\cdots, g_d',f_2),(g_1,\cdots, g_d,f_2))\\
			=&E((g_1',\cdots, g_d',f_2),(g_1,\cdots, g_d,f_1))-E((g_1',\cdots, g_d',f_2),(g_1,\cdots, g_d,f_2))\\
			=&E((g_1,\cdots, g_d,f_2),(g_1,\cdots, g_d,f_1))\\
			=&\sum_{\omega\in\Omega}\nu(\omega)\int_{U_\omega} f_\omega\, c_1(\overline{D_{1,\omega}})\wedge\cdots \wedge c_1(\overline{D_{d,\omega}}),
		\end{align*}
		the second and third equality are from \cite[Proposition~10.10]{cai2024abstract}. This completed the proof of \ref{convergences of intersection number with integrable divisors}.
\end{proof}

\subsection{Positive intersection product}
We recall the \emph{positive intersection product} defined in \cite[Definition~3.2]{nijerdiff}.
\begin{art}\label{positive intersection number}
	Let $X$ be a projective variety over $K$. We say a model $S$-metrized divisor $\overline{A}\in \widehat{\Div}_{S,\Q}(X)_\mo$ is \emph{free} if $\overline{A}\in N_{S,\Q}(X)_\mo$ (see \cref{model metrics and model Green functions}) and there is $n\in\N_{\geq 1}$ such that $n\overline{A}$ is with $\Z$-coefficients induced by some semiample divisor $\mathcal{A}$ on some projective $C$-model of $X$.
	
	Let $\overline{D}\in\widehat{\Div}_{S,\Q}(U)_\cpt^\YZ$ and $\overline{E}\in\widehat{\Div}_{S,\Q}(U)_\arnef^\YZ$. If $\overline{D}$ is big, the \emph{positive intersection product} of $\overline{D}$ and $\overline{E}$ is defined as
	\[\langle\overline{D}^d\rangle\cdot\overline{E}\coloneq \sup\limits_{(\pi,\overline{A})}\{\overline{A}^{d}\cdot\pi^*\overline{E}\},\]
	where $(\pi,\overline{A})$ runs over all tuples such that $\pi\colon U'\to U$ is a birational morphism, and $\overline{A}\in\widehat{\Div}_{S,\Q}(U')_\mo$ is induced by a free model $S$-metrized divisor on some projective model of $U'$ such that $\pi^*\overline{D}-\overline{A}\geq 0$. By linearity, if $\overline{E}\in \widehat{\Div}_{S,\Q}(U)^\YZ_{\arint}$, $\langle\overline{D}^d\rangle\cdot\overline{E}$ is well-defined, see \cite[Lemma~3.10]{nijerdiff}. Notice that the positive intersection product is stable under birational pull-back, see \cite[Lemma~3.3]{nijerdiff}.
\end{art}

\begin{lemma} \label{lemma:positive intersection product for relnef}
	Let $\overline{D}=(D,g)\in \widehat{\Div}_S(U)^{\arnef,\YZ}_{\relnef}$ with $D$ is big, and $(0,f)\in\widehat{\Div}_{S,\Q}(U)_\arint^\YZ$. 
	If $\inf\limits_{\lambda\in\Delta(D)^\circ}G_{\overline{D}}(\lambda)>0$, then 
	\begin{equation}\label{eq:positive intersection product for relnef and underlying trivial}
		\langle\overline{D}^d\rangle\cdot (0,f)= \sum\limits_{\omega\in\Omega}\nu(\omega)\int_{U^\an_\omega}f_\omega\, c_1(\overline{D}_\omega)^d.
	\end{equation}
\end{lemma}
\begin{proof}
	Write $\overline{E}=(0,f)$. As the proof of Step~3 in the proof of \cref{thm:Hilbert-Samuel formula}, after shrinking $U$ (notice that the both sides of \eqref{eq:positive intersection product for relnef and underlying trivial} are stable under birational base change), there is a sequence $(\overline{D_m})_{m\in\N_{\geq 1}}\subset \widehat{\Div}_{S,\Q}(U)_\relsnef$ and a sequence $(c_m)_{m\in\N_{\geq 1}}\subset \mathscr{L}^1(\Omega,\mathcal{A},\nu)_{\geq 0}$ satisfying the following properties:
	\begin{itemize}
		\item $(\overline{D_m})_{m\in\N_{\geq 1}}$ is decreasing converging to $\overline{D}$ with respect to the $\overline{B}$-boundary topology for some weak boundary divisor $\overline{B}\in N_{S,\Q}(U)_\mo$ and 
		\begin{equation}\label{eq:convergence of intersection numbers}
			\lim\limits_{m\to\infty}(\overline{D_m}^d\cdot \overline{E}\mid U)_S =\sum\limits_{\omega\in\Omega}\nu(\omega)\int_{U_\omega^\an}f_\omega\, c_1(\overline{D})^d
		\end{equation}
		by \cref{lemma:existence of sequence of intersections to nef intersection}; 
		\item for any $m\in\N_{\geq 1}$, $c_m(\omega)=0$ for all but finitely many $\omega\in\Omega$;
		\item $\overline{D_m}(c_m)\in \widehat{\Div}_{S,\Q}(U)^\YZ_\arsnef$, then $$\langle\overline{D_m}(c_m)^d\rangle\cdot\overline{E}=(\overline{D_m}(c_m)^d\cdot\overline{E}\mid U)_S$$ by the linearity and \cite[Corollary~3.8]{nijerdiff}.
	\end{itemize} 
	We claim that
	$$\langle\overline{D_m}(c_m)^d\rangle\cdot\overline{E}=\langle\overline{D_m}^d\rangle\cdot\overline{E}.$$ 
	If this holds, notice that
	\[(\overline{D_m}(c_m)^d\cdot\overline{E}\mid U)_S = (\overline{D_m}^d\cdot\overline{E}\mid U)_S\]
	(see \cite[Theorem~11.3]{cai2024abstract}), then $$\langle\overline{D_m}^d\rangle\cdot\overline{E} = (\overline{D_m}^d\cdot\overline{E}\mid U)_S.$$
	Taking $m\to\infty$ on both sides, by \cite[Lemma~3.6]{nijerdiff} and \eqref{eq:convergence of intersection numbers}, the lemma holds. It remains to show the claim. To show the claim, we consider the functions $\varphi, \psi$ on $\R$ as follows:
	\[\varphi(t)\coloneq \widehat{\vol}(\overline{D_m}(c_m)+t\overline{E}), \ \ \psi(t)\coloneq \widehat{\vol}(\overline{D_m}+t\overline{E}).\]
	Since $\overline{D_m}$ is decreasing converging to $\overline{D}$, then \[G_{\overline{D_m}(c_m)}\geq G_{\overline{D_m}}\geq G_{\overline{D}}\geq \inf\limits_{\lambda\in\Delta(D)^\circ}G_{\overline{D}}(\lambda)>0.\]
	Notice that there is $b\in \mathscr{L}^1(\Omega,\mathcal{A},\nu)_{\geq 0}$ such that $b(\omega)=0$ for all but finitely many $\omega\in\Omega$ and $\sup\limits_{x\in U^\an_\omega}\{|f_\omega(x)|\}\leq b(\omega)$ for any $\omega\in\Omega$.
	Then for any $t\in \R$,
	\[\overline{D_m}-|t|(0,b)\leq\overline{D_m}+t\overline{E}\leq \overline{D_m}+|t|(0,b),\]
	\[G_{\overline{D_m}+t\overline{E}}\geq G_{\overline{D_m}-|t|(0,b)}= G_{\overline{D_m}}-|t|\int_{\Omega}b(\omega)\, \nu(d\omega).\] 
	Then by \cref{theorem:concavemain}, when $|t|$ is small enough, we have that 
	\begin{align*}
		\varphi(t)=&\widehat{\vol}(\overline{D_m}(c_m)+t\overline{E})\\ 
		=& \widehat{\vol}_\chi(\overline{D_m}(c_m)+t\overline{E})\\
		=& \widehat{\vol}_\chi(\overline{D_m}+t\overline{E}) +\sum\limits_{\omega\in\Omega}c_m(\omega)\nu(\omega)\cdot D^{d}_m\\
		=&\psi(t)+\sum\limits_{\omega\in\Omega}c_m(\omega)\nu(\omega)\cdot D^{d}_m.
	\end{align*}
	Taking the derivatives on both sides of $\varphi(t)=\psi(t)+\sum\limits_{\omega\in\Omega}c_m(\omega)\nu(\omega)\cdot D^{d}_m$, then \cite[Theorem~3.13]{nijerdiff} implies our claim. This finishes the proof of the lemma.
\end{proof}
\subsection{Equidistribution}
We quickly recall the setting in which we prove the equidistribution.

\begin{art}
 Let $(x_m)_{m\in I}\subset U(\overline{K})$ a net of geometric points.
\begin{itemize}
         \item We say that $(x_m)_{m\in I}\subset U(\overline{K})$ is \emph{generic} if it is Zariski dense in the Zariski topology of $U$.	
         
	\item 
	We say that $(x_m)_{m\in I}\subset U(\overline{K})$ is \emph{small with respect to $\overline{D}$} if
	\[\lim_{m\in I}h_{\overline{D}}(x_m)=\frac{\widehat{\vol}_\chi^\num(\overline{D})}{(d+1)\vol(D)}.\]  
 
       \item Let $v\in \Omega$, and $\mu$ a measure on $U^\an_v$.
	 We say that the Galois orbit of $(x_m)_{m\in I}$ \emph{equidistributes in $U_v^\an$ with respect to $\mu$} if 
	\[\lim_{m\in I} \frac{1}{\#O(x_m)}\sum_{\sigma\in\Gal{\overline{K}}{K}}f(\sigma(x_m))=\int_{U^\an_v} f\, d\mu\]
	for every compactly supported function $f$ on $U_v^\an$, where $O(x_m)$ denotes the Galois orbit of $x_m$ under the action of $\Gal{\overline{K}}{K}$ and $U^\an_v$ denotes the Berkovich analytification\ of $U\times_K K_v$. See \cite[Section~2]{burgos2019the} for more details.
 \end{itemize}
\end{art}
\begin{remark}
    \label{remark:equidistributioncriterion}
    Let $(x_m)_{m\in I}\subset U(\overline{K})$ a net of geometric points, and $\overline{D}=(D,g)\in\widehat{\Div}_{S,\Q}(U)_{\cpt}$ with $D$ big. By \cref{thm:essential minimum}, we have that
	\[\limsup_{m\in I}h_{\overline{D}}(x_m)\geq \zeta_{\mathrm{ess}}(\overline{D}) \geq \widehat{\mu}_{\max}^{\mathrm{asy}}(\overline{D}) \geq \frac{\widehat{\vol}_\chi^{\num}(\overline{D})}{(d+1)\vol(D)}.\]
	If $(x_m)_{m\in I}$ is small with respect to $\overline{D}$,  then $G_{\overline{D}}(\lambda) = \widehat{\mu}_{\max}^{\mathrm{asy}}(\overline{D})$ on $\Delta(D)^\circ$. In particular, if this happen, then $\inf\limits_{\lambda\in\Delta(D)^\circ}G(\lambda) = \widehat{\mu}_{\max}^{\mathrm{asy}}(\overline{D})>-\infty$ and \[\widehat{\vol}_\chi(\overline{D})=\widehat{\vol}_\chi^\num(\overline{D}) = \widehat{\mu}_{\max}^{\mathrm{asy}}(\overline{D}) = \zeta_{\mathrm{ess}}(\overline{D}).\]
 	Furthermore, let us consider the following statements which often appear as hypothesis in equidistribution theorem:
	\begin{enumerate}[resume,leftmargin=*,label=\it(\alph*),ref=\it{(\alph*)}]
		\item \label{rmk:cpt small points} $(x_m)_{m\in I}$ is small with respect to $\overline{D}$;
		\item \label{rmk:minimal slope small points} $\widehat{\mu}_{\min}^{\mathrm{asy}}(\overline{D})>-\infty$ and 
		\[\lim\limits_{m\in I}h_{\overline{D}}(x_m) = \frac{\widehat{\vol}_\chi(\overline{D})}{(d+1)\vol(D)};\]
		\item \label{rmk:bigness small points} $\overline{D}$ is big and 
		\[\lim\limits_{m\in I}h_{\overline{D}}(x_m) = \frac{\widehat{\vol}(\overline{D})}{(d+1)\vol(D)};\]
		\item \label{rmk:relatively nef small points} $\overline{D}\in\widehat{\Div}_{S,\Q}(U)^{\arnef,\YZ}_{\relnef}$ and
		\[\lim\limits_{m\in I}h_{\overline{D}}(x_m) = \frac{(\overline{D}^{d+1}\mid U)_S^{d+1}}{(d+1)D^d}.\]
	\end{enumerate}
If $\widehat{\mu}_{\min}^{\mathrm{asy}}(\overline{D})>-\infty$, by \cref{theorem:concavemain}, then  \ref{rmk:cpt small points} \ref{rmk:minimal slope small points} are equivalent. If $\overline{D}$ is big, by \cref{cor:bigness}, then $\widehat{\mu}_{\max}^{\mathrm{asy}}(\overline{D})>0$. By \cref{lemma:refokoun}, \cref{properties of concave transforms}~\ref{inf and sup of G} and  \cref{theorem:concavemain}, we have that 
\[\widehat{\mu}_{\max}^{\mathrm{asy}}(\overline{D})= \frac{d!}{\vol(D)}\int_{\Delta(D)^\circ}\widehat{\mu}_{\max}^{\mathrm{asy}}(\overline{D})\, d\lambda \geq \frac{d!}{\vol(D)}\int_{\Delta(D)^\circ}\max\{G_{\overline{D})},0\}\, d\lambda = \frac{\widehat{\vol}(\overline{D})}{(d+1)\vol(D)},\]
and the equality holds if and only if $G_{\overline{D}}(\lambda) = \widehat{\mu}_{\max}^{\mathrm{asy}}(\overline{D})$ on $\Delta(D)^\circ$. Hence in this case, $\widehat{\vol}(\overline{D}) = \widehat{\vol}_\chi^{\mathrm{num}}(\overline{D})$, then \ref{rmk:cpt small points} \ref{rmk:bigness small points}  are equivalent. If $\overline{D}\in\widehat{\Div}_{S,\Q}(U)^{\arnef,\YZ}_{\relnef}$, by \cref{thm:Hilbert-Samuel formula}, then \ref{rmk:cpt small points} \ref{rmk:relatively nef small points} are equivalent.
\end{remark}

For any $v\in\Omega$, and $f_v\in C_c(U_v^\an)$, we simply denote by $(0,f_v)$ the compactified metrized divisor $\overline{D}=(D,g)$ of $U$ such that $D$ is trivial, $g_v=f_v$, and $g_\omega=0$ if $\omega\in\Omega\setminus v$.
\begin{theorem}
		\label{theorem:equidsitributionforfinitenergy}
		Let $\overline{D}=(D,g)\in \widehat{\Div}_{S,\Q}(U)^{\YZ}_{\cpt}$ with $D$ big. Let $(x_m)_{m\in I}$ be a generic net of points which is small with respect to $\overline{D}$. Then for any $c\in\mathscr{L}^1(\Omega,\mathcal{A},\nu)$ such that $\overline{D}(c)$ is big, for any $v\in\Omega$ and for any $f_v\in C_c(U_v^\an),$ we have that
		\begin{align}\label{eq:equidistribution}
\lim_{m\in I} \frac{\nu(v)}{\#O(x_m)}\sum_{\sigma\in\Gal{\overline{K}}{K}}f_v(\sigma(x_m))=\frac{\langle \overline{D}(c)^d\rangle\cdot (0,f_v)}{\vol(D)}.
		\end{align}
		In particular, if $\overline{D}\in \widehat{\Div}_{S,\Q}(U)^{\arnef,\YZ}_{\relnef}$, then
		the Galois orbit of $(x_m)_{m\in I}$ equidistributes in $U^\an_v$ with respect to to $\frac{1}{D^d}\cdot c_1(\overline{D}_v)^{d}$.
\end{theorem}
\begin{proof}
	By \cref{remark:equidistributioncriterion}, we have that $G_{\overline{D}}(\lambda) \equiv \widehat{\mu}_{\max}^{\mathrm{asy}}(\overline{D})>-\infty$ on $\Delta(D)^{\circ}$. Let $c\in\mathscr{L}^1(\Omega,\mathcal{A},\nu)$ such that $\overline{D}(c)$ is big. Then $\inf\limits_{\lambda\in\Delta(D)^\circ}G_{\overline{D}(c)}(\lambda)>0$.
	
	Let $X$ be a projective $K$-model of $U$. Recall the space $\widehat{\Div}_\Q(X_v)_\mo$ in \cref{local model divisors}. 
		Set $\overline{E}\coloneq(0,f_v)\in \widehat{\Div}_{S,\Q}(X)_\mo$. We consider functions on $\R$: for any $m\in I$,
\begin{align}\label{eq:equidistribution varphi}
\varphi_{m}(t)\coloneq h_{\overline{D}(c)+t\overline{E}}(x_m)=h_{\overline{D}(c)}(x_m)+\frac{\nu(v)\cdot t}{\#O(x_m)}\sum_{\sigma\in\Gal{\overline{K}}{K}}f_v(\sigma(x_m)),
\end{align}	
		\[\psi(t)\coloneq \frac{\widehat{\vol}_\chi^{\num}(\overline{D}(c)+t\overline{E})}{(d+1)\cdot \vol(D)}\]
		(the last equality of \eqref{eq:equidistribution varphi} is from \cite[Proposition~2.3]{burgos2019the}).
		We prove the theorem in the following steps.
	
	\vspace{2mm} \noindent
	Step 1: {\it The following statements hold:
		\begin{enumerate}
			\item \label{varphi concave} $\varphi_m(t)$ is concave and differentiable with derivative $\frac{\nu(v)}{\#O(x_m)}\sum\limits_{\sigma\in\Gal{\overline{K}}{K}}f_v(\sigma(x_m))$ at $t=0$;
			\item \label{inf larger} $\liminf\limits_{m\in I}\varphi_m(t)\geq \psi(t)$;
			\item \label{value at 0} $\lim\limits_{m\in I}\varphi_m(0)=\psi(0)$.
	\end{enumerate}}
	
	\vspace{2mm} \noindent
	\ref{varphi concave} is from the definition. For \ref{inf larger}, since  $(x_m)_{m\in I}$ is generic, we have that
	\[\liminf\limits_{m\in I}h_{\overline{D}(c)+t\overline{E}}(x_m)\geq \zeta_{\mathrm{ess}}(\overline{D}(c)+t\overline{E})\geq \frac{\widehat{\vol}_\chi^\num(\overline{D}(c)+t\overline{E})}{(d+1)\cdot \vol(D)^d}\]
	which proves \ref{inf larger}, here the last equality is from \cref{thm:essential minimum}. \ref{value at 0} is from \cref{properties of concave transforms}~\ref{shifting} the assumption that $(x_m)_{m\in I}$ is small.

	\vspace{2mm} \noindent
	Step 2: {\it If $(0,f_v)\in \widehat{\Div}_{S,\Q}(X)_\mo$ (i.e. locally $(0,f_v)\in \widehat{\Div}_{\Q}(X_v)_\mo$), then $\psi(t)$ is differentiable at $t=0$ with  derivative given by \[\frac{d\psi}{dt}(0)=\frac{\langle \overline{D}(c)^d\rangle\cdot (0,f_v)}{\vol(D)}.\]
	Hence, by Step~2 and \cite[Lemma~7.6]{berman2010growth}, we have that
	\begin{equation} \label{eq: equidistribution}
		\lim\limits_{m\in I}\frac{\nu(v)}{\#O(x_m)}\sum_{\sigma\in\Gal{\overline{K}}{K}}f_v(\sigma(x_m))=\frac{\langle \overline{D}(c)^d\rangle\cdot (0,f_v)}{\vol(D)}.
		\end{equation}}
		
		\vspace{2mm} \noindent
		Let $b_v\coloneq \sup\limits_{x\in X_v^\an}\{|f_v(x)|\}<\infty$. 
		For any $\omega\in\Omega\setminus\{v\}$, set $b_\omega=0$
		Write $b=(b_\omega)_{\omega\in\Omega}$.
		Let $t\in\R$. Then
		\[\overline{D}(c)+t\overline{E}\geq \overline{D}(c)-|t|(0,b)\]
		which implies that 
		\[G_{\overline{D}(c)+t\overline{E}}\geq G_{\overline{D}(c)}-|t|b_v\geq \inf_{\lambda\in\Delta(D)^\circ}G_{\overline{D}(c)}(\lambda)-|t|b_v\]
		by \cref{properties of concave transforms}~\ref{monotone}. By \cref{theorem:concavemain}, we have that 
		\[\widehat{\vol}_\chi^{\num}(\overline{D}(c)+t\overline{E})+c_v=\widehat{\vol}(\overline{D}(c)+t\overline{E})\]
		when $|t|$ is small enough. By \cite[Theorem~3.13]{nijerdiff}, we have that
		\[\frac{d\psi}{dt}(0)=\frac{\langle\overline{D}(c)^d\rangle\cdot\overline{E}}{\vol(D)}.\]
		
		\vspace{2mm} \noindent
		Step 3: {\it \eqref{eq:equidistribution} holds for any $f_v\in C_c(U_v^\an)$.}
		
		\vspace{2mm} \noindent
		To show the theorem is equivalent to show \eqref{eq: equidistribution} holds for any $f_v\in C_c(U_v^\an)$. Notice that $C_c(U_v^\an)\subset C(X_v^\an)$. By \cite[Theorem~7.12]{gubler1998local} and \cite[Lemma~3.5]{yuan2008big}, the set
		\[\{f_v\in C(X_v^\an)\mid (0,f_v)\in \widehat{\Div}_\Q(X_v)_\mo\}\]
		is dense in $C(X_v^\an)$ under the topology of uniform convergence. Hence it suffices to show \eqref{eq: equidistribution} for $f_v\in C(X_v^\an)$ such that $(0,f_v)\in \widehat{\Div}_{\Q}(X_v)_\mo$. 
		This is Step~2, so the theorem holds. 
		
		\vspace{2mm} \noindent
		Step 4: {\it If $\overline{D}\in \widehat{\Div}_{S,\Q}(U)^{\arnef,\YZ}_{\relnef}$, then
			the Galois orbit of $(x_m)_{m\in I}$ equidistributes in $U^\an_v$ with respect to to $\frac{1}{D^d}\cdot c_1(\overline{D}_v)^{d}$.}
		
		\vspace{2mm} \noindent
		This is from \eqref{eq:equidistribution} and \cref{lemma:positive intersection product for relnef}. Notice that  $c_1(\overline{D}(c)_v)^{d} =  c_1(\overline{D}_v)^{d}$.
\end{proof}


\begin{thebibliography}{BMPS16}
	
	\bibitem[AB95]{Abbes1995arithmetic}
	A.~Abbes and T.~Bouche.
	\newblock Th{\'e}or{\`e}me de hilbert-samuel "arithm{\'e}tique".
	\newblock {\em Ann. de l'Inst. Fourier}, 45(2):375--401, 1995.
	
	\bibitem[BB10]{berman2010growth}
	R.~Berman and S.~Boucksom.
	\newblock Growth of balls of holomorphic sections and energy at equilibrium.
	\newblock {\em Invent. Math.}, 181(2):337--394, 2010.
	
	\bibitem[BC11]{boucksom2011okounkov}
	S.~Boucksom and H.~Chen.
	\newblock Okounkov bodies of filtered linear series.
	\newblock {\em Compos. Math.}, 147(4):1205--1229, 2011.
	
	\bibitem[BE21]{boucksom2021spaces}
	S.~Boucksom and D.~Eriksson.
	\newblock Spaces of norms, determinant of cohomology and {F}ekete points in
	non-{A}rchimedean geometry.
	\newblock {\em Adv. Math.}, 378:Paper No. 107501, 124, 2021.
	
	\bibitem[Ber90]{berkovich1990spectral}
	V.G. Berkovich.
	\newblock {\em Spectral theory and analytic geometry over non-{A}rchimedean
		fields}, volume~33 of {\em Mathematical Surveys and Monographs}.
	\newblock American Mathematical Society, Providence, RI, 1990.
	
	\bibitem[BF14]{bermansingular}
	R.J. Berman and G.~{Freixas i Montplet}.
	\newblock An arithmetic {H}ilbert-{S}amuel theorem for singular hermitian line
	bundles and cusp forms.
	\newblock {\em Compos. Math.}, 150(10):1703--1728, 2014.
	
	\bibitem[BGM21]{boucksom2021non}
	S.~Boucksom, W.~Gubler, and F.~Martin.
	\newblock Non-{A}rchimedean volumes of metrized nef line bundles.
	\newblock {\em \'{E}pijournal G\'{e}om. Alg\'{e}brique}, 5:Art. 13, 34, 2021.
	
	\bibitem[Bis23a]{nijeribanra}
	D.~Biswas.
	\newblock Convex bodies associated to linear series of adelic divisors on
	quasi-projective varieties, 2023.
	\newblock \href{https://arxiv.org/pdf/2301.08120}{\tt arXiv:2301.08120}.
	
	\bibitem[Bis23b]{nijerdiff}
	D.~Biswas.
	\newblock Differentiability of adelic volumes and equidistribution on
	quasi-projective varieties, 2023.
	\newblock \href{https://arxiv.org/pdf/2312.12084}{\tt arXiv:2312.12084}.
	
	\bibitem[BK24]{burgos2024on}
	J.I. {Burgos Gil} and J.~Kramer.
	\newblock On the height of the universal abelian variety, 2024.
	\newblock \href{https://arxiv.org/pdf/2403.11745}{\tt arXiv:2403.11745}.
	
	\bibitem[BKK07]{burgos2007cohomological}
	J.I. {Burgos Gil}, J.~Kramer, and U.~K\"{u}hn.
	\newblock Cohomological arithmetic {C}how rings.
	\newblock {\em J. Inst. Math. Jussieu}, 6(1):1--172, 2007.
	
	\bibitem[BMPS16]{burgos2016arithmetic}
	J.I. {Burgos Gil}, A.~Moriwaki, P.~Philippon, and M.~Sombra.
	\newblock Arithmetic positivity on toric varieties.
	\newblock {\em J. Algebraic Geom.}, 25(2):201--272, 2016.
	
	\bibitem[BPRS19]{burgos2019the}
	J.I. {Burgos Gil}, P.~Philippon, J.~{Rivera-Letelier}, and M.~Sombra.
	\newblock The distribution of galois orbits of points of small height in toric
	varieties.
	\newblock {\em Amer. J. Math.}, 141:309--381, 2019.
	
	\bibitem[BT76]{bedford1976dirchlet}
	E.~Bedford and B.~Taylor.
	\newblock The dirchlet problem for a complex monge-amp{\`e}re equation.
	\newblock {\em Invent. Math.}, 37:1, 1976.
	
	\bibitem[CD12]{chambert2012formes}
	A.~Chambert{--}Loir and A.~Ducros.
	\newblock Formes diff\'{e}rentielles r\'{e}elles et courants sur les espaces de
	berkovich.
	\newblock \href{https://arxiv.org/pdf/1204.6277}{\tt arXiv:1204.6277}, 2012.
	
	\bibitem[CG24]{cai2024abstract}
	Y.~Cai and W.~Gubler.
	\newblock Abstract divisorial spaces and arithmetic intersection numbers.
	\newblock \href{https://arxiv.org/pdf/2409.00611}{\tt arXiv:2409.00611}, 2024.
	
	\bibitem[{Cha}06]{chambert2006mesures}
	A.~{Chambert--Loir}.
	\newblock Mesures et {\'e}quidistribution sur les espaces de berkovich.
	\newblock {\em J. Reine Angew. Math.}, 595:215--235, 2006.
	
	\bibitem[CM15]{chen2015distribution}
	H.~Chen and C.~Maclean.
	\newblock Distribution of logarithmic spectra of the equilibrium energy.
	\newblock {\em Manuscripta math.}, 146(3):365--394, 2015.
	
	\bibitem[CM20]{chen2020arakelov}
	H.~Chen and A.~Moriwaki.
	\newblock {\em Arakelov geometry over adelic curves}, volume 2258 of {\em
		Lecture Notes in Mathematics}.
	\newblock Springer, Singapore, [2020] \copyright 2020.
	
	\bibitem[CM21]{chen2021arithmetic}
	H.~Chen and A.~Moriwaki.
	\newblock Arithmetic intersection theory over adelic curves.
	\newblock \href{https://arxiv.org/pdf/2103.15646}{\tt arXiv:2103.15646}, 2021.
	
	\bibitem[CM22]{chen2022hilbert}
	H.~Chen and A.~Moriwaki.
	\newblock Hilbert-samuel formula and equidistribution theorem over adelic
	curves.
	\newblock
	\url{https://webusers.imj-prg.fr/~huayi.chen/Recherche/Hilbert_Samuel_Adelic_Curves.pdf},
	2022.
	
	\bibitem[GS90]{gillet1990arithmetic}
	H.~Gillet and C.~Soul\'{e}.
	\newblock Arithmetic intersection theory.
	\newblock {\em Inst. Hautes \'{E}tudes Sci. Publ. Math.}, 72:93--174 (1991),
	1990.
	
	\bibitem[Gub97]{MfieldGubler}
	W.~Gubler.
	\newblock Heights of subvarieties over {$M$}-fields.
	\newblock In {\em Arithmetic geometry ({C}ortona, 1994)}, volume XXXVII of {\em
		Sympos. Math.}, pages 190--227. Cambridge Univ. Press, Cambridge, 1997.
	
	\bibitem[Gub98]{gubler1998local}
	W.~Gubler.
	\newblock Local heights of subvarieties over non-archimedean fields.
	\newblock {\em J. Reine Angew. Math.}, 1998(498):61--113, 1998.
	
	\bibitem[Laz04]{lazarsfeld2017positivity}
	R.~Lazarsfeld.
	\newblock {\em Positivity in algebraic geometry. {I}}, volume~48 of {\em
		Ergebnisse der Mathematik und ihrer Grenzgebiete. 3. Folge. A Series of
		Modern Surveys in Mathematics [Results in Mathematics and Related Areas. 3rd
		Series. A Series of Modern Surveys in Mathematics]}.
	\newblock Springer-Verlag, Berlin, 2004.
	\newblock Classical setting: line bundles and linear series.
	
	\bibitem[LM09]{lazarsfeld2009convex}
	R.~Lazarsfeld and M.~Musta\c{t}\u{a}.
	\newblock Convex bodies associated to linear series.
	\newblock {\em Ann. Sci. \'{E}c. Norm. Sup\'{e}r. (4)}, 42(5):783--835, 2009.
	
	\bibitem[Mor08]{moriwaki2008continuity}
	A.~Moriwaki.
	\newblock Continuity of volumes on arithmetic varieties.
	\newblock {\em J. Algebraic Geom.}, 18(3):407--457, 2008.
	
	\bibitem[Mor14]{moriwaki2014arakelov}
	A.~Moriwaki.
	\newblock {\em Arakelov geometry}, volume 244 of {\em Translations of
		Mathematical Monographs}.
	\newblock American Mathematical Society, Providence, RI, 2014.
	\newblock Translated from the 2008 Japanese original.
	
	\bibitem[Nys14]{nystrom2014transforming}
	D.~W. Nystr{\"o}m.
	\newblock Transforming metrics on a line bundle to the okounkov body.
	\newblock {\em Ann. Sci. {\'E}c. Norm. Sup{\'e}r. (4)}, 47(6):1111--1161, 2014.
	
	\bibitem[Oko96]{okounkov1996brunn}
	A.~Okounkov.
	\newblock Brunn--minkowski inequality for multiplicities.
	\newblock {\em Invent. math.}, 125:405--411, 1996.
	
	\bibitem[ROC70]{rockafellar1970convex}
	R.~TYRRELL ROCKAFELLAR.
	\newblock {\em Convex Analysis}.
	\newblock Princeton Mathematical Series. Princeton University Press, 1970.
	
	\bibitem[S{\'e}d23]{sedillot2023differentiability}
	A.~S{\'e}dillot.
	\newblock Differentiability of the $\chi$-volume function over an adelic curve.
	\newblock \href{https://arxiv.org/pdf/2303.03377}{\tt arXiv:2303.03377}, 2023.
	
	\bibitem[SW65]{shephard1965metrics}
	G.C. Shephard and R.J. Webster.
	\newblock Metrics for sets of convex bodies.
	\newblock {\em Mathematika}, 12(1):73--88, 1965.
	
	\bibitem[Yua08]{yuan2008big}
	X.~Yuan.
	\newblock Big line bundles over arithmetic varieties.
	\newblock {\em Invent. math.}, 173(3):603--649, 2008.
	
	\bibitem[Yua09]{yuan2009volumes}
	X.~Yuan.
	\newblock On volumes of arithmetic line bundles.
	\newblock {\em Compos. Math.}, 145(6):1447--1464, 2009.
	
	\bibitem[Yua21]{yuan2021arithmetic}
	X.~Yuan.
	\newblock Arithmetic bigness and a uniform bogomolov-type result.
	\newblock \href{https://arxiv.org/pdf/2108.05625}{\tt arXiv:2108.05625}, 2021.
	
	\bibitem[YZ21]{yuan2021adelic}
	X.~Yuan and S.~Zhang.
	\newblock Adelic line bundles over quasi-projective varieties.
	\newblock \href{https://arxiv.org/pdf/2105.13587}{\tt arXiv:2105.13587}, 2021.
	
	\bibitem[Zha95]{zhang1995small}
	S.~Zhang.
	\newblock Small points and adelic metrics.
	\newblock {\em J. Algebraic Geom.}, 4(2):281--300, 1995.
	
\end{thebibliography}
\end{document}